\documentclass[10pt]{amsart}
\usepackage{graphicx}
\usepackage{xcolor}
\usepackage{amssymb,amsfonts,amsthm,amsmath,calligra, enumerate, tikz}
\usepackage[utf8]{inputenc}
\usepackage{accents}
\usepackage{slashed}
\usepackage{yfonts}
\usepackage{mathrsfs,pifont}
\theoremstyle{definition}
\usepackage{tikz}
\usepackage[all]{xy}

\usepackage[bookmarksnumbered=true, bookmarksopen=true]{hyperref}

\def\be{\begin{eqnarray}}
\def\ee{\end{eqnarray}}
\def\ben{\begin{eqnarray*}}
\def\een{\end{eqnarray*}}

\usepackage{geometry}
\allowdisplaybreaks

\def\ZZ{{\mathbb{Z}}}
\def\RR{{\mathbb{R}}}
\def\PP{{\mathbb{P}}}
\def\QQ{{\mathbb{Q}}}
\def\CC{{\mathbb{C}}}

\def\TT{\mathbb{T}}

\newcommand{\bp}{\mathbf{p}}
\newcommand{\bq}{\mathbf{q}}
\newcommand{\bt}{\mathbf{t}}
\newcommand{\bx}{\mathbf{x}}

\newcommand{\cE}{\mathscr{E}}

\newcommand{\Or}{\mathsf{O}}

\newcommand{\bone}{\mathbf{1}}

\newcommand{\cA}{\mathcal{A}}
\newcommand{\cB}{\mathcal{B}}
\newcommand{\cC}{\mathcal{C}}
\newcommand{\cD}{\mathcal{D}}
\newcommand{\cF}{\mathcal{F}}
\newcommand{\cG}{\mathcal{G}}
\newcommand{\cH}{\mathcal{H}}
\newcommand{\cI}{\mathcal{I}}
\newcommand{\cK}{\mathcal{K}}
\newcommand{\cL}{\mathcal{L}}
\newcommand{\cM}{\mathcal{M}}
\newcommand{\cN}{\mathcal{N}}
\newcommand{\cO}{\mathcal{O}}

\newcommand{\cT}{\mathcal{T}}

\newcommand{\bStab}{\mathbf{Stab}}

\newcommand{\Attr}{\operatorname{Attr}}
\newcommand{\Cont}{\operatorname{Cont}}
\newcommand{\Eff}{\operatorname{Eff}}
\newcommand{\Ell}{\operatorname{Ell}}
\newcommand{\End}{\operatorname{End}}
\newcommand{\ev}{\operatorname{ev}}
\newcommand{\Hom}{\operatorname{Hom}}
\newcommand{\Id}{\operatorname{Id}}

\newcommand{\Lie}{\operatorname{Lie}}
\newcommand{\Pic}{\operatorname{Pic}}
\newcommand{\pr}{\operatorname{pr}}
\newcommand{\pt}{\operatorname{pt}}
\newcommand{\QM}{\operatorname{QM}}
\newcommand{\Res}{\operatorname{Res}}
\newcommand{\rk}{\operatorname{rk}}
\newcommand{\Span}{\operatorname{Span}}
\newcommand{\Spec}{\operatorname{Spec}}
\newcommand{\Stab}{\operatorname{Stab}}
\newcommand{\Supp}{\operatorname{Supp}}

\newcommand{\com}{\mathrm{com}}
\newcommand{\loc}{\mathrm{loc}}
\newcommand{\ns}{\mathrm{ns}}

\newcommand{\rel}{\mathrm{rel}}
\newcommand{\vir}{\mathrm{vir}}

\newcommand{\fK}{\mathfrak{K}}
\newcommand{\fP}{\mathfrak{P}}

\newcommand{\fX}{\mathfrak{X}}

\newcommand{\bA}{\mathsf{A}}

\newcommand{\bO}{\mathsf{O}}
\newcommand{\Pol}{\mathsf{Pol}}
\newcommand{\bS}{\mathsf{S}}
\newcommand{\bT}{\mathsf{T}}
\newcommand{\bK}{\mathsf{K}}

\theoremstyle{definition}
\newtheorem{Definition}{Definition}[section]

\newtheorem{Remark}[Definition]{Remark}

\numberwithin{equation}{section}

\theoremstyle{Theorem}
\newtheorem{Theorem}[Definition]{Theorem}
\newtheorem{Proposition}[Definition]{Proposition}
\newtheorem{Lemma}[Definition]{Lemma}
\newtheorem{Corollary}[Definition]{Corollary}

\newcommand{\fC}{\mathfrak{C}} 

\begin{document}
	
\title{3d Mirror Symmetry and Quantum $K$-theory of Hypertoric Varieties}
\author{Andrey Smirnov and Zijun Zhou}
\date{}
\maketitle
\thispagestyle{empty}

\setlength{\parskip}{1ex}
	
\begin{abstract}
	
Following the idea of Aganagic--Okounkov \cite{AOelliptic}, we study vertex functions for hypertoric varieties, defined by $K$-theoretic counting of quasimaps from $\PP^1$. We prove the 3d mirror symmetry statement that the two sets of $q$-difference equations of a 3d hypertoric mirror pairs are equivalent to each other, with K\"ahler and equivariant parameters exchanged, and the opposite choice of polarization. Vertex functions of a 3d mirror pair, as solutions to the $q$-difference equations, satisfying particular asymptotic conditions, are related by the elliptic stable envelopes. Various notions of quantum $K$-theory for hypertoric varieties are also discussed.

\end{abstract}
	
\setcounter{tocdepth}{2}

\tableofcontents

\section{Introduction}

\subsection{Motivation and background in physics}

Let $G_\RR$ be a compact Lie group, and $\mathbf{M}$ be a quaternionic representation of $G$. The pair $(G_\RR, \mathbf{M})$ defines a 3d $\cN = 4$ supersymmetric gauge theory in physics, where $G_\RR$ is the gauge group, and the representation $\mathbf{M}$ describes the collection of matter fields. There are two interesting components of the moduli space of vacua associated with such theories, called the \emph{Higgs branch} and \emph{Coulomb branch} respectively, which recently received plenty of attention in mathematics. Given a 3d $\cN = 4$ theory $\cT$, its Higgs branch $\cM_H(\cT)$ is mathematically the hyper-K\"ahler quotient associated with the $G_\RR$-representation $\mathbf{M}$, while its Coulomb branch $\cM_C (\cT)$, also admits a mathematical construction recently by Bravermann--Finkelberg--Nakajima \cite{Nak1, BFN2}, in the case where $\mathbf{M} = \mathbf{N} \oplus \mathbf{N}^*$ is of cotangent type.

3d mirror symmetry \cite{PhysMir3, PhysMir2, PhysMir1, Ga-Wit, HW, BDGH} predicts a duality phenomenon between certain pairs of 3d $\cN = 4$ supersymmetric gauge theories, $\cT$ and $\cT'$, which are called mirror pairs. Given explicitly in terms of lagrangian descriptions, these two theories are expected to be different presentations of the same physical theory, and hence admit the same, or equivalent correlation functions. Moreover, a particular property of the duality is that, the Higgs and Coulomb branches of a mirror pair is expected to be exchanged:
$$
\cM_H (\cT) = \cM_C (\cT'), \qquad \cM_C (\cT) = \cM_H (\cT'),
$$
as well as the FI parameters and mass parameters, which are often added to deform the theories and result in resolutions of the branches.

There are several aspects of the 3d mirror symmetry, from which one can extract interesting mathematical conjectures.

\begin{enumerate}[$\bullet$]
	
\setlength{\parskip}{1ex}

\item As in \cite{BDGH}, one can introduce boundary conditions in the Omega background, which implies an equivalence of categories of modules over the quantized Higgs and Coulomb branches. Mathematically, this is realized as Kozsul duality for the category $\cO$'s, or symplectic duality, for symplectic resolutions on the categorical level \cite{BLPW, matdu}.

\item One can consider geometric interpretations of the identification $\cM_H (\cT) = \cM_C (\cT')$ between Higgs and Coulomb branch of a mirror pair. For example, the Hikita conjecture \cite{Hik} and related works \cite{KMP}.

\item There should be some interplay between the Coulomb branch $\cM_C(\cT)$ and Higgs branch $\cM_H (\cT)$ of the same theory. In particular, the (quantized) Coulomb branch could be related to quasimap counting on the Higgs branch. There are several attempts in this direction, for example \cite{BDGHK} and recent papers \cite{MSY, HKW}.

\item Correlation functions or partition functions of a mirror pair should be equated or related, for example \cite{BFK, CDZ}. The approach of this paper is also of this kind, where we follow the idea of Aganagic--Okounkov \cite{Oko, AOelliptic, AOBethe}. Physically, the invariants we consider are the \emph{vortex partition functions} with domain $S^1 \times_q D^2$, or the 3d holomorphic blocks, in the sense of \cite{BDP}. Mathematically, these are the generating functions of quasimaps to the Higgs branch, or quantum $K$-theory, as we will explain in the paper, which also implies symmetries among geometric objects called elliptic stable envelopes. Attempts in this direction are such as \cite{GaKor, Kor, RSVZ, RSVZ2}. We also notice the recent work \cite{Hik}, related to hypertoric elliptic stable envelopes.

\end{enumerate}

\subsection{Vertex functions, $q$-difference equations, and elliptic stable envelopes}

From now on, in the algebraic-geometric language, we consider a complex reductive group $G$ (which is considered as the complexification of $G_\RR$), a $G$-representation $\mathbf{N}$. The quaternionic representation, with a fixed chosen complex structure, is considered as  $\mathbf{M} = \mathbf{N} \oplus \mathbf{N}^*$. The hyperk\"ahler quotient is then equivalent to the holomorphic symplectic reduction. More precisely, the Higgs branch of the associated 3d $\cN = 4$ theory is the GIT quotient $X := \mu^{-1} (0) /\!/_\theta G$, where $\mu: \mathbf{M} \to \mathfrak{g}^*$ is the complex moment map, and $\theta$ is a chosen character of $G$, serving as the stability condition.

A quasimap from $\PP^1$ to the holomorphic symplectic quotient $X$ is defined to be a morphism from $\PP^1$ to the stacky quotient $[\mu^{-1} (0) / G]$, which generically maps into the stable locus $X$. In \cite{Oko}, A. Okounkov introduced the \emph{vertex function} $V(q, z, a)$, defined as generating functions for the $K$-theoretic equivariant counting of those quasimaps, which satisfy an extra requirement that the point $\infty \in \PP^1$ (or the $\infty$ of the last bubble component, in the relative version) is not a base point. Here, $q$ is a fixed complex number such that $|q|<1$; $z$ and $a$ stand respectively for the collections of K\"ahler parameters (which records the degrees of the quasimaps) and equivariant parameters.

In the case that $X$ is a Nakajima quiver variety, the vertex function is shown to satisfy two sets of $q$-difference equations: either by $q$-shifts of $z$-variables, or by $q$-shifts of $a$-variables. The analytic property of such $q$-difference equations as studied in \cite{Oko, OS} shows that the associated $q$-difference modules are holonomic and admit regular singularities with respect to the variables $z$ and $a$ \emph{separately}, but not simultaneously. The vertex function (scaled by an appropriate prefactor), by definition, happens to generate the $z$-solutions, i.e., (multi-valued) solutions that are holomorphic in $z$-variables in a punctured neighborhood of the limit point $z\to 0$, but have infinitely many poles in any punctured neighborhood of the chosen limit point $a\to 0$. It is then natural to look for the monodromy transformation that relates the two kind of solutions: $z$-solutions and $a$-solutions, which is done by Aganagic--Okounkov \cite{AOelliptic}. The monodromy matrix is found to be the \emph{elliptic stable envelopes}, an elliptic analogue of the cohomological and $K$-theoretic stable envelopes \cite{MO}.

Motivated by 3d mirror symmetry, it is naturally conjectured by Aganagic--Okounkov \cite{AOelliptic} that in cases where the 3d mirror $X'$ of $X$ exists, the $a$-solutions are indeed the vertex functions defined for $X'$, and the two sets of $q$-difference equations for $X'$ are the same as those for $X$, with K\"ahler parameters $z$ and equivariant parameters $a$ exchanged with each other. As a corollary, one can also deduce a conjecture that the elliptic stable envelopes for $X$ and $X'$, properly renormalized, are transpose to each other. The 3d mirror symmetry for elliptic stable envelopes is proved for $T^* Gr (n, k)$, $n\geq 2k$ in \cite{RSVZ}, and cotangent bundle of a complete flag variety in \cite{RSVZ2}. In this paper, we prove the 3d mirror symmetry for both the vertex functions and elliptic stable envelopes, in the hypertoric case.

\subsection{Hypertoric 3d mirror symmetry}

In the special case of \emph{abelian} gauge theories, in other words, the gauge group is a torus, the mirror theory always exists and admits very explict descriptions \cite{PhysMir1, KS}. Mathematically, their Higgs branches are abelian hyper-K\"ahler quotients of a quaternionic representation, which are called \emph{toric hyperk\"ahler varieties} or \emph{hypertoric varieties}. To define a hypertoric variety, one starts with a short exact sequence
$$
\xymatrix{
	0 \ar[r] & \ZZ^k \ar[r]^\iota & \ZZ^n \ar[r]^\beta & \ZZ^d \ar[r] & 0.
}
$$
The map $\iota$ then describes a action of the torus $\bK := (\CC^*)^k$ on $\mathbf{N} = \CC^n$, and then on the symplectic vector space $T^*\CC^n$. The hypertoric variety is then defined as
$$
X := \mu^{-1} (0) /\!/_\theta \bK ,
$$
where $\mu: T^*\CC^n \to (\CC^k)^\vee$ is the complex moment map. $\theta$ here is a character of $(\CC^*)^k$, which we identify with an element in $(\ZZ^k)^\vee$. We will always assume that $\theta$ is chosen generically, in which case $X$ is smooth. We will consider the group action on $X$ by $\TT := \bT \times \CC^*_\hbar$, where $\bT := (\CC^*)^n$ acts on $\CC^n$, which descends to $X$, and $\CC^*_\hbar$ scales the symplectic form, whose equivariant parameter we denote by $\hbar$.

We denote by $a_1, \cdots, a_n$ the equivariant parameters of $\bT$, considered as coordinates on the torus $\bA := \bT / \bK = (\CC^*)^d$. We also take $z_1, \cdots, z_n$ as \emph{K\"ahler parameters}, treated as coordinates on $\bK^\vee$. Both collections of parameters are redundant and subject to certain relations (see Section \ref{sec-para}).

Denote by $p_1 = 0 \in \PP^1$ and $p_2:= \infty \in \PP^1$. Let $\QM (X, \beta)$ be the moduli space of quasimaps from $\PP^1$ to $X$, with degree $\beta \in H_2 (X, \ZZ)$. It is equipped with the natural perfect obstruction theory, and hence the associated virtual structure sheaf $\cO_\vir \in K_\TT (\QM (X, \beta))$. Let $\QM(X, \beta)_{\ns \ p_2}$ be the open substack consisting of those quasimaps where $p_2$ is \emph{nonsingular}, i.e., not a base point. The (bare) \emph{vertex function} with descendent insertion $\tau$, is defined as
$$
V^{(\tau)}(q,z, a) := \sum_{\beta \in \Eff(X)} z^\beta \ev_{2, *} \left( \QM(X, \beta)_{\ns\ p_2}, \widehat\cO_\vir \cdot \ev_1^* \tau  \right) \quad \in K_{\TT\times \CC^*_q} (X)_\loc [[ z^{\Eff(X)} ]] ,
$$
where $\tau\in K_\TT (\fX)$ is a Kirwan lift of $K$-theory class to the stacky quotient $\fX := [\mu^{-1} (0) / \bK]$, $\ev_1$ and $\ev_2$ are evaluation maps at $p_1$ or $p_2$ from the moduli stack to $\fX$ or $X$, depending on whether the point is assumed to be nonsingular, and $\widehat\cO_\vir$ is the twist of $\cO_\vir$ by a square root of the virtual canonical bundle and a chosen polarization $T^{1/2}_X$, i.e., a ``half" of the tangent bundle $T_X$. Moreover, $q$ is the character of $T_{p_1} \PP^1$ under the action of the torus $\CC^*_q$ on $\PP^1$. The invariants have to lie in $K_{\TT \times \CC_q^*} (X)_\loc$, which means to apply $\CC_q^*$-localization and pass to the fraction field $\CC(q)$,  because the map $\ev_2$ is not proper, and its push-forward has to be defined via such localization.

The geometry of hypertoric varieties can be described very nicely in combinatorics, using the language of hyperplane arrangements, which makes it convenient to apply $\TT \times \CC_q^*$-localization computations. , The vertex functions can then be calculated explicitly, and written in the form of a contour integral of Barnes--Mellin type. One can then realize them, appropriately renormalized by some prefactors and denoted by $\widetilde V$, as solutions of $q$-difference systems. Let $Z_i$ (resp $A_i$) be the operators that shifts $z_i \mapsto q z_i$ (resp. $a_i \mapsto q a_i$) and keeps other variables unchanged.

\begin{Theorem} [Theorem \ref{q-diff-eqn}]
	\begin{enumerate}[1)]
		
		\item The modified vertex function $\widetilde V^{(\bone)} (q,z, a) \big|_\bp$ is annihilated by the following $q$-difference operators:
		$$ 
		\prod_{i\in S^+} ( 1 - Z_i ) \prod_{i\in S^-} ( 1 - \hbar  Z_i )  -  z_\sharp^\beta \prod_{i\in S^+} ( 1 - \hbar  Z_i ) \prod_{i\in S^-} ( 1 - Z_i ) , \qquad S = S^+ \sqcup S^- : \text{circuit} \footnote{A circuit is a minimal subset $S\subset \{1, \cdots, n\}$, such that $\{\beta (e_i) \mid i\in S\}$ are linearly dependent in $\CC^d$, where $e_i$'s are the standard basis for $\CC^n$. A circuit admits a unique decomposition $S = S^+ \sqcup S^-$ such that $\sum_{i\in S^+} e_i - \sum_{i\in S^-} e_i \in \ker \beta$, determined by the stability condition $\theta$. Circuits correspond to indecomposible effective curves of $X$. For details, see Section \ref{sec-hyper-arr}.},
		$$ 
		where $z_{\sharp, i}:= z_i (-\hbar^{-1/2})$, $\beta$ is the curve class corresponding to $S$, and $z_\sharp^\beta := \prod_{i\in S^+} z_{\sharp, i} \prod_{i\in S^-} z_{\sharp, i}^{-1}$.
		
		\item The modified vertex function $\widetilde V^{(\bone)} (q,z, a) \big|_\bp \cdot e^{-\sum_{i=1}^n \frac{\ln z_{\sharp, i} \ln a_i}{\ln q} }$ is annihilated by the following $q$-difference operators:
		$$ 
		\prod_{i\in R^+} (1- A_i) \prod_{i\in R^-} (1  - q\hbar^{-1} A_i ) - (\hbar a)^\alpha \prod_{i\in R^+} (1  - q\hbar^{-1} A_i ) \prod_{i\in R^-} (1  -  A_i ) , \qquad R = R^+ \sqcup R^- : \text{cocircuit}\footnote{A cocircuit is a minimal subset $R\subset \{1, \cdots, n\}$, such that $\{\iota^\vee (e_i^*) \mid i\in R \}$ are linearly independent in $\CC^k$, where $e_i^*$'s are the dual standard basis for $\CC^n$. A cocircuit admits a unique decomposition $R = R^+ \sqcup R^-$ such that $\sum_{i\in R^+} e_i^* - \sum_{i\in R^-} e_i^* \in \ker \iota^\vee$, determined by the cocharacter $\sigma$. Cocircuits correspond to simple roots of the torus action on $X$. For details, see Section \ref{section-root}.},
		$$ 
		where $\alpha$ is the root corresponding to $R$, and $(\hbar a)^\alpha := \prod_{i\in R^+} (\hbar a_i) \prod_{i\in R^-} (\hbar a_i)^{-1}$.
	\end{enumerate}
	
\end{Theorem}

Here for the equations 2), as in \cite{MO, Oko}, one also needs to choose a cocharacter $\sigma: \CC^* \to \bA$, which we identify with an element in $\ZZ^d$. The choice of $\sigma$ determines a chamber in the space of equivariant parameters, as well as an ordering of the fixed point set $X^\TT$. 

The 3d mirror of a hypertoric variety $X$ is still a hypertoric variety $X'$, constructed by dualizing the defining short exact sequence
$$
\xymatrix{
	0 \ar[r] & (\ZZ^d)^\vee \ar[r]^{\beta^\vee} & (\ZZ^n)^\vee \ar[r]^{\iota^\vee} & (\ZZ^k)^\vee \ar[r] & 0.
}
$$
The stability condition and cocharacter of $X'$ is chosen as $\theta' = -\sigma$, $\sigma' = -\theta$. There are natural identifications of the spaces of parameters $\bK' = \bA^\vee$, $\bA' = \bK^\vee$, and also bijections between the fixed point sets $X^\TT$ and $(X')^{\TT'}$. We consider the vertex function $V' (q, z', a')$ for $X'$, but defined for an \emph{opposite} choice of polarization $T^{1/2}_{X'}$. We finally have the following main result.

\begin{Theorem} [Theorem \ref{main-theorem}]
	Under the identification of parameters
	$$
	\kappa_{\mathrm{vtx}}:  \qquad \bK^\vee \times \bA \times \CC^*_\hbar \xrightarrow{\sim} \bA' \times (\bK')^\vee \times \CC^*_\hbar, \qquad ( z_{\sharp, i} , a_i, \hbar) \mapsto ( (a'_i)^{-1}, z'_{\sharp, i},  q \hbar^{-1}),
	$$	
	the product
	$$
	V'(q, z', a') = \fP \cdot V (q, z, a)  \quad \in \quad  K_\TT (X^\TT)
	$$
	forms a global class in $K_{\TT'} (X')$, and coincides with the vertex function $V'(q, z', a')$ of the 3d-mirror $X'$, with the opposite polarization $T^{1/2}_{X'}$.
\end{Theorem}

Here the matrix $\fP$ is defined via the elliptic stable envelope matrix $ \dfrac{\Stab_\sigma (\bq) |_\bp}{\Theta (T^{1/2}_X |_\bp ) }$, $\bp, \bq \in X^\TT$, appropriately renormalized by other factors contributed from the fixed point $\bp$. For any fixed point $\bq \in X^\TT$, the elliptic stable envelope $\Stab_\sigma (\bq)$ is defined as a particular section of a line bundle over the equivairant elliptic cohomology scheme $\Ell_\TT (X) \times \Ell_{\bT^\vee} (\pt)$, and admits an explicit expression as a monomial of theta functions. They also satisfy a 3d mirror symmetry correspondence as follows.

\begin{Theorem} [Theorem \ref{Thm-Stab}]
	
	Under the isomorphism of parameters
	$$
	\kappa_{\Stab} :  \ \bK^\vee \times \bA \times \CC^*_\hbar \xrightarrow{\sim} \bA' \times (\bK')^\vee \times \CC^*_\hbar, \qquad ( z_i , a_i, \hbar) \mapsto ( a'_i, z'_i,  \hbar^{-1}),
	$$
	we have:
	\begin{enumerate}[1)]
		
		\item There is a line bundle $\mathfrak{M}$ on $\Ell_{\bT \times \bT' \times \CC_\hbar^*} (X \times X')$ such that
		$$
		(i_{\bp'}^*)^* \mathfrak{M} = \mathfrak{M} (\bp) , \qquad (i_\bp^*)^* \mathfrak{M} = \mathfrak{M} (\bp').
		$$
		
		\item There is a section $\mathfrak{m}$ of $\mathfrak{M}$, called the ``duality interface", such that
		$$
		(i_{\bp'}^*)^* \mathfrak{m} = \bStab_{\sigma} (\bp) , \qquad (i_\bp^*)^* \mathfrak{m} = \bStab'_{\sigma'} (p_{\bp'}) .
		$$
		
		\item In the hypertoric case, the duality interface $\mathfrak{m}$ admits a simple explicit form:
		$$
		\mathfrak{m} = \prod_{i=1}^n \vartheta (x_i x'_i)  .
		$$
		In particular, it comes from a section of a universal line bundle on the prequotient $\Ell_{\bT \times \bT' \times \CC_\hbar^* \times \bT^\vee \times (\bT')^\vee} (\pt)$, and does not depend on the choices of $\theta$ or $\sigma$.
	\end{enumerate}
\end{Theorem}

\begin{Corollary} [Corollary \ref{Cor-Stab}]
	We have the following symmetry between elliptic stable envelopes:
	$$
	\frac{\Stab_\sigma (\bp) |_\bq }{\Stab_\sigma (\bq) |_\bq } = \frac{\Stab'_{\sigma'} (\bq') |_{\bp'}}{\Stab'_{\sigma'} (\bp') |_{\bp'}},
	$$
	where $\bp, \bq \in X^\TT$, and $\bp', \bq' \in (X')^{\TT'}$ are fixed points corresponding to each other.
\end{Corollary}

\subsection{Quantum $K$-theory}

From the enumerative geometric point of view, quasimaps play a crucial role in Gromov--Witten type theories. Using the notion of \emph{relative quasimaps} introduced in \cite{Oko}, where one allow the domain to ``bubble" at $\infty$, in \cite{PSZ, KPSZ} a version of quantum $K$-theory (which we call PSZ quantum $K$-theory to avoid confusion) is defined, as a deformation of the usual ring structure $K_\TT(X)$. The $q$-difference system above satisfied by the vertex functions, with respect to the $z$-variables, actually determines the PSZ quantum $K$-theory ring structure.

\begin{Theorem} [Theorem \ref{PSZ-relations}]
	We have the following presentations of ring structures (which are equivalent to each other):
	
	\begin{enumerate}[1)]
		
		\item The PSZ quantum $K$-theory ring of $X$ is generated by the quantum tautological line bundles $\widehat L_i (z)$, $1\leq i\leq n$, up to the relations
		$$
		\prod_{i\in S^+} ( 1 - \widehat L_i (z) ) * \prod_{i\in S^-} ( 1 - \hbar  \widehat L_i (z) )  -  z_\sharp^\beta \prod_{i\in S^+} ( 1 - \hbar  \widehat L_i (z) ) * \prod_{i\in S^-} ( 1 - \widehat L_i (z) ) , \qquad S = S^+ \sqcup S^- : \text{circuit},
		$$
		where $z_{\sharp, i}:= z_i (-\hbar^{-1/2})$, $\beta$ is the curve class corresponding to $S$, $z_\sharp^\beta := \prod_{i\in S^+} z_{\sharp, i} \prod_{i\in S^-} z_{\sharp, i}^{-1}$, and all products $\prod$ are quantum products $*$.
		
		\item The divisorial quantum $K$-theory ring of $X$ is generated by the line bundles $L_i$, $1\leq i\leq n$, up to the relations
		$$
		\prod_{i\in S^+} ( 1 -  L_i )  \prod_{i\in S^-} ( 1 - \hbar L_i )  -  z_\sharp^\beta \prod_{i\in S^+} ( 1 - \hbar L_i )  \prod_{i\in S^-} ( 1 -  L_i ) , \qquad S = S^+ \sqcup S^- : \text{circuit},
		$$
		where $z_{\sharp, i}:= z_i (-\hbar^{-1/2})$, $\beta$ is the curve class corresponding to $S$, and $z_\sharp^\beta := \prod_{i\in S^+} z_{\sharp, i} \prod_{i\in S^-} z_{\sharp, i}^{-1}$.
		
	\end{enumerate}
	
\end{Theorem}

The quasimap we considered here is actually a special case of the $\epsilon$-stable quasimaps for $\epsilon = 0+$ with a parametrized domain component defined by \cite{CKM, CK-wall}. It is then natural to ask how it is related to the quantum $K$-theory defined via counting stable maps by Lee and Givental \cite{Lee, GL}. In general, it is expected to be studied via certain $\epsilon$-wall-crossing techniques for stability conditions over the moduli stacks. But for us, thanks to the explicit computations available, we are able to extract information directly from the vertex functions.

\begin{Corollary} [Coroallary \ref{V-tau}]
	Let $\tau$ be a Laurent polynomial in $q$ with coefficients in $K_\bT(X) \otimes \Lambda$. The descendent bare vertex function
	$$
	(1-q) V^{(\tau)} (q^{-1} , z) \big|_{z_\sharp = Q}
	$$
	lies in the range of permutation-equivariant big $J$-function of $X$.
\end{Corollary}

Here one has to consider Givental's \emph{permutation-equivariant} quantum $K$-theory \cite{Giv1}, which behaves better with localizations and twisted theories \cite{GV}. We then introduce a Givental type quantum $K$-theory, based on the the $K$-theoretic Gromov--Witten potential with both permutation-equivariant and ordinary inputs. Following the idea in \cite{IMT}, there are some operators $B_{i, \com} \in \End K_\TT(X) \otimes \Lambda$, introduced via $q$-difference operators. Let $t_0 \in K_\TT (X) \otimes \Lambda$ be the point where $(1-q) V^{(\bone)} (q^{-1} , z) \big|_{z_\sharp = Q} = J(t_0, Q)$.

\begin{Theorem} [Theorem \ref{GivQK-relation}]
	
	\setlength{\parskip}{1ex}
	
	Let $X$ be a hypertoric variety, and $t_0 \in K(X) \otimes \Lambda$ be as in Theorem \ref{V}. We fix the insertions $x = 0$ and $t = t_0$.
	
	\begin{enumerate}[1)]
		
		\item For any circuit $S = S^+\sqcup S^-$, and the corresponding curve class $\beta$, the identity class $\bone \in K_\bT (X)$ is annihilated by the following operator
		$$
		\prod_{i\in S^+} ( 1 - B_{i, \com}) \prod_{i\in S^-} (\hbar - B_{i, \com} ) - Q^\beta \prod_{i\in S^+} (\hbar - B_{i, \com} ) \prod_{i\in S^-} (1 - B_{i, \com}),
		$$
		where $Q^\beta := \prod_{i\in S^+} Q_i \prod_{i\in S^-} Q_i^{-1}$. 
		
		\item The Givental quantum $K$-theory ring of $X$ is generated by the classes $B_{i, \com} \bone$, $1\leq i\leq n$, up to the following relations: for any circuit $S = S^+\sqcup S^-$, and the corresponding curve class $\beta$
		$$
		\prod_{i\in S^+} ( 1 - B_{i, \com} \bone) \bullet \prod_{i\in S^-} (\hbar - B_{i, \com} \bone ) = Q^\beta \prod_{i\in S^+} (\hbar - B_{i, \com} \bone ) \bullet \prod_{i\in S^-} (1 - B_{i, \com} \bone ),
		$$
		where all the products are the quantum product $\bullet$.
		
	\end{enumerate}
	
\end{Theorem}

\subsection{Structure of the paper}

The paper is organized as follows. In Section 2, we review basic constructions and the equivariant geometry of hypertoric varieties. We explicitly described the $K$-theory $K_\TT (X)$, the tautological line bundles $L_i$ and the characters one obtains when restricting them to a fixed point $\bp$. In Section 2.7, we introduce the redundant and global equivariant and K\"ahler parameters. In Section 3, we recall the definition of elliptic stable envelopes and their characterization via theta functions in the hypertoric case. We define the vertex function and PSZ quantum $K$-theory in Section 4, and then compute them in Section 5 via $\TT$-localization. In Section 6, we prove our main theorems on the 3d mirror symmetry for vertex functions and elliptic stable envelopes. The relationship to Givental's permutation-equivariant quantum $K$-theory is studied in Section 7.

\subsection{Acknowledgements}

The authors would like to thank Mina Aganagic and Andrei Okounkov for their extraordinary work \cite{AOelliptic}, where we learn the idea and start this project. The second author would also like to thank Ming Zhang and Yaoxiong Wen for discussions on Givental's quantum $K$-theory. The work of A.S. is supported by NSF grant DMS - 2054527, the Russian Science Foundation under grant 18-01-00926 and the AMS travel grant. The work of Z.Z. is supported by FRG grant 1564500, and World Premier International Research Center Initiative (WPI), MEXT, Japan.

\vspace{3ex}

\section{Geometry of hypertoric varieties}

\subsection{Basic construction} \label{Sec-basic}

In this section we review the definition and geometric properties of hypertoric varieties. For the details, there are plenty of references, such as \cite{BD, HH, HP, HS, Kon, Kon2}. As hyperk\"ahler analogues of toric varieties, which can be constructed as symplectic reductions of complex vector spaces, hypertoric varieties are hyperk\"ahler reductions of quarternion vector spaces.

Let $k,n$ be nonnegative integers, with $k\leq n$. Consider a short exact sequence of free $\ZZ$-modules:
\begin{equation} \label{knd-seq}
\xymatrix{
	0 \ar[r] & \ZZ^k \ar[r]^\iota & \ZZ^n \ar[r]^\beta & \ZZ^d \ar[r] & 0.
}
\end{equation}
Denote the complex tori by $\bK := (\CC^*)^k$, $\bT := (\CC^*)^n$ and $\bA := (\CC^*)^d$. When tensoring with $\CC$, the embedding $\iota$ specifies an embedding of the complex torus $\bK$ into $\bT$, and hence defines an action of $\bK$ on the affine space $\CC^n$. The abelian groups $\ZZ^k$ and $\ZZ^n$ above are viewed as lattices of cocharacters of the tori.

The $\bK$-action naturally extends to the cotangent space $T^* \CC^n \cong \CC^n \oplus \CC^n$, preserving the canonical holomorphic symplectic form. Let $\mu: T^* \CC^n \to \Lie(\bK)^\vee$ be the moment map of this action.  With a choice of a stability parameter $\theta \in \Lie(\bK)^\vee$, the hypertoric variety $X$ is defined as the GIT quotient
$$
X := \mu^{-1} (0) /\!/_\theta \bK = \mu^{-1} (0)^{ss} / \bK,
$$
where $\mu^{-1}(0)^{ss} \subset \mu^{-1} (0)$ is the semistable locus, which is open. 
We will also consider the stacky quotient, denoted by $\mathfrak{X} := [\mu^{-1} (0) / \bK ]$, of which $X$ is an open substack. 

\begin{Remark}
Later we will see that the quasimap theory actually depend on the presentation of the GIT quotient, i.e. the sequence (\ref{knd-seq}), rather than the resulting quotient variety $X$ itself. More precisely, we will consider hypertoric data \emph{up to automorphisms of $\ZZ^k$ and $\ZZ^d$}. In other words, two sets of hypertoric data as (\ref{knd-seq}) will be considered equivalent, if they can be related by change of bases for $\ZZ^k$ and $\ZZ^d$. The corresponding K\"ahler and equivariant parameters will also be understood up to those changes of bases.
\end{Remark}

\subsection{Hyperplane arrangement and circuits} \label{sec-hyper-arr}

The information of a hypertoric variety can be conveniently organized in the combinatoric data of a collection $\cH$ of affine hyperplanes in $\RR^d$, called a \emph{hyperplane arrangement}.

In the sequence (\ref{knd-seq}), let $e_i\in \ZZ^n$, $1\leq i\leq n$ be the standard basis. Consider the dual exact sequence
$$
\xymatrix{
	0 \ar[r] & (\ZZ^d)^\vee \ar[r]^{\beta^\vee} & (\ZZ^n)^\vee \ar[r]^{\iota^\vee} & (\ZZ^k)^\vee \ar[r] & 0.
}
$$

Let $\widetilde\theta \in (\ZZ^n)^\vee$ be a lift of $\theta$ along $\iota^\vee$. Then the hyperplane arrangement $\cH = \{ H_i \mid 1\leq i \leq n\}$ is defined as the collection of the following (affine) hyperplanes
$$
H_i := \{x\in (\RR^d)^\vee \mid \langle x, \beta (e_i) \rangle = - \langle \widetilde\theta, e_i \rangle \}.
$$
The hypertoric variety $X$ can be recovered from such an $\cH$. A different choice of the lift $\widetilde\theta$ corresponds to translating the hyperplanes simultaneously by an element in $(\ZZ^d)^\vee$, which does not affect the associated $X$.

Each hyperplane $H_i$ divides $(\RR^d)^\vee$ into two half-spaces:
$$
H_i^+ := \{x\in (\RR^d)^\vee \mid \langle x, \beta (e_i) \rangle \geq - \langle \widetilde\theta, e_i \rangle \}, \qquad H_i^- := \{x\in (\RR^d)^\vee \mid \langle x, \beta (e_i) \rangle \leq - \langle \widetilde\theta, e_i \rangle \}.
$$

A hypertoric variety $X$ is smooth if and only if the hyperplane arrangement $\cH$ satisfies the following two conditions:

\begin{enumerate}[1)]
	
	\item simple, i.e., for any $0\leq m \leq n$, every $m$ hyperplanes in $\cH$ intersect, if nonempty, in codimension $m$;
	
	\item unimodular, i.e., any collection of $d$ linearly independent vectors in the conormals $\{\beta (e_1), \cdots, \beta (e_n) \}$ form a basis of $\ZZ^d$ over $\ZZ$.
	
\end{enumerate}
Unless otherwise specified, we will always assume that our hypertoric variety $X$ is \emph{smooth}, which is always the case when the stability condition $\theta$ is chosen generically.
In that case we have $\mu^{-1} (0)^{ss} = \mu^{-1}(0)^s$, i.e. the semistable locus coincides with the stable locus.

A subset $S \subset \{1, \cdots, n\}$ is called a \emph{circuit}, if it is a minimal subset such that $\bigcap_{i\in S} H_i = \emptyset$. In other words, it gives a minimal relations among the images of $e_i$'s for $i\in S$.

Each circuit has a unique splitting $S = S^+ \sqcup S^-$, determined as follows. The smoothness assumption on $X$ implies that in  the relation among $e_i$'s for $i\in S$, all coefficients of $e_i$'s must be $\pm 1$. The subsets $S^\pm$ are determined by the presentation of the relation
$$
\beta_S := \sum_{i\in S^+} e_i - \sum_{i\in S^-} e_i \in \ker \beta,
$$
such that $\langle \beta_S, \widetilde\theta \rangle \geq 0$.

\subsection{Equivariant geometry and line bundles}

Analogous to toric varieties, the hypertoric variety $X$ admits a torus action, naturally inherited from the standard torus action by $(\CC^*)^n$ on $\CC^n$. Moreover, there is a 1-dimensional torus $\CC^*_\hbar$ that scales the cotangent fiber of $T^* \CC^n$, which also descends to $X$. Hence $X$ admits a torus action by $\TT := \bT \times \CC^*_\hbar$, whose equivariant parameters are denoted by $a_1, \cdots, a_n, \hbar \in K_\TT (\pt)$.

Each character in $(\ZZ^n)^\vee$ defines a natural $\TT$-equivariant line bundle on $X$. In particular, for the standard dual basis $e_i^* \in (\ZZ^n)^\vee$, $1\leq i\leq n$, we have line bundles
$$
L_i := \mu^{-1}(0)^s \times_\bK \CC_{e_i^*},
$$
where $\mu^{-1}(0)^s \subset \mu^{-1}(0)$ is the stable locus for the $\bK$-action, and $\CC_{e_i^*}$ denotes the 1-dimensional $\bK$-representation defined by the character $\iota^\vee e_i^*$.

Similarly, each character in $(\ZZ^k)^\vee$ defines a (non-equivariant) line bundle on $X$. Let $\{f_j^* \mid 1\leq j\leq k\}$ be the standard dual basis for $(\ZZ^k)^\vee$. We have the tautological line bundles
$$
N_j := \mu^{-1}(0)^s \times_\bK \CC_{f_j^*}.
$$

The relationship between $L_i$ and $N_j$'s is encoded in the map $\iota^\vee$:
\begin{equation} \label{LL}
L_i = \CC_{e_i^*} \otimes \bigotimes_{j=1}^k N_j^{\otimes \iota_{ij}},
\end{equation}
where $(\iota_{ij})$ is the $n\times k$ matrix given by $\iota$.

We are interested in the $\TT$-equivariant $K$-theory of $X$. Recall that \cite{HH} $K_\TT (X)$ satisfies the \emph{Kirwan surjectivity} \footnote{The statement in \cite{HH} is only on the Kirwan surjectivity for cohomology. But the GKM method adopted there applies also to the $K$-theory, and allows one to obtain the explicit presentation (\ref{K(X)})}, i.e. the following surjection
$$
K_\TT (\pt) [s_1^{\pm 1}, \cdots, s_k^{\pm 1} ] \twoheadrightarrow K_\TT (X),
$$
where the image of $s_j$ is the $K$-theory class of $N_j$.

More precisely, the kernel of the surjection can be described explicitly. We have
\begin{equation} \label{K(X)}
K_\TT (X) \cong \CC [ a_1^{\pm 1}, \cdots, a_n^{\pm 1}, \hbar^{\pm 1}, s_1^{\pm 1}, \cdots, s_k^{\pm 1} ] / \langle \prod_{i\in S^+} (1-x_i) \prod_{i\in S^-} (1- \hbar x_i) \mid S \text{ is a circuit} \rangle   ,
\end{equation}
where $x_i$ is the class of $L_i$ in $K_\TT (X)$, which can be expressed in $s_j$'s through (\ref{LL}), i.e., $x_i = a_i \prod_{j=1}^k s_j^{\iota_{ij}}$.

In particular, if we view $\Spec K_\TT (X)$ as an affine scheme embedded in an algebraic torus $(\CC^*)^{n+k+1}$ with coordinates $a_i$, $\hbar$, $s_j$, it is given by the intersection of certain hypersurfaces, each defined by a circuit $S$ as a union of hyperplanes in the torus. Moreover, one can see that this intersection is indeed \emph{transversal}, reflecting the fact that $X$, equipped with the $\TT$-action, is a \emph{GKM variety} (i.e., admits finitely many fixed points and finitely many 1-dimensional orbits).

\begin{Remark}
	In later sections of the paper, we will sometimes choose particular bases of $\ZZ^k$ and $\ZZ^d$ to make computations more convenient. The corresponding constructions, such as $N_j$, $\theta$, will change according to the change of bases. However, we will always fix the basis of $\ZZ^n$. In other words, the line bundle $L_i$ and the stability parameter $\widetilde\theta$  will stay the same for all choices.
\end{Remark}

\subsection{Restriction to $\bT$-fixed points}

To conclude this section, we would like to specify the restriction of line bundles $N_j$ and $L_i$ to each fixed point in $X^\TT$. The $\TT$-invariant loci of $X$ can be described by the hyperplane arrangement $\cH$: a $\TT$-fixed point of $X$ corresponds to a vertex in $\cH$; a $\TT$-invariant 1-dimensional orbit corresponds to an edge, etc.

Let $\bp = \{\bp_1, \cdots, \bp_d \} \subset \{1, \cdots, n\}$, with $\bp_1 < \cdots < \bp_d$, be a subset with $d$ elements such that
$$
\bigcap_{i\in \bp } H_i \neq \emptyset.
$$
By our smooth assumption, the above intersection of $H_i$'s is a vertex in $\cH$, and therefore corresponds to a fixed point $\bp \in X^\TT$. From now on, we will abuse the notation $\bp$ (and also $\bq$) for the followings:

(i) the subset $\bp$,

(ii) the bijection from $\{1, \cdots, d\}$ to the set $\bp$,

(iii) the vertex in the hyperplane arrangement,

(iv) the fixed point $\bp \in X^\TT$.

Let $\cA_\bp := \{1, \cdots, n \} \backslash \bp$ be the complement subset. Denote that $\cA_\bp = \{ \cA_{\bp,1} , \cdots, \cA_{\bp, k} \}$, where $\cA_{\bp, 1}  < \cdots < \cA_{\bp, k}$. We will also abuse the notation $\cA_\bp$ for the bijection from $\{1, \cdots, k\}$ to $\cA_\bp$.

In particular, for (ii) above, we mean that if $i = \bp_I$ for some $1\leq I\leq d$, we denote by $I = \bp^{-1} (i)$; similarly if $j = \cA_{\bp, J}$ for some $1\leq J\leq k$, denote $J = (\cA_\bp)^{-1} (j)$.

The orientations given by the conormal vectors $\beta (e_i)$ of $H_i$'s, for $i\not\in \bp$, determines a splitting $\cA_\bp = \cA_\bp^+ \sqcup \cA_\bp^-$, where
$$
\cA_\bp^+ := \{i \not\in \bp \mid \bp \in H_i^+ \}, \qquad \cA_\bp^- := \{i \not\in \bp \mid \bp \in H_i^- \}.
$$

\begin{Lemma}
The following system of equations
$$
a_m s_1^{\iota_{m 1}} \cdots s_k^{\iota_{mk}} = \left\{ \begin{aligned}
& 1, \qquad && m \in \cA_\bp^+ \\
& \hbar^{-1}, \qquad && m \in \cA_\bp^-
\end{aligned}\right.
$$
admits a unique set of solutions $(s_1 (\bp), \cdots, s_k (\bp))$. The restriction of the line bundle $N_j$ to $\bp$ is
$$
\left. N_j \right|_{\bp} = s_j (\bp).
$$
The restriction of the line bundle $L_i$ to $\bp$ is
\begin{equation} \label{restriction}
\left.  L_i \right|_\bp = \left\{ \begin{aligned}
& 1 , \qquad && i \in \cA_\bp^+ \\
& \hbar^{-1}, \qquad && i \in \cA_\bp^- \\
& a_i s_1 (\bp)^{\iota_{i1}} \cdots s_k (\bp)^{\iota_{ik}}, \qquad && i \in \bp.
\end{aligned} \right.
\end{equation}
\end{Lemma}

\begin{proof}
This is essentially Theorem 3.5 in \cite{HH}.
\end{proof}

We choose the basis of $\ZZ^k$ such that the matrix $\iota$ is of the following special form:
\begin{equation} \label{V-frame-i}
\iota_{mj} (\bp) := \delta_{mj}, \qquad \iota_{ij} (\bp) := C_{ij} (\bp),  \qquad m, j \not\in \bp, \ i \in \bp
\end{equation}
where $C_{ij} (\bp)$ is a $d\times k$-matrix, with indices taken in $\bp \times \cA_\bp$. Moreover, one can choose a basis of $\ZZ^d$ such that the matrix $\beta$ is also of a special form:
\begin{equation} \label{V-frame-b}
\beta_{ij} (\bp) = -C_{ij} (\bp), \qquad \beta_{il} (\bp) = \delta_{il}, \qquad j \not\in \bp, \ i, l \in \bp.
\end{equation}
Here the $\bp$ in parentheses is to emphasize the dependence on $\bp$. We call this choice of bases the \emph{standard $\bp$-frame}.

\begin{Corollary}
For the standard $\bp$-frame, the relationship between $x_i = L_i$, $1\leq i\leq n$, and $s_J = N_J$, $1\leq J\leq k$, is
$$
x_i = \left\{ \begin{aligned}
& a_i s_I, \qquad && i = \cA_{\bp, I} \in \cA_\bp \\
& a_i s_1^{C_{i, \bp^c_1}} \cdots s_k^{C_{i, \bp^c_k}} , \qquad && i \in \bp.
\end{aligned} \right.
$$
The restriction of the line bundle $N_J$, $1\leq J\leq k$ to $\bp$ is
\begin{equation} \label{restriction-s(V)}
\left. N_J \right|_\bp = s_J (\bp) = \left\{ \begin{aligned}
& a_j^{-1}, \qquad && j = \cA_{\bp, J} \in \cA_\bp^+  \\
& \hbar^{-1} a_j^{-1} , \qquad && j = \cA_{\bp, J} \in \cA_\bp^-.
\end{aligned}\right.
\end{equation}
The restriction of the line bundle $L_i$, $1\leq i\leq n$ to $\bp$ is
\begin{equation} \label{restriction-V}
\left.  L_i \right|_\bp = \left\{ \begin{aligned}
& 1 , \qquad && i \in \cA_\bp^+ \\
& \hbar^{-1}, \qquad && i \in \cA_\bp^- \\
& a_i \prod_{j\not\in \bp} a_j^{-C_{ij} (\bp)} \cdot  \hbar^{- \sum_{j \in \bp^{c-}} C_{ij} (\bp) } .  \qquad && i \in \bp,
\end{aligned} \right.
\end{equation}
\end{Corollary}

We would like to rewrite this formula in a more intrinsic way. The presentation of $K_\TT (X)$ (\ref{K(X)}) can be expressed as
$$
\CC [ a_1^{\pm 1}, \cdots, a_n^{\pm 1}, \hbar^{\pm 1}, x_1^{\pm 1}, \cdots, x_n^{\pm 1} ] / \langle \prod_{i=1}^n (x_i / a_i)^{\beta_{ji}} - 1, 1\leq j\leq d;   \prod_{i\in S^+} (1-x_i) \prod_{i\in S^-} (1- \hbar x_i) , S: \text{circuits} \rangle.
$$
The picture of the affine scheme $\Spec K_\TT (X)$ is clear (view $\hbar$ as a constant): the first set of relations cuts out a codimension-$d$ subspace in the ambient torus $(\CC^*)^{2n}$, and the second furthermore cuts out a union of subspaces, each isomorphic to $(\CC^*)^n$, intersecting transversally. Each fixed point $\bp \in X^\TT$ corresponds to an irreducible component $\Spec K_\TT (\bp) \cong (\CC^*)^n$ of $\Spec K_\TT (X)$.

View $x_i$ as a function on $\Spec K_\TT (X)$. The restriction formula (\ref{restriction-V}) can then be written as the residue of the function $x_i$ along one of the components:
$$
x_i |_\bp :=  L_i |_{\bp} = \int_{\gamma(\bp)} x_i \cdot \frac{d\ln x_1 \wedge \cdots \wedge d\ln x_n}{\bigwedge_{m=1}^d \Big( \sum_{i=1}^n \beta_{mi} d\ln x_i \Big) }  ,
$$
where $\gamma(\bp)$ is the compact real $k$-cycle around the irreducible component $\Spec K_\TT (\bp)$, specified by $x_j = 1$, $j\in \cA_\bp^+$ and $x_j = \hbar^{-1}$, $j\in \cA_\bp^-$.

\subsection{Effective curves, walls, and chambers} \label{section-eff}

There is a bijection \cite{Kon} between circuits and primitive effective curves in $X$. We will abuse the notation and denote also by $\beta_S$ the primitive effective curve corresponding to the circuit $S$.

All irreducible $\TT$-invariant curves $C$ in $X$ are of the following form. Let $\bp$ and $\bq$ be two vertices in the hyperplane arrangement, such that $\bp = (\bq \backslash \{j\} ) \sqcup \{i\}$ and $\bq = (\bp \backslash \{i\} ) \sqcup \{j\}$, for some $1 \leq i\neq j\leq n$. There is a unique $\TT$-invariant curve $C$ connecting the fixed points $\bp$ and $\bq$. It is clear that the circuit that defines $C$ is
$$
S_{\bp \bq} := \bp \sqcup \{j\} = \bq \sqcup \{ i \}.
$$

\begin{Lemma} \label{bridge}
	\begin{enumerate}[(i)]
		
		\item $\deg L_m \big|_C = 0$, for $m \not\in S_{\bp \bq}$; $\deg L_i \big|_C = \pm 1$, if $i\in \cA_\bq^\pm$.
		
		\item The character
		$$
		T_\bp C = \left\{ \begin{aligned}
		&  x_i |_\bp  , && \quad i\in \cA_\bq^+ \\
		&  \hbar^{-1} x_i^{-1} |_\bp , && \quad i\in \cA_\bq^- .
		\end{aligned} \right.
		$$
		
	\end{enumerate}
\end{Lemma}

\begin{proof}
	
(i) is true by direct computation. (ii) follows from the fact that $x_i |_\bp = x_i |_\bq \cdot (T_\bp C)^{\deg L_i |_C}$, and $x_i |_\bq = 1$ or $\hbar^{-1}$, depending on $i\in \cA_\bq^\pm$.
\end{proof}

Recall that as in \cite{Kon}, the space $\RR^k$ of stability parameters of $X$ admits a wall-and-chamber structure. For each circuit $S = S^+ \sqcup S^-$, there is a codimension-1 hyperplane
$$
P_S := \Span_\RR \{ \iota^\vee e_i^* \mid i\not\in S \}\subset \RR^k,
$$
which we call a \emph{wall} in $\RR^k$. A connected component of the complement of walls
$$
\mathfrak{K} \subset \RR^k \backslash \bigcup_{S: \text{circuit}} P_S
$$
is called a \emph{chamber}.

Constructed by the real moment map in the hyperK\"ahler definition, there is a family of (possibly singular) hypertoric varieties over $\RR^k$. For each $\theta \in \RR^k$, the fiber over $\theta$ is the hypertoric variety $X$ defined with the choice of stability condition $\theta$, which is smooth if $\theta$ is away from all the walls. The geometry of $X$ stays the same for all $\theta$'s in a given chamber $\mathfrak{K}$, and admits a symplectic flop phenomenon when $\theta$ crosses a wall.

With the space $\RR^k$ identified with $H^2 (X, \RR)$, the real moment map can also be understood as a real period map, whose image is the real K\"ahler class. In this sense, the chamber $\mathfrak{K}$ is nothing but the \emph{K\"ahler cone} of $X$. The effective cone $\Eff (X) \otimes \RR$ can be described as the dual of $\mathfrak{K}$.

\subsection{Roots, walls and chambers} \label{section-root}

Following Maulik--Okounkov \cite{MO}, the space of equivariant parameters $\RR^d = \RR^n / \RR^k$ also admits a wall-and-chamber structure. Choose a cocharacter of the torus $\sigma : \CC^* \to (\CC^*)^d$, which we view as an element in $\ZZ^d$. Let $X^\sigma \subset X$ be the fixed loci of $X$ under the action of the 1-dimensional subtorus defined by $\sigma$. We also choose a lift $\widetilde \sigma \in \ZZ^n$ of $\sigma$ along $\beta$, i.e. $\beta (\widetilde\sigma) = \sigma$.

The subset in $\RR^d$ where $\sigma$ is ``generic", i.e., $X^\sigma = X^{(\CC^*)^d} = X^\TT$, is the complement of a union of hyperplanes, which we call \emph{walls}. Each wall is of the form $W_\alpha := \{ \lambda \in \RR^d \mid \langle \lambda, \alpha \rangle = 0 \}$, for some primitive $\alpha \in (\ZZ^d)^\vee$, which we call a \emph{root}.

Alternatively, if we identify $\alpha$ with its image $\beta^* \alpha$ in $(\ZZ^n)^\vee$, then a root is a minimal relation among images of the standard basis vectors $\iota^\vee e_i^*$. Therefore, we may also identify $\alpha$ as a \emph{cocircuit}, i.e. a subset $R \subset \{1, \cdots, n\}$, with unique splitting $R = R^+ \sqcup R^-$, such that
$$
\alpha_R := \sum_{i\in R^+} e_i^* - \sum_{i\in R^-} e_i^* \in \ker \iota^\vee,
$$
and $\langle \alpha_R, \widetilde\sigma \rangle \geq 0$.

A connected component of the complement of the union of all root hyperplanes is also called a \emph{chamber}:
$$
\fC \subset \RR^d \backslash \bigcup_{\alpha: \text{cocircuit}} W_\alpha.
$$
For a fixed choice of $\widetilde\sigma$, a root $\alpha$ is called positive if it is nonzero and $\langle \widetilde\sigma, \alpha \rangle \geq 0$.

\subsection{Equivariant and K\"ahler parameters} \label{sec-para}

To end this section, we would like to elaborate more on the parameters which our vertex functions and elliptic stable envelopes will depend on. Recall the multiplicative version of the short exact sequence (\ref{knd-seq}):
$$
\xymatrix{
	1 \ar[r] & \bK \ar[r]^-{\exp \iota} & \bT \ar[r]^-{\exp \beta} & \bA \ar[r] & 1,
}
$$
where $\bK = (\CC^*)^k$, $\bT = (\CC^*)^n$ and $\bA = (\CC^*)^d$. We have used the action by $(\CC^*)^n$ in the previous section to define the equivariant $K$-theory. The coordinates $a_1, \cdots, a_n$ on $\bT$ are called the \emph{equivariant parameters}. They correpond to the standard basis on $\ZZ^n$, which we fix once and for all.

However, this torus action is actually redundant: it acts by the factorization through the morphism $\exp \beta: \bT \to \bA$. The actual non-redundant equivariant parameters are functions on the quotient torus $\bA$, or in other words, monomials in $a_1, \cdots, a_n$ that vanish on the kernel $\bK$.

Every choice of basis on $\ZZ^d$, or equivalently, every presentation of the map $\beta$, defines a particular choice of coordinates on $\bA$. For example, let $\bp$ be a vertex in the hyperplane arrangement $\cH$. If we choose the presentation $\beta = (-C , I)$ in the standard $\bp$-frame, the corresponding choice of coordinates would be
$$
\alpha_i (\bp) :=  a_i \prod_{j\not\in \bp} a_j^{-C_{ij}}, \qquad i\in \bp.
$$

The same happens for the K\"ahler parameters. Consider the dual exact sequence of tori
$$
\xymatrix{
	1 \ar[r] & \bA^\vee \ar[r]^-{\exp \beta^\vee} & \bT^\vee \ar[r]^-{\exp \iota^\vee} & \bK^\vee \ar[r] & 1.
}
$$
We fix coordinates on $z_1, \cdots, z_n$ on $\bT^\vee$ once and for all, as the ``redundant" K\"ahler parameters. Then the non-redundant K\"ahler parameters are coordinates on the quotient torus $\bK^\vee$, or equivalently, monomials in $z_1, \cdots, z_n$, vanishing on the dual kernal $\bA^\vee$.

Every choice of basis on $\ZZ^k$, or equivalently, every presentation of the map $\iota$, defines a particular choice of coordinates on $\bK^\vee$. For the standerd $\bp$-frame, $\iota = \begin{pmatrix}
I \\
C
\end{pmatrix}$, there is a particular choice of K\"ahler parameters, or in other words,  a choice of representatives for the K\"ahler parameters
$$
\zeta_j (\bp) := z_j \prod_{i\in \bp} z_i^{C_{ij}}, \qquad j\not\in \bp.
$$

The vertex function for $X$, once defined, will be a $K_\TT (X)$-valued function, over the product of the dual K\"ahler torus $\bK^\vee$ and the equivariant torus $\bA$ (and furthermore, certain partial compactifications of them). In particular, for any effective curve $\beta$, with associated circuit $S = S^+ \sqcup S^-$, the monomial
$$
z^\beta := \prod_{i\in S^+} z_i \prod_{i\in S^-} z_i^{-1}
$$
is a well-defined function on the quotient torus $\bK^\vee$. Similarly, for any root $\alpha$ \footnote{Unfortunately, we use $\alpha$ for both roots and equivariant parameters. To distinguish them, an equivariant parameter will always be followed with a dependence on its associated fixed $\bp$, e.g., $\alpha_i (\bp)$.}, with associated cocircuit $R = R^+ \sqcup R^-$, the monomial
$$
a^\alpha := \prod_{i\in R^+} a_i \prod_{i\in R^-} a_i^{-1}
$$
is a well-defined function on the quotient torus $\bA$.

\subsection{Polarization}

The notion of polarization, although not essentially, is important for our formation of elliptic stable envelopes and vertex functions.

\begin{Definition} \label{Defn-pol}
Let $X$ be a hypertoric variety.
\begin{enumerate}[1)]
	
	\setlength{\parskip}{1ex}

\item A collection of $K$-theoretic classes $T^{1/2}_{\bp} \in K_\TT (\bp)$, for $\bp\in X^\TT$, is called a \emph{localized polarization}, if it satisfies
$$
T^{1/2}_\bp + \hbar^{-1} (T_\bp^{1/2})^\vee = T_X |_\bp  \in K_\TT (\bp), \qquad \bp \in X^\TT.
$$

\item A localized polarization is called a \emph{global polarization}, or simply a polarization, if it comes from a global $K$-class, i.e., there exists $T_X^{1/2} \in K_\TT (X)$, such that $T^{1/2}_\bp = T_X^{1/2} |_\bp$.

\end{enumerate}

\end{Definition}

Usually a polarization will be given as a Kirwan lift in the $K_\TT (\fX)$, i.e., as a Laurent polynomial in the Chern roots $x_i$'s.

\vspace{3ex}

\section{Elliptic cohomology and stable envelopes}

\subsection{Equivariant elliptic cohomology for hypertoric varieties}

Let $q\in \CC^*$ be a complex number, with $|q|<1$, and let $E := \CC^* / q^\ZZ$ be the elliptic curve, with modular parameter $q$. Equivariant elliptic cohomology is a covariant functor that associates to every $\TT$-variety a scheme $\Ell_\TT (X)$. In this section, we describe explicitly the $\TT$-equivariant elliptic cohomology and its extended version of the hypertoric variety $X$. For general definitions of equivariant elliptic cohomology, we refer the readers to \cite{ell1,ell2,ell3,ell4,ell5,ell6}.

For $X = \pt$, the equivariant elliptic cohomology is the abelian variety
$$
\cE_\TT := \Ell_\TT (\pt) = E^{\dim \TT},
$$
the coordinates on which we refer to as the \emph{elliptic equivariant parameters}, and by abuse of notation, still denote by $a_1, \cdots, a_n$ and $\hbar$. Let $\bS (X) := E^k$, whose coordinates we call \emph{elliptic Chern roots}, and still denote by $s_1, \cdots, s_k$.

Let $X$ be a hypertoric variety, which is in particular, a GKM variety. The explicit description of equivariant $K$-theory $K_\TT (X)$ can be generalized to the elliptic setting, hence the following diagram
$$
\xymatrix{
	\Ell_\TT(X) \ar[d]  \ar@{^{(}->}[r] &  \bS (X) \times \cE_\TT \\
	\cE_\TT .
}
$$
Again, $\Ell_\TT (X)$ is a closed subvariety in the ambient space $\bS (X) \times \cE_\TT$, finite over $\cE_\TT$, with simple normal crossing singularities.

By $\TT$-localization, the irreducible components of $\Ell_\TT (X)$ are parameterized by the fixed point set $X^\TT$, and each of them is isomorphic to the base $\cE_\TT$. We denote by $\Or_\bp$ the irreducible component corresponding to a fixed point $\bp \in X^\TT$, and call it an \emph{orbit}. The fact that $X$ is a GKM variety implies that orbits are glued in a very nice way to form the scheme $\Ell_\TT (X)$:

\begin{Proposition}
	We have
	$$
	\Ell_{\TT}(X)=\Big(\coprod\limits_{\bp \in X^{\TT}}\, \Or_\bp \Big) /\Delta,
	$$	
	where $/\Delta$ denotes the intersections of $\TT$-orbits $\Or_\bp$ and $\Or_\bq$
		along the hyperplanes
	$$
	\Or_\bp \supset \chi^{\perp}_{C} \subset \Or_\bq,
	$$
	for all $\bp$ and $\bq$ connected by an equivariant curve $C$, and $\chi_{C}$ is the $\TT$-character of the tangent space $T_\bp C$.
	
The intersections of the orbits $\Or_\bp$ and $\Or_\bq$ are transversal and hence the scheme $\Ell_\TT (X)$ is a variety with simple normal crossing singularities.
\end{Proposition}

Let $\cE_{\bT^\vee} := E^n$ \footnote{In \cite{AOelliptic}, this is called $\cE_{\Pic_\bT (X)}$.} be the space whose coordinates we call the \emph{elliptic K\"ahler parameters} and still denote by $z_1, \cdots, z_n$. The \emph{extended equivariant elliptic cohomology} of $X$ is defined to be the product
$$
\textsf{E}_\TT (X) := \Ell_\TT (X) \times \cE_{\bT^\vee}.
$$
It admits the same structure as above $\textsf{E}_\TT (X) = (\coprod\limits_{\bp \in X^{\TT}}\, \widehat\Or_\bp ) /\Delta$, with $\widehat\Or_\bp := \Or_\bp \times \cE_{\bT^\vee}$ and $\Delta$ the same as above.

\begin{Remark}
	Recall that in (\ref{V-frame-i}) and (\ref{V-frame-b}), we can choose different presenation of $\iota$ and $\beta$, for different fixed points $\bp$. That means for different orbits $\Or_\bp$, we have chosen different coordinates $s_1, \cdots, s_k$ on the ambient space $\bS (X)$.
\end{Remark}

\subsection{Elliptic functions}

We define the theta function associated with $E$ explicitly by:
$$
\vartheta (x) := (x^{1/2} - x^{-1/2} ) \prod_{d=1}^\infty (1 - q^d x) \prod_{d=1}^\infty ( 1- q^d x^{-1} ), \qquad x \in \CC^*.
$$
Note that $\vartheta (1) = 0$ and $\vartheta (qx) = - q^{-1/2} x^{-1} \vartheta (x)$, which means that $\vartheta (x)$ defines a section of a line bundle of degree one on the elliptic curve $E$. It will be convenient to describe sections for line bundles on product of elliptic curves in terms of theta functions.

For a sum of variables $\sum_i x_i$, we denote
$$
\Theta \Big( \sum_i x_i \Big) := \prod_i \vartheta (x_i).
$$

\subsection{Elliptic stable envelopes}

Recall that $\bA= (\CC^*)^n / (\CC^*)^k$ is the quotient torus. Since $\bK = (\CC^*)^k$ acts on $X$ trivially, $\bA$ is the actual non-redundant torus acting on $X$. Equivariant parameters on $\bA$ can be viewed as functions on $(\CC^*)^n$ that vanish on $(\CC^*)^k$. As in Section \ref{sec-para}, For a given fixed point $\bp \in X^\TT$ and the correponding standard $\bp$-frame, there is a convenient choice of coordinates of equivariant parameters:
$$
\alpha_i (\bp) = a_i \prod_{j\not\in \bp} a_j^{-C_{ij}}, \qquad i\in \bp .
$$

The elliptic stable envelope depends on the  choice of a polarization (see Definition \ref{Defn-pol}). In this section, we choose the polarization as $T^{1/2}_X = \sum_{i=1}^n L_i - \cO^{\oplus k}$, or $\sum_{i=1}^n x_i - k$, written in terms of Chern roots.

One also needs to choose an element $\sigma \in \ZZ^d$, which determines a cocharacter $\sigma: \CC^* \to \bA$. Let $\widetilde\sigma \in \ZZ^n$ be a lift of $\sigma$. We assume that $\sigma$ is chosen generically, such that the fixed point set $X^\sigma$ under the 1-dimension subtorus action by $\sigma$ is the same as $X^\TT$.

The choice of $\sigma$ determines a splitting $\bp = \bp^+ \sqcup \bp^-$, where
\begin{equation} \label{split-V}
\bp^+ := \{ i \in \bp \mid \langle \alpha_i (\bp)  , \sigma \rangle >0 \},
\end{equation}
and similar with $\bp^-$.

We denote by
$$
\Attr_\sigma (\bp) := \{ x\in X \mid \lim_{t\to \infty} \sigma (t) \cdot x = \bp \}
$$
the attracting set of the fixed point $\bp$. The full attracting set $\Attr^f_\sigma (\bp)$ is defined to be the minimal closed subset of $X$ which contains $\bp$ and is closed under taking $\Attr_\sigma(\cdot)$

The attracting set has the following description, originally given \cite{She}. Recall that the hyperplane arrangement $\cA$ associated with $X$ gives a collection of affine hyperplanes $H_i$ in $\RR^d$. The space $\RR^d$ is then divided into polytopes by $H_i$'s, where each polytope corresponds to a $\TT$-invariant lagrangian submanifold in $X$. Each fixed point $\bp$ is identified with a vertex, which is the intersection $\bigcap_{i\in \bp} H_i$. Given the chosen cocharacter $\sigma$, consider the cone defined as
$$
\bigcap_{i\in \bp^+} H_i^+ \cap \bigcap_{i\in \bp^-} H_i^-.
$$
The attacting set $\Attr_\sigma (\bp)$ is then the union of lagrangians corresponding to the polytopes in this cone. In particular, we have:

\begin{Lemma} \label{Attr}
Given $\bp, \bq \in X^\TT$,  $\bq \in \overline{\Attr_\sigma (\bp)}$ if and only if $\bq \in \bigcap_{i\in \bp^+} H_i^+ \cap \bigcap_{i\in \bp^-} H_i^-$, which is also equivalent to $\cA_\bq^- \cap \bp^+ = \cA_\bq^+ \cap \bp^- = \emptyset$.
\end{Lemma}

The polarization, when restricted to a fixed point, decomposes into three parts $T^{1/2}_X |_\bp = T^{1/2}_X |_\bp^+ + T^{1/2}_X |_\bp^\bA + T^{1/2}_X |_\bp^-$, whose characters are positive, zero and negative with respect to $\sigma$. For our choice of polarization $T^{1/2}_X = \sum_{i=1}^n x_i - k$, we see that $T^{1/2}_X |_\bp^A = \sum_{j\not\in \bp} x_j |_\bp$, $T^{1/2}_X |_\bp^\pm = \sum_{i\in \bp^\pm} x_i |_\bp$.

The elliptic stable envelope \cite{AOelliptic} associates each fixed point $\bp$ with a section $\Stab_\sigma (\bp)$ of a certain explicit line bundle $\mathcal{T} (\bp)$ on the orbit $\widehat\Or_{\bp} \subset \mathsf{E}_\bT (X)$, which is the  pull-back of a line bundle $\cT^{os} (\bp)$ on the ambient abelian variety $\mathsf{S}(X) \times \cE_\TT \times \cE_{\bT^\vee}$ along the elliptic Chern class map $\mathsf{E}_\TT (X) \to \mathsf{S}(X) \times \cE_\TT \times \cE_{\bT^\vee}$. The line bundle is uniquely determined by the $q$-quasi-periods of its sections, which can be read off from the explicit formula for $\Stab_\sigma (\bp)$ we give in Theorem \ref{Stab-formula}.

\begin{Theorem}[\cite{AOelliptic}]
There exists a unique section $\Stab_\sigma (\bp)$ of $\cT(\bp)$, holomorphic over $\cE_\TT$, meromorphic over $\cE_{\bT^\vee}$ and in $\hbar$, satisfying the following properties:
\begin{enumerate}[(i)]
\setlength{\parskip}{1ex}

\item $\Supp (\Stab_\sigma (\bp)) \subset \Attr^f_\sigma (\bp)$;

\item its restriction to $\bp$ is
$$
\Stab_\sigma (\bp) |_\bp = (-1)^{|\bp^+|} \Theta (N_\bp^-) = \prod_{i\in \bp^+} \vartheta (\hbar x_i |_\bp) \prod_{i\in \bp^-} \vartheta (x_i |_\bp) ,
$$
where $N_\bp^-$ is the negative half in the decomposition $T_\bp X = N_\bp^+ + N_\bp^-$ which pairs with $\sigma$ negatively.

\end{enumerate}
\end{Theorem}

We have the following explicit formula for the elliptic stable envelope. Note that in our hypertoric case, $\Stab_\sigma(\bp)$ is actually supported on $\overline{\Attr_\sigma (\bp)}$, which is stronger than the general definition. 

\begin{Theorem} \label{Stab-formula}
The stable envelope $\Stab_\sigma (\bp)$ is the pull-back along $\mathsf{S}(X) \times \cE_\TT \times \cE_{\bT^\vee}$ of the following section
$$
\prod_{i \in \bp^+} \vartheta ( \hbar x_i ) \prod_{i \in \bp^-} \vartheta (x_i ) \prod_{j \in \cA_\bp^+} \dfrac{\vartheta \Big( x_j \zeta_j (\bp) \hbar^{- \sum_{i \in \bp^+} C_{ij}} \Big) }{\vartheta \Big( \zeta_j (\bp) \hbar^{- \sum_{i \in \bp^+} C_{ij}}  \Big) } \prod_{j \in \cA_\bp^-} \dfrac{\vartheta \Big( x_j \zeta_j (\bp) \hbar^{- \sum_{i \in \bp^+} C_{ij}} \Big) }{\vartheta \Big( \hbar^{-1} \zeta_j (\bp) \hbar^{- \sum_{i \in \bp^+} C_{ij} }  \Big) }
$$
of the line bundle $\cT^{os} (\bp)$.
\end{Theorem}

\begin{proof}
The shift of $\hbar$-factors here is exactly determined by the line bundle $\cT^{os} (\bp)$. It suffices to check that the given section satisfies (i) and (ii), which follow from Lemma \ref{Attr} and direct computation.
\end{proof}

\vspace{3ex}


\section{Vertex function and quantum product}

\subsection{Quasimaps and bare vertex}

Recall that the hypertoric variety $X = \mu^{-1}(0) /\!/_\theta \bK = \mu^{-1}(0)^s / \bK$ is a subscheme of the quotient stack $\fX = [\mu^{-1} (0) / \bK ] $. Consider the following definition from \cite{Oko}.

\begin{Definition}
A quasimap from $\PP^1$ to $X$ is a morphism $f: \PP^1 \to \fX$, such that away from a $0$-dimensional subscheme in $\PP^1$, the morphism $f$ maps into the stable locus $X\subset \fX$.
\end{Definition}

Let $L_i$, $1\leq i\leq n$ and $N_j$, $1\leq j \leq k$ be the tautological line bundles\footnote{Here we use the same notations for tautological line bundles on $\fX$ and $X$, since they are defined in the same manner and compatible to each other via the inclusion $X \subset \fX$. } on $\fX$, defined similarly as in previous sections. The datum of a quasimap $f$ is equivalent to the collection of line bundles $f^* N_j$, and sections of $\bigoplus_{i=1}^n L_i \oplus \hbar^{-1} \bigoplus_{i=1}^n L_i^{-1}$, satisfying the moment map equations and stability condition. A point on $\PP^1$ is called a base point for $f$ if it is mapped into the unstable locus $\fX \backslash X$.

For a quasimap $f$, taking the degrees of $f^* N_j$ defines a homomorphism $H^2 (X, \ZZ) \to \ZZ$, or equivalently, a curve class $\deg f \in H_2 (X, \ZZ)$, which we call the degree class of $f$. We will call $\deg f$ \emph{effective}, if it lies in $\Eff (X)$, the monoid of effective curve classes in $X$.

Now we consider the moduli of quasimaps. Let $\Hom(\PP^1, \fX)$ be the moduli stack parameterizing all representable morphisms from $\PP^1$ to $\fX$, which is an Artin stack locally of finite type. Since the domain of quasimaps is fixed, the universal curve is a trivial family over the moduli stack, fitting in the following universal diagram
$$\xymatrix{
	\Hom(\PP^1, \fX) \times \PP^1 \ar[r]^-f \ar[d]_\pi & \fX \\
	\Hom(\PP^1, \fX).
}$$
There is a perfect obstruction theory, with virtual tangent complex given by
$$T_\vir = R^\bullet \pi_* f^* T_\fX.$$
One observes that the obstruction part actually vanishes. Hence $\Hom(\PP^1, \fX)$ is a smooth Artin stack.

Fix $\beta \in \Eff (X)$, and $p_1=0$,  $p_2=\infty \in \PP^1$. Let $\QM(X, \beta)$ be the stack parameterizing quasimaps from $\PP^1$ to $\fX$, which lies in $\Hom(\PP^1, \fX)$ as an open substack, and is of finite type since we fix the degree. Hence $\QM(X, \beta)$ is a Deligne--Mumford stack, equipped with the inherited perfect obstruction theory. The standard construction in \cite{BF, Lee} defines a virtual structure sheaf
$$
\cO_\vir \in K_\TT (\QM(X, \beta)).
$$

By Okounkov \cite{Oko}, it is more natural to twist the obstruction theory by a certain square root line bundle and form a modified virtual structure sheaf.

Let $T^{1/2}_X$ be a fixed global polarization. The modified virtual structure sheaf is defined as:
\begin{equation} \label{defn-cO_vir}
\widehat\cO_\vir:= \cO_\vir \otimes \left( K_\vir \frac{\det f^* \cT^{1/2} \big|_{p_2}}{\det f^* \cT^{1/2} \big|_{p_1}} \right)^{1/2} \in K_\TT (\QM(X, \beta)),
\end{equation}
where $K_\vir:= \det T_\vir^\vee$, and $\cT^{1/2}$ is the tautological bundle associated with a lift of the polarization $T^{1/2}_X$ to $K_\TT (\fX)$. 
The existence of the square root line bundle $K_\vir^{1/2}$ relies on the existence of a polarization on the target $X$, and a spin structure $K_{\PP^1}^{1/2}$ on the domain, which only makes sense if the target is symplectic and the domain is fixed. This crucial modification has the effect of making the obstruction theory equivariantly symmetric.

In order to define invariants, one has to be able to insert $K$-theoretic classes from $X$. Consider the following open substack
$$
\QM(X, \beta)_{\ns\ p_2} \subset \QM(X, \beta)
$$
consisting of those quasimaps for which $p_2\in \PP^1$ is nonsingular, or in other words, \emph{not} a base point. There are evaluation maps
$$
\ev_2: \QM(X, \beta)_{\ns\ p_2} \to X, \qquad \ev_1: \QM(X, \beta)_{\ns\ p_2} \to \fX,
$$
where a quasimap $f$ is mapped to the image $f(p_2)$ or $f(p_1)$. Let $\CC^*_q$ be the torus on $\PP^1$, where $q\in K_{\CC_q^*}(\pt)$ is the character defined by $T_{p_1} \PP^1$.

A general $K$-theoretic invariant one can define is to pair the modified virtual structure sheaf $\widehat\cO_\vir$ with classes pulled back from $X$ or $\fX$ via $\ev_2$ or $\ev_1$, and then push forward to $X$ or alternatively, to $\pt$. As the stack $\QM(X, \beta)_{\ns\ p_2}$ is not proper, but admits proper fixed loci under the action of $\TT \times \CC^*_q$. the push-forward in the last step is only well-defined if we work in the \emph{localized} $\CC^*_q$-equivariant theory, i.e., to work in $K_{\TT \times \CC^*_q} (X)_\loc := K_{\TT \times \CC^*_q} (X) \otimes \CC(q)$.

\begin{Definition}
	Given $\tau\in K_\TT(\fX)$, the \emph{bare vertex function with decendent insertion $\tau$} is defined as
	$$
	V^{(\tau)}(q,z) := \sum_{\beta \in \Eff(X)} z^\beta \ev_{2, *} \left( \QM(X, \beta)_{\ns\ p_2}, \widehat\cO_\vir \cdot \ev_1^* \tau  \right) \quad \in K_{\TT\times \CC^*_q} (X)_\loc [[ z^{\Eff(X)} ]] .
	$$
\end{Definition}

\begin{Remark}
General definitions of stable quasimaps are given in \cite{CKM}. Our definition of quasimaps with fixed domain $\PP^1$ is the special case of Definition 7.2.1 there, as stable quasimaps of genus $0$ to $X$, with one parameterized domain component, and without any marked points.
\end{Remark}

\subsection{Relative quasimaps and capped vertex}

Another version of the vertex function comes from counting the relative quasimaps,  motivated from the relative DT/PT theory. Let $l\geq 0$ be an integer. Consider the following
$$
\PP^1[l]:= \PP^1 \cup \PP^1 \cup \cdots \cup \PP^1,
$$
constructed by attaching a chain of $l$ $\PP^1$'s to the domain $\PP^1$ at the point $\infty$. The newly attached rational curves are called bubbles or rubber components. Let $p_2$ be the point $\infty$ on the last bubble.

\begin{Definition}
A \emph{relative} quasimap to $X$ is a map $f: \PP^1 [l] \to \fX$, for some integer $l\geq 0$, such that it generically maps into the stable locus $X\subset \fX$, and moreover, the point $p_2$ is not a base point.
\end{Definition}

Rather than quasimaps with fixed domain $\PP^1$ in the previous section, relative quasimaps admit nontrivial automorphisms from scaling the bubbles. To describe this construction, we adopt the language from the relative GW/DT/PT theory. For more detailed definitions, see for example \cite{Li, ACFJ, OP, Oko, Zhou}.

By constructions in \cite{Li, ACFJ}, there exists a smooth Artin stack $\cB$, together with a universal famity $\cC$ over $\cB$, parameterizing all possible \emph{extended pairs} of the form $(\PP^1[l], p_2)$, in the following sense. For any geometric point $\Spec \CC \to \cB$, the fiber of the family $\cC \to \cB$ over that point is of the form $\PP^1 [l]$, for some integer $l\geq 0$. Moreover, the automorphism group for this point is $(\CC^*)^l$, acting on the fiber $\PP^1 [l]$ by scaling the $l$ bubbles.

\begin{Definition}
Let $S$ be a scheme. A family of expanded pairs of $(\PP^1, \infty)$ over $S$ is the family $\cC_S \to S$, arising from a Cartesian diagram of the following form
$$
\xymatrix{
	\cC_S \ar[d] \ar[r] & \cC \ar[d] \\
	S \ar[r] & \cB.
}
$$
A \emph{family of relative quasimaps} (with respect to the divisor $\infty \in \PP^1$) is a family of expanded pairs $\pi: \cC_S \to S$, together with a map $f: \cC_S \to \fX$, such that over each geometric point $s\in S$, the fiber gives a relative quasimap.
\end{Definition}

\begin{Definition}
A relative quasimap $f: \PP^1 \to \fX$ is called \emph{stable}, if its automorphism group is finite. Equivalently, it means that the degree of $f$ on each bubble is nontrivial.
\end{Definition}

Let $X_0$ be the affine quotient, defined by $\mu^{-1}(0) /\!/_{\theta = 0} \bK$. Let $\QM (X, \beta)_{\rel \ p_2}$ be the stack parameterizing all stable relative quasimaps to $X$. Standard argument as in relative DT/PT theory shows that it is a DM stack of finite type, proper over $X_0$, and it admits a perfect obstruction theory, relative over the smooth Artin stack parameterizing principal $\bK$-bundles over the fibers of $\cC \to \cB$. After the same twisting as in (\ref{defn-cO_vir}), we have the modified virtual structure sheaf $\widehat\cO_\vir$.

Similarly we have the evaluation maps
$$
\ev_2: \QM(X, \beta)_{\rel \ p_2} \to X, \qquad \ev_1: \QM(X, \beta)_{\rel\ p_1} \to \fX.
$$
However, this time $\ev_2$ is \emph{proper}, and one can work in the \emph{non-localized} $\CC^*_q$-equivariant theory or the non-equivariant theory.

\begin{Definition}
	Given $\tau\in K_\TT(\fX)$, the \emph{capped vertex function with descendent insertion $\tau$} is defined as
	$$
	\widehat V^{(\tau)} (q,z) := \sum_{\beta \in \Eff(X)} z^\beta \ev_{2, *} \left( \QM(X, \beta)_{\rel\ p_2}, \widehat\cO_\vir \cdot \tau\big|_{p_1} \right) \quad \in K_{\TT\times \CC^*_q} (X) [[ z^{\Eff(X)} ]].
	$$
	Its non-equivariant counterpart with respect to $\CC^*_q$, defined by the non-equivariant pushforward along the evaluation map, is called the \emph{quantum tautological class}:
	$$
	\widehat\tau(z):= \widehat V^{(\tau)}(z) \big|_{q=1} \quad \in K_\TT (X) [[ z^{\Eff(X)} ]].
	$$
\end{Definition}

\begin{Remark}
The notion of relative quasimaps is also a special case of Definition 7.2.1 in \cite{CKM}. For each domain $\PP^1[l]$ of the relative quasimaps, there is a contraction map $\PP^1[l] \to \PP^1$, mapping all bubble components to the node where the bubbles are attached to the rigid $\PP^1$. The contraction map can also be defined in families. Let $\QM_{0, 1}^{CKM} (X, \beta)_{para}$ be the moduli stack of stable quasimaps of genus $0$ to $X$, with one parameterized domain component, and $1$ marked point $p$. The moduli space of relative quasimaps $\QM(X, \beta)_{\rel \ p_2}$ is the closed substack in $\QM_{0, 1}^{CKM} (X, \beta)_{para}$, consisting of quasimaps where the marked point $p$ is mapped to $p_2 \in \PP^1$ under the contraction map.
\end{Remark}

\subsection{Gluing operator, degeneration and PSZ quantum $K$-theory} \label{QK}

More generally, if we choose $N$ points $p_1, \cdots, p_N$ on $\PP^1$, one can specify a ``nonsingular" or ``relative" condition at each point, and insert appropriate descendent insertions at other points. Pushforward by the evaluation map $\ev_1\times \cdots \times \ev_N$ would define a $K$-theory class in $K(X)^{\otimes N}$, i.e. an $N$-point class. However, one has to work with the $\CC_q^*$-equivariant theory or non-equivariant theory, depending on the following specific cases.

(i) If $N\leq 2$, and all points $p_i$ (identified with either $0$ or $\infty\in \PP^1$) are equipped with the ``relative" condition, this $N$-point class can be defined to lie in the non-localized ring $K_{\TT\times \CC_q^*}(X)^{\otimes N} [[ z^{\Eff(X)} ]]$.

(ii) If $N\leq 2$ and some of the points $p_i$ (identified with either $0$ or $\infty \in \PP^1$) are equipped with the ``nonsingular" condition, then it lies in the localized ring $K_{\TT\times \CC_q^*}(X)_\loc^{\otimes N} [[ z^{\Eff(X)} ]]$.

(iii) Finally, if $N\geq 3$, then $\CC_q^*$ does not preserves the points $p_i$'s, and all $p_i$'s have to be assigned the ``relative" condition. The push-forward map only makes sense $\CC^*_q$-non-equivariantly, producing an $N$-point class in $K_{\TT}(X)^{\otimes N} [[ z^{\Eff(X)} ]]$.

Let $K_X$ denote the canonical line bundle. A natural nondegenerate bilinear form on $K_\TT(X)$ is
$$
(\cF, \cG):= \chi \Big( X, \cF\otimes \cG \otimes K_X^{-1/2} \Big), \qquad \cF, \cG \in K_\TT (X).
$$
By the Fourier--Mukai philosophy, an $(N+M)$-point class in $K_\TT (X)^{\otimes (N+M)}$ can be viewed as an operator $K_\TT (X)^{\otimes N} \to K_\TT (X)^{\otimes M}$, where we are free to choose $N$ points as the input and $M$ points as the output. In particular, a class  $\bO \in K_\TT (X)^{\otimes 2}$ defines an operator $\bO \in \End K_\TT (X)$ by $\bO \cG := \pr_{1*} \left( \bO \otimes \pr_2^* \cG \right)$. We will often abuse the same notation for both the class in $K_\TT (X)^{\otimes 2}$ and the operator in $\End K_\TT (X)$.

\begin{Definition}
The \emph{gluing operator} is defined as
$$
G(q,z) := \sum_{\beta \in \Eff(X)} z^\beta (\ev_{p_1}\times \ev_{p_2})_* \left( \QM(X, \beta )_{\rel\ p_1, \rel\ p_2}, \widehat\cO_\vir \right) \in K_{\TT \times \CC_q^*} (X)^{\otimes 2} [[ z^{\Eff(X)} ]].
$$
\end{Definition}

The gluing operator plays a crucial role in the degeneration of quasimap theory. Note that the $\beta =0$ part has only contribution from the constant maps, and therefore $G = (\Delta_X)_* K_X^{1/2} + O(z) \in K_{\TT \times \CC_q^*} (X)^{\otimes 2} [[ z^{\Eff(X)} ]]$. In particular, $G = K_X^{1/2} \otimes  + O(z)$ as an operator, and admits an inverse $G^{-1}$.

We introduce the following $N$-to-$M$ operator
$$
C(p_{N+1}, \cdots, p_{N+M} \mid p_1, \cdots, p_N ) : K_\bT(X)^{\otimes N} \to K_\bT (X)^{\otimes M} [[z^{\Eff (X)} ]],
$$
defined by quasimap counting. Let $p_1, \cdots, p_{N+M}$ be points on $C = \PP^1$.

\begin{Definition}
For any $N, M \geq 0$, define
\ben
&& C(p_{N+1}, \cdots, p_{N+M} \mid p_1, \cdots, p_N ) (\cF_1, \cdots, \cF_N) \\
&:=& \sum_{\beta\in \Eff (X)} z^\beta ( \ev_{p_{N+1}, \cdots, p_{N+M}} )_* \left( \QM(X, \beta )_{\rel p_1, \cdots, p_{N+M}} , \widehat\cO_\vir \otimes \prod_{i=1}^N \ev_{p_i}^* (G^{-1} |_{q= 1} \cF_N) \right),
\een
where the push-forward $\ev_{p_{N+1}, \cdots, p_{N+M}}$ is taken to be $\CC^*_q$-non-equivariant.
\end{Definition}

Now suppose that the domain $C=\PP^1$ degenerates into the union of two rational curves $C_0 = C_+ \cup C_-$. Degeration formula relates the quasimap theory with domain $C$ to the relative quasimap theories with domains $C_\pm$, with respect to the new relative points introduced by the node. The degeneration formula \cite{Oko} states that the operators satisfy the identity
$$
C(p_{N+1}, \cdots, p_{N+M} \mid p_1, \cdots,  p_N ) = C_+( p_{N+1}, \cdots, p_{N+M} \mid \bullet ) \circ G^{-1} |_{q=1} \circ C_-( \bullet \mid p_1, \cdots, p_N ),
$$
where $\bullet$ denotes the new relative point introduced by the node. Moreover, if $N+M \leq 2$, the degeneration formula also works $\CC_q^*$-equivariantly, and the non-equivariant gluing $G^{-1} |_{q=1}$ should be replaced by the equivariant version $G^{-1}$.

In particular, the quantum tautulogical class associated to the identity can be recovered as the $1$-point function: $\hat\bone (z) = C(p_1 \mid \ )$, and the specialization of the inverse gluing operator can be recovered as the $2$-point function: $\chi(X, \cF \otimes ( G(q, z)^{-1} |_{q=1}  \cdot \cG ) ) = C( \ | p_1, p_2) (\cF, \cG)$.

We are particularly interested in the case $(N,M) = (2,0)$ and $(2,1)$. The following definition is made by Pushkar--Smirnov--Zeitlin \cite{PSZ}.

\begin{Definition}
The \emph{PSZ quantum K-theory ring} of $X$ is defined as the vector space $K_\TT(X) [[ z^{\Eff(X)} ]]$, equipped with the \emph{quantum pairing} $\langle\ , \ \rangle_{PSZ}$, the \emph{quantum product} $*$, and the \emph{quantum identity} $\hat\bone (z)$, defined as follows
$$
\langle \cF, \cG \rangle_{PSZ} := C(\ \mid p_1, p_2) (\cF, \cG), \qquad \cF * \cG := C(p_3 \mid p_1, p_2) (\cF, \cG), \qquad \hat\bone (z) := C(p_1 \mid \ ).
$$

\end{Definition}

As in \cite{PSZ}, one can prove that this is actually what we expect:

\begin{Proposition}
The quantum $K$-theory ring defined as above is a Frobenius algebra.
\end{Proposition}

\subsection{Capping operator and $q$-difference equation}

The two versions of vertex functions are related by the following \emph{capping operator}
\footnote{
Our convention here is different from \cite{Oko} and \cite{PSZ}: $p_1$ and $p_2$ are exchanged. Under this convention we always view $K$-classes from $p_1$ as inputs and those from $p_2$ as outputs. Our $\Psi$ is the conjugate of the original one, with $q \mapsto q^{-1}$.
}
$$
\Psi(q,z) := \sum_{\beta \in \Eff(X)} z^\beta (\ev_{p_1}\times \ev_{p_2})_* \left( \QM(X, \beta )_{\ns\ p_1, \ \rel\ p_2}, \widehat\cO_\vir \right) \in K_{\TT \times \CC_q^*} (X)_\loc^{\otimes 2} [[ z^{\Eff(X)} ]],
$$
where the notation means we consider quasimaps that are required to be nonsingular at $p_1\in \PP^1$ and allowed to bubble out at $p_2\in\PP^1$. The relationship between the bare and capped vertices is the following, which is proved in \cite{Oko} by relative localization and degeneration.

\begin{Proposition} \label{capping-eqn}
View $\Psi (q, z)$ as in $\End K_{\TT \times \CC^*_q} (X)_\loc$. We have
	$$
	\hat V^{(\tau)}(q,z) = \Psi(q,z) \, V^{(\tau)}(q,z) .
	$$
\end{Proposition}

\begin{Remark}
	Note that both factors on the right hand side lie in the localized K-group, while the left hand side is a non-localized $K$-class. This would be clear in a more precise way as we consider the $q\to 1$ asymptotic of this equality.
\end{Remark}

Analogous to the quantum differential equation in the cohomological theory, the capping operator satisfies a certain $q$-difference equation. Consider the operator
\begin{eqnarray*}
	M_i(q,z) &:=& \sum_{\beta \in \Eff(X)} z^\beta (\ev_{p_1}\times \ev_{p_2})_* \left( \QM(X, \beta )_{\rel\ p_1, \ \rel\ p_2}, \widehat\cO_\vir \cdot \det H^\bullet (L_i \otimes \pi^* \cO_{p_2}) \right) \circ G^{-1} \\
	&\in& K_{\TT \times \CC_q^*} (X)^{\otimes 2} [[ z^{\Eff(X)} ]],
\end{eqnarray*}
where $\pi: \PP^1 [l]\to \PP^1$ is the projection, and the term $H^\bullet(-)$ here means the (virtual) tautological bundle on the moduli space $\QM(X, \beta)$ whose fiber at a quasimap is represented by the (virtual) space $H^\bullet(L_i \otimes \pi^* \cO_{p_2})$. It is important that the operator $M_i(z)$ here lies in the non-localized K-group. The following equation is proved by considering the degeneration of capping operators with descendents \cite{Oko, OS}.

\begin{Proposition} \label{Psi}
View $\Psi (q, z)$ as in $\End K_{\TT \times \CC^*_q} (X)_\loc$.	The capping operator $\Psi(z)$ satisfies
	$$
	Z_i \Psi(q, z) = (L_i^{-1} \otimes ) \circ \Psi(q,z) \circ M_i (q^{-1} ,z),
	$$
	where $Z_i$ is the operator which shifts $z_i \mapsto q z_i$ and keep the other K\"ahler parameters.
	
	Moreover, the non-equivariant limit (with respect to $\CC_q^*$) of $M_i$ is the quantum multiplication by the quantum tautological line bundle:
	$$
	M_i(1, z) = \widehat L_i (z) *  .
	$$
\end{Proposition}

\subsection{Divisorial quantum $K$-theory}

In this subsection we would like to slightly modify the PSZ ring structure, and obtain one that can be compared with Givental's quantum $K$-theory. Recall that $\fX = [\mu^{-1} (0) / \bK ]$ is the stacky quotient, and $K_\TT (\fX) \cong K_\TT (B \bK)$ consists of Laurent polynomials in $x_1, \cdots, x_n$, up to the linear relations:
$$
K_\TT (\fX) \cong \CC [ a_1^{\pm 1}, \cdots, a_n^{\pm 1}, \hbar^{\pm 1}, x_1^{\pm 1}, \cdots, x_n^{\pm 1} ] / \langle \prod_{i=1}^n (x_i / a_i)^{\beta_{ji}} - 1, 1\leq j\leq d  \rangle.
$$
The ordinary $K$-theory ring $K_\TT (X)$ is the quotient of $K_\TT (\fX)$ with relations parameterized by the circuits of $X$. We will define a new ring structure, also as a quotient of $K_\TT (\fX)$, which deforms $K_\TT (X)$. On the other hand, it is essentially ``the same" as the PSZ ring.

\begin{Lemma} \label{quantum-prod}
Given $\tau, \eta \in K_\TT (\fX)$, one has
$$
\widehat{\tau \eta} (z) = \widehat\tau (z) * \widehat\eta (z).
$$
\end{Lemma}

\begin{proof}
By definition, $\widehat{\tau \eta} (z)$ is the evaluation at $q = 1$ of $\widehat V^{(\tau\eta)} (q, z)$. The descendent insertion is $(\tau \eta) |_{p_1}$, which is equivalent to $\tau |_{p_2} \cdot \eta |_{p_3}$ under the $q=1$ evaluation, where $p_2$, $p_3$ are two arbitrary points on $\PP^1$. By the degeneration formula, this is the same as $\widehat\tau (z) * \widehat\eta (z)$.
\end{proof}

The lemma implies that the following
$$
\cI := \{ \tau \in K_\TT (\fX) [[z^{\Eff(X)}]] \mid \hat\tau (z) = 0  \},
$$
is an ideal in $K_\TT (\fX)  [[z^{\Eff(X)}]] $. Furthermore, the bilinear pairing
$$
\langle \tau, \eta \rangle := \langle \hat\tau (z), \hat\eta (z) \rangle_{PSZ}
$$
vanishes for any $\tau$ or $\eta\in \cI$, and hence descends to a pairing $\langle \ , \ \rangle_{div}$ on the quotient. The map
\begin{equation} \label{isom}
K_\TT (\fX)  [[z^{\Eff(X)}]] / \cI \to (K_\TT (X) [[z^{\Eff (X)} ]], * ) , \qquad \tau \mapsto \hat\tau (z)
\end{equation}
is an isomorphism of $K_\TT (\pt) [[z^{\Eff (x)}]]$-algebras, preserving the pairings. It is therefore an isomorphism of Frobenius algebras.

\begin{Definition}
We define the quotient ring $K_\TT (\fX)  [[z^{\Eff(X)}]] / \cI$ as the \emph{divisorial quantum $K$-theory} of $X$, with \emph{quantum identity} $\bone$, and \emph{quantum pairing} $\langle \ , \ \rangle_{div}$. In particular, it is a Frobenius algebra, whose specialization at $z = 0$ recovers the ordinary $K$-theory ring $K_\TT (X)$.
\end{Definition}


\vspace{3ex}

\section{Localization computations}

\subsection{Bare vertex and integral presentation}

In this section we compute the explicit formula for bare vertex functions of the hypertoric variety $X$, by equivariant localization. Let $\bp \subset \{1, \cdots, n\}$ denote a vertex in the hyperplane arrangement $\cH$, and also the corresponding $\TT$-fixed point.

For this particular $\bp$, we choose the standard $\bp$-frame as in (\ref{V-frame-i}) and (\ref{V-frame-b}). For example, if $\bp = \{ k+1, \cdots, n \}$, then the matrices $\beta$ and $\iota$ take the form
$$
\beta = \begin{pmatrix}
-C & I
\end{pmatrix}, \qquad \iota = \begin{pmatrix}
I \\
C
\end{pmatrix}.
$$

The prequotient vector space $M = \CC^n$, as a $\TT$-representation, splits as
$$
M = \sum_{i=1}^n \CC_{e_i^*} = \sum_{i\not\in \bp} \CC_{e_i^*} + \sum_{i\in \bp} \CC_{e^*_i},
$$
whose associated tautological bundle (by (\ref{restriction-V})) can be written as
$$
\cM = \sum_{i=1}^n L_i = \sum_{i\in \cA_\bp^+ } L_i + \sum_{i\in \cA_\bp^-} \hbar^{-1} \otimes L_i + \sum_{i\in \bp} a_i \prod_{j\not\in \bp} a_j^{-C_{ij}} \cdot  \hbar^{- \sum_{j \in \bp^{c-} } C_{ij} } \otimes  \prod_{j\not\in \bp} L_j^{\otimes C_{ij}}.
$$

\begin{Lemma}
Let $f: \PP^1 \to \fX$ be a $\TT$-equivariant quasimap whose image lies in the fixed point $\bp$. Then
$$
\deg f^* L_i \geq 0, \quad \text{for any } i\in \cA_\bp^+; \qquad \deg f^* L_i \leq 0, \quad \text{for any } i\in \cA_\bp^-.
$$
\end{Lemma}

\begin{proof}
Let $(z, w) \in T^* \CC^n$ be a representative of $\bp$. We know from \cite{HP} that $z_i \neq 0$, $w_i = 0$ for $i\in \cA_\bp^+$, and $z_i = 0$, $w_i\neq 0$ for $i\in \cA_\bp^-$. Since $f$ generically maps to $\bp$, the section of $f^* L_i^{\pm 1}$ defined by $f$ is nonzero, for $i\in \cA_\bp^\pm$. The lemma follows.
\end{proof}

The lemma states that all quasimaps mapping into $\bp$ lie in the cone
$$
\prod_{i\in \cA_\bp^+} \RR_{\geq 0} \times \prod_{i\in \cA_\bp^- } \RR_{\leq 0} \subset H_2 (X, \RR).
$$

\begin{Lemma} \label{subcone}
$\prod_{i\in \cA_\bp^+ } \RR_{\geq 0} \times \prod_{i\in \cA_\bp^- } \RR_{\leq 0}$ is a subcone of $\Eff(X) \otimes \RR$. We denote it by $\Eff_\bp (X)$.
\end{Lemma}

\begin{proof}
As $\iota^\vee = (I, C^T)$, given stability condition $\theta \in (\ZZ^k)^\vee$, we can choose an equivalent lift $\widetilde\eta \in (\ZZ^n)^\vee$ as $\widetilde\eta_j = \theta_j$ for $j\not\in \bp$, and $\widetilde\eta_i = 0$ for $i\in \bp$. By definition we have $\cA_\bp^+ = \{j\not\in \bp \mid \theta_j >0 \}$, $\cA_\bp^- = \{ j \not\in \bp \mid \theta_j <0 \}$. In other words, $\theta$ lies in the cone $\Eff_\bp (X)^\vee = \prod_{i\in \cA_\bp^+ } \RR_{\geq 0} \times \prod_{i\in \cA_\bp^- } \RR_{\leq 0} \subset H^2 (X, \RR)$. On the other hand, the K\"ahler cone $\fK = \Eff (X)^\vee \otimes \RR$ is the smallest $k$-dimensional cone generated by $\iota^\vee e_i^*$, which contains $\theta$. Therefore, we have $\fK \subset \Eff_\bp (X)^\vee$. Taking the dual proves the lemma.
\end{proof}

Let $f: \PP^1 \to \fX$ be a $\TT$-equivariant quasimap whose image lies in the fixed point $\bp$.  The virtual tangent sheaf at the moduli point $[f]$ is
$$
T_\vir = H^\bullet(\cM\oplus \hbar^{-1}\cM^*) - (1+\hbar^{-1}) \Big( \sum_{i\not\in \bp} \cO \Big).
$$

By the definition (\ref{defn-cO_vir}), i.e.
$$\cO_\vir \big|_f = \frac{1}{\Lambda^\bullet(T_\vir^\vee \big|_f)} = S^\bullet(T_\vir^\vee \big|_f), \qquad \widehat\cO_\vir = q^{- \frac{1}{2} \deg \cT^{1/2}} \cO_\vir\otimes K_\vir^{1/2},$$
we see that the contribution to $\cO_\vir$ from a term
$$
a L + \hbar^{-1} a^{-1} L^{-1}, \qquad \text{with} \ \deg L = d
$$
in $T_\vir \big|_f$ is
$$
\{a\}_d := (-q^{1/2}\hbar^{-1/2})^d \frac{(\hbar a)_d}{(q a)_d} = (-q^{1/2}\hbar^{-1/2})^{-d} \frac{(a^{-1} )_{-d} }{(q \hbar^{-1} a^{-1} )_{-d} } , \qquad d\in \ZZ,
$$
where
$$
(x)_d := \frac{\phi(x)}{\phi (q^d x)}, \qquad \phi(x) := \prod_{l=0}^\infty (1-q^l x) .
$$
Note that $\phi (x)$ is convergent if we view $q$ as a complex number with $|q|<1$, and $\{a\}_d$ being rational, is well-defined for general $q$.

\begin{Proposition} \label{Prop-vertex}
	The bare vertex function is
	\begin{equation} \label{vertex}
		V^{(\tau)}(q,z)  = \sum_{\beta \in \Eff(X)} z^\beta q^{-\frac{1}{2} (\beta,  \det \cT^{1/2}) } \prod_{i=1}^n \{ x_i \}_{D_i}  \cdot \tau ( x_1 q^{D_1}, \cdots,  x_n q^{D_n} ),
	\end{equation}
	where $x_i = L_i \in K_\TT (X)$, $D_i := \deg L_i$, for $1\leq i\leq n$, $\tau = \tau (x_1, \cdots, x_n) \in K_\TT (\fX)$ is a Laurent polynomial in $n$ variables, and $z^\beta := z_1^{D_1} \cdots z_n^{D_n}$.
\end{Proposition}

\begin{proof}
Note that the expression on the RHS of (\ref{vertex}) is independent of the choice of basis for $\ZZ^k$. The proposition follows from direct computations of the vertex function, restricted to each fixed point.
\end{proof}

\begin{Remark}
Note that by our definition of K\"ahler parameters in Section \ref{sec-para}, we have $z^\beta = z_1^{D_1} \cdots z_n^{D_n} = \prod_{j\not\in \bp} \zeta_j (\bp)^{D_j}$. Moreover, the $z^\beta$ term of the restriction $V^{(\tau)}(q,z) \big|_\bp$ of (\ref{vertex}) to the fixed point $\bp$ vanishes unless $\beta \in \Eff_\bp (X)$.
\end{Remark}

Note that different choices of the global  polarization $T^{1/2}_X$, and hence the associated $\cT^{1/2}$, result in different vertex functions. But they are related by $q$-shifts of the K\"ahler parameters. As observed in \cite{AOelliptic}, with some extra factors, the bare vertex function can be written in the form of an integral over the Chern roots, along some appropriately chosen contours.

\begin{Definition} \label{loc-pol}
\begin{enumerate}[1)]
	
\setlength{\parskip}{1ex}

\item Choose the \emph{global polarization} as
\begin{equation} \label{Polar-X}
T^{1/2}_X = \sum_{i=1}^n  L_i -  \cO^{\oplus k} = \sum_{i=1}^n x_i - k \ \in \ K_\TT (X).
\end{equation}
Define the shifted K\"ahler parameters
$$
z_\sharp^\beta := z^\beta \cdot (-\hbar^{ \frac{1}{2} \beta \cdot \det T^{1/2}_X }), \qquad \text{i.e.} \qquad  z_{\sharp, i} := z_i \cdot (-\hbar^{-1/2}), \qquad 1\leq i\leq n.
$$

\item Define the \emph{localized polarization} as
\begin{equation} \label{polarization}
\Pol_\bp (x) := \sum_{i\in \cA_\bp^+} \hbar^{-1} x_i^{-1} + \sum_{i\in \cA_\bp^- \cup \bp} x_i - k \hbar^{-1} \ \in \ K_\TT (X).
\end{equation}
In particular, the restriction $\Pol_\bp (x) |_\bp = \sum_{i\in \bp} x_i |_\bp = T_X^{1/2} |_\bp$ is well-defined. Define another version of shifted K\"ahler parameters
$$
z_\epsilon^\beta := z_\sharp^\beta \cdot (-q^{1/2} \hbar^{-1/2})^{\frac{1}{2} (-1+\beta \cdot \det \Pol_\bp (x) )}, \qquad \text{i.e.} \qquad  z_{\epsilon, i} (\bp) := z_{\sharp, i} \cdot (q\hbar^{-1})^{-\epsilon (i)},
$$
where $\epsilon (i) = 1$ if $i\in \cA_\bp^+$, and $\epsilon (i) = 0$ if $i\in \cA_\bp^- \cup \bp$.

\item Define the \emph{modified bare vertex functions} as
$$
\widetilde V^{(\tau)} (q,z) |_\bp := V^{(\tau)}(q,z) \big|_\bp \cdot e^{\sum_{i=1}^n \frac{\ln z_{\epsilon,i} (\bp) \ln x_i |_\bp }{\ln q} } \cdot \Phi ((q-\hbar) T_X^{1/2} |_\bp),
$$
where $\Phi$ be the multiplicative function determined by $\Phi(\sum_i x_i) := \prod_i \phi(x_i)$.
\end{enumerate}
\end{Definition}

\begin{Remark}
Recall that the bare vertex function $V^{(\tau)}$, lives in $K_\TT (X)$, and depends only on the choice of the global polarization $T_X^{1/2}$. However, the modified bare vertex function $\widetilde V^{(\tau)} |_\bp$ depends also on the localized polarization $\Pol_\bp$, which is only defined for each fixed point $\bp$, and does not necessarily lift to a global class in $K_\TT (X)$. In general, one is allowed to choose a different localized polarization such as $\Pol_\bp (x) = \sum_{i\in \cA_\bp^+ \cup \bp'} \hbar^{-1} x_i^{-1} + \sum_{i\in \cA_\bp^- \cup \bp''} x_i - k \hbar^{-1}$, and the corresponding $z_\epsilon$ such as $\epsilon (i) = 1$ if $i\in \cA_\bp^+ \cup \bp'$, and $\epsilon (i) = 0$ if $i\in \cA_\bp^- \cup \bp''$. We will see such a different choice on the mirror side.
\end{Remark}

Direct computation yields the following.

\begin{Proposition}
$$
\widetilde V^{(\tau)} (q,z) \big|_\bp = \frac{1}{(2\pi i)^k} \int_{q \gamma(\bp)} \frac{d\ln x_1 \wedge \cdots d\ln x_n}{\bigwedge_{m=1}^d \Big( \sum_{i=1}^n \beta_{mi} d\ln x_i \Big) } \cdot e^{\sum_{i=1}^n \frac{\ln z_{\epsilon,i} (\bp) \ln x_i }{\ln q} } \cdot \Phi'((q - \hbar) \mathsf{Pol}_\bp (x)) \cdot \tau( x_1 , \cdots, x_n),
$$
	where $q\gamma (V)$ is a noncompact real $k$-cycle in $(\CC^*)^k = \left\{ \prod_{i=1}^n (x_i / a_i)^{\beta_{ji}} = 1 \right\} \subset (\CC^*)^n$, enclosing the following $q$-shifts of poles:
	\begin{equation} \label{poles}
	x_i = \left\{ \begin{aligned}
	& q^{d_i}, \qquad &&  d_i \geq 0, \ i \in \cA_\bp^+  \\
	& \hbar^{-1} q^{d_i}, \qquad && d_i \leq 0, \ i\in \cA_\bp^- .
	\end{aligned} \right.
	\end{equation}
\end{Proposition}

\begin{Remark} \label{limit-point}
If we use $z_{\epsilon, i}$, $i\not\in \bp$ as the coordinates, then up to an exponential factor, the integral is convergent in the region described as follows:
$$
|z^\beta| \ll 1,
$$
for any $\beta \in \Eff (X) - \{0\}$. In particular, we have $|\zeta_i (\bp)|\ll 1$ for $i\in \cA_\bp^+$ and $|\zeta_i (\bp)| \gg 1$ for $i\in \cA_\bp^- $, and there is a limit point
$$
\zeta_i (\bp) \to 0, \qquad  i \in \cA_\bp^+; \qquad \zeta_i(\bp) \to \infty , \qquad i\in \cA_\bp^-.
$$
We will denote by $\zeta (\bp) \xrightarrow{\theta} 0$ the process of taking the limit of $\zeta$, or equivalently $z$, in the manner as above, in the well-defined $\theta$-dependent region. In particular, under this limit, $z^\beta \to 0$ or $\infty$ if and only if $\beta \cdot \theta >0$ or $<0$.
\end{Remark}

\subsection{Asymptotics}

In this subsection, we would like to prove a rigidification result which particularly holds for hypertoric varieties, which follows from an estimate of the behavior of the are vertex function as $q\to 0$ or $\infty$. The result does not hold in general for non-abelian holomorphic symplectic quotients.

\begin{Theorem} \label{rig}
	Let $X$ be a hypertoric variety. The capped vertex function $\widehat V^{(\bone)} (q,z)$ is independent of $q$, and hence equal to the PSZ quantum identity class $\widehat\bone (z)$. Moreover, one has the limit
	$$
	V^{(\bone)} (q , z) \to \left\{ \begin{aligned}
	& \widehat \bone (z) , && \qquad q \to 0 \\
	& G(q, z) \big|_{q=1}^{-1} \cdot \widehat \bone (z) , && \qquad q \to \infty.
	\end{aligned} \right.
	$$
\end{Theorem}

Recall that by definition $\widehat \bone (z) = \lim_{q\to 1} \widehat V^{(\bone)} (q,z)$, and by Proposition \ref{capping-eqn} the capped and bare vertices are related by the capping operator
$$
\hat V^{(\bone)}(q,z) = \Psi(q,z) \cdot V^{(\bone)}(q,z).
$$
Let's study the $q\to 1$ asymptotic behavior of this equation.

\begin{Lemma} \label{capping-limit}
	The capping operator $\Psi (q,z)$ satisfies
	$$
	\Psi (q, z) \to \left\{ \begin{aligned}
	& \Id, && \qquad q \to 0 \\
	& G(q,z)|_{q=1} , && \qquad q \to \infty
	\end{aligned} \right.
	$$
\end{Lemma}

\begin{proof}
	This is Lemma 7.1.11 in \cite{Oko}.
\end{proof}

\begin{Lemma} \label{bounded}
	The bare vertex function $V^{(\bone)} (q,z)$ admits finite limits when $q \to 0$ or $\infty$.
\end{Lemma}

\begin{proof}
	We have the explicit formula
	$$
	\left. V^{(\bone)} (q, z) \right|_\bp = \sum_{\beta \in \Eff (X)} z^\beta q^{- \frac{1}{2} \sum_{i=1}^n D_i } \prod_{i=1}^n \{ x_i |_\bp \}_{D_i}.
	$$
	By definition, we have for any $x$ and $D\in \ZZ$,
	$$
	\{ x \}_D \sim \left\{ \begin{aligned}
	& \mathrm{const} \cdot q^{|D|/2}, && \qquad q \to 0 \\
	& \mathrm{const} \cdot q^{-|D| / 2} , && \qquad q \to \infty.
	\end{aligned} \right.
	$$
	We see that the factor $q^{ - \frac{1}{2} \sum_{i=1}^n D_i } $ is completely controlled by the term $q^{\pm \sum_{i=1}^n  |D_i| / 2}$ and hence $V^{(\bone)} (q, z)$ is bounded as $q \to 0$ or $\infty$.
\end{proof}

\begin{proof}[Proof of Theorem \ref{rig}]
	As a class in the \emph{non-localized} $K$-theory ring $K_{\TT \times \CC_q^*} (X) [[ z^{\Eff (X)} ]]$, each $z^\beta$-term of the capped vertex $\widehat V^{(\bone)} (q, z)$ is a Laurent polynomial in $q$. Moreover, by the previous two lemmas, it admits no poles at $q = 0$ and $\infty$, which implies that it is actually constant in $q$. Therefore,
	$$
	\lim_{q \to 1} \widehat V^{(\bone)} (q , z) = \lim_{q\to 0} \widehat V^{(\bone)} (q , z) .
	$$
	The last statement follows from Lemma \ref{capping-limit}.
\end{proof}

For some special targets $X$, one can explicitly compute the limit of the vertex function, and hence the PSZ quantum identity class. Consider the following two assumptions:

$\mathrm{(A+)}$ for any circuit $\beta\neq 0$, there exists some $i$, such that $D_i >0$;

$\mathrm{(A-)}$ for any circuit $\beta\neq 0$, there exists some $i$, such that $D_i <0$.

For example, $T^*\PP^n$ satisfies $\mathrm{(A+)}$ but not $\mathrm{(A-)}$; $\cA_n$ satisfies both $\mathrm{(A+)}$ and $\mathrm{(A-)}$.

\begin{Corollary}
\begin{enumerate}[1)]

\item If $X$ satisfies $\mathrm{(A+)}$, then $\widehat \bone (z) = G(q,z) \big|_{q=1} \cdot \bone$.

\item If $X$ satiesfies $\mathrm{(A-)}$, then $\widehat \bone (z) = \bone$. 

\end{enumerate}
\end{Corollary}

\subsection{$q$-difference equations}

For any function $f(a, z)$ depending on the (redundant) equivariant parameters $a_i$ and K\"ahler parameters $z_i$, $1\leq i\leq n$, consider the following $q$-shift operators:
\ben
(A_i f) (a_1, \cdots, a_i, \cdots a_n, z ) &:=& f (a_1, \cdots, q a_i, \cdots a_n, z ) \\
(Z_i f) (a, z_1, \cdots, z_i, \cdots, z_n ) &:=& f (a, z_{1}, \cdots, q z_{i}, \cdots, z_{n} ).
\een
The effect of the inverse operators $A_i^{-1}$ and $Z_i^{-1}$ are to shift the variables by $q^{-1}$.

We would like to apply those shift operators to the vertex function
\ben
\widetilde V^{(\bone)} (q,z) \big|_\bp
&=& \frac{1}{(2\pi i)^k} \frac{\phi(q)^k}{\phi(q \hbar^{-1})^k } \int_{q \gamma (\bp)} \frac{d\ln x_1 \wedge \cdots d\ln x_n}{\bigwedge_{m=1}^d \Big( \sum_{i=1}^n \beta_{mi} d\ln x_i \Big) } \cdot e^{\sum_{i=1}^n \frac{\ln z_{\epsilon,i} (\bp) \ln x_i }{\ln q} } \\
&& \cdot \prod_{i\in \cA_\bp^+} \frac{ \phi (q\hbar^{-1} x_i^{-1} ) }{ \phi (x_i^{-1} ) } \prod_{i\in \cA_\bp^- \cup \bp } \frac{\phi (q x_i)}{ \phi (\hbar x_i )}  \\
&=:& \int_{q \gamma (\bp )} e^{W(x)} .
\een
Recall in Definition \ref{loc-pol} 2) that $z_{\sharp, i} = z_i \cdot (-\hbar^{-1/2}) = z_{\epsilon, i} (\bp) \cdot (q\hbar^{-1})^{\epsilon (i)}$, where $\epsilon (i) = 1$ or $0$ for $i\in \cA_\bp^+$ or $\cA^- \cup \bp$ respectively. Note that unlike $z_\epsilon$, $z_\sharp$ is independent of $\bp$.

\begin{Lemma} \label{q-diff-operator}
The $q$-shift operators act as:
$$
Z_i \widetilde V^{(\bone)} (q,z) \big|_\bp =
 \int_{q \gamma (\bp)} x_i \cdot e^{W(x)} , \qquad
A_i^{-1} \widetilde V^{(\bone)} (q,z) \big|_\bp =
 (q\hbar^{-1}) z_{\sharp, i}^{-1} \int_{q \gamma (\bp)} \Big( \frac{1 - x_i^{-1} }{1 - q \hbar^{-1} x_i^{-1} } \Big) \cdot e^{W(x)} ,
 $$
 $$
 A_i \widetilde V^{(\bone)} (q,z) \big|_\bp =
 z_{\sharp, i}  \int_{q \gamma (\bp)} \Big( \frac{1 - \hbar x_i }{1 - q x_i } \Big) \cdot e^{W(x)}  .
$$
In particular, all these $q$-shift operators commute with each other.
\end{Lemma}

\begin{proof}
The only nontrivial actions are those of $A_i$, $i\not\in \bp$, since they not only act on the integrand, but also shift the contour $q \gamma(\bp)$. Let's consider the case $i\in \cA_\bp^+$; the case $i\in \cA_\bp^-$  is similar.

Recall that the contour $q \gamma (\bp)$ encloses poles of the integrand, described as in (\ref{poles}). The operator $A_i^{-1}$ for $i\in \cA_\bp^+$ shifts to a contour $A_i^{-1} (q \gamma (\bp) )$. The poles enclosed by $A_i^{-1} (q \gamma (\bp))$ sasitsfy the same conditions as in (\ref{poles}), except for $i$: $x_i = q^{1 + d_i}$, $d_i \geq 0$. Note that the points with $x_i = q$ are no longer poles of the integrand. Thus it does no harm to change the contour back to $q \gamma (\bp)$. We then have
$$
A_i^{-1} \widetilde V^{(\bone)} (q,z) \big|_\bp = z_{\epsilon, i} (\bp)^{-1} \int_{q \gamma (\bp)} \Big( \frac{1 - x_i^{-1}}{1 - q \hbar^{-1} x_i^{-1}} \Big) \cdot e^{W(x)} .
$$
The action of $A_i$ follows directly.
\end{proof}

\begin{Theorem} \label{q-diff-eqn}
\begin{enumerate}[1)]

\item The modified vertex function $\widetilde V^{(\bone)} (q,z) \big|_\bp$ is annihilated by the following $q$-difference operators:
\begin{equation} \label{q-diff-Z}
\prod_{i\in S^+} ( 1 - Z_i ) \prod_{i\in S^-} ( 1 - \hbar  Z_i )  -  z_\sharp^\beta \prod_{i\in S^+} ( 1 - \hbar  Z_i ) \prod_{i\in S^-} ( 1 - Z_i ) , \qquad S = S^+ \sqcup S^- : \text{circuit},
\end{equation}
where $z_{\sharp, i}:= z_i (-\hbar^{-1/2})$, $\beta$ is the curve class corresponding to $S$, and $z_\sharp^\beta := \prod_{i\in S^+} z_{\sharp, i} \prod_{i\in S^-} z_{\sharp, i}^{-1}$.

\item The modified vertex function $\widetilde V^{(\bone)} (q,z) \big|_\bp \cdot e^{-\sum_{i=1}^n \frac{\ln z_{\sharp, i} \ln a_i}{\ln q} }$ is annihilated by the following $q$-difference operators:
\begin{equation} \label{q-diff-A}
\prod_{i\in R^+} (1- A_i) \prod_{i\in R^-} (1  - q\hbar^{-1} A_i ) - (\hbar a)^\alpha \prod_{i\in R^+} (1  - q\hbar^{-1} A_i ) \prod_{i\in R^-} (1  -  A_i ) , \qquad R = R^+ \sqcup R^- : \text{cocircuit},
\end{equation}
where $\alpha$ is the root corresponding to $R$, and $(\hbar a)^\alpha := \prod_{i\in R^+} (\hbar a_i) \prod_{i\in R^-} (\hbar a_i)^{-1}$.
\end{enumerate}

\end{Theorem}

\begin{proof}
Let $\cM_1$ be the left $\cD_q$-module generated by $\widetilde V \big|_\bp$. We know that $\widetilde V \big|_\bp$ only depends on the non-redundant equivariant parameters $\alpha_i (\bp) := a_i \prod_{j\not\in \bp} a_j^{-C_{ij}}$, $i\in \bp$. So for each circuit $S = S^+ \sqcup S^-$, the operator $\prod_{i\in S^+} A_i \prod_{i\in S^-} A_i^{-1}$ acts as identity in $\cM_1$. In other words, on $\cM_1$ we have
$$
\prod_{i\in S^+} A_i  = \prod_{i\in S^-} A_i , \qquad \forall j\not\in V.
$$
On the other hand, for any $i$, the relation between $Z_i$ and $A_i$'s (in $\cM_1$) is $( 1 - q Z_i ) z_{\sharp, i}^{-1} A_i  =  1 - \hbar Z_i$, or equivalently (using $q Z_i z_{\sharp, i}^{-1} = z_{\sharp, i}^{-1} Z$)
\begin{equation} \label{ZA}
z_{\sharp, i}^{-1} (1 - Z_i) A_i  =  1 - \hbar Z_i  ,
\end{equation}
Therefore, we have in $\cM_1$,
\ben
 \prod_{i\in S^+} ( 1 - Z_i ) \prod_{i\in S^-} ( 1 - \hbar Z_i )
&=&  \prod_{i\in S^-} z_{\sharp, i}^{-1}  \prod_{i\in S^+ \sqcup S^-} ( 1 - Z_i )  \prod_{i\in S^-} A_i   \\
&=& \prod_{i\in S^-} z_{\sharp ,i}^{-1} \prod_{i\in S^+ \sqcup S^-} ( 1 - Z_i )   \prod_{i\in S^+} A_i   \\
&=& z_\sharp^\beta \prod_{i\in S^+} ( 1 - \hbar  Z_i )  \prod_{i\in S^-} ( 1 - Z_i )  ,
\een
We obtain 1).

For 2), let $\cM_2$ be the $\cD_q$-module generated by $\widetilde V^{(\bone)} (q,z) \big|_\bp \cdot e^{-\sum_{i=1}^n \frac{\ln z_{\sharp, i} \ln a_i}{\ln q} }$. There is an isomorphism $\cM_1 \cong \cM_2$, sending $Z_i \mapsto a_i Z_i$, $A_i \mapsto z_{\sharp, i} A_i$.   One can check that  
$$
\sum_{i=1}^n \ln z_{\epsilon, i} (\bp) \ln x_i |_\bp - \sum_{i=1}^n \ln z_{\epsilon, i} (\bp) \ln a_i  = -\sum_{j\in \cA_\bp^+ } \ln \zeta_{\epsilon, j} (\bp) \ln a_j - \sum_{j\in \cA_\bp^- } \ln \zeta_{\epsilon, j} (\bp) \ln (\hbar a_j), 
$$ 
and hence for each cocircuit $R$, the operators $\prod_{i\in R^+} Z_i \prod_{i\in R^-} Z_i^{-1}$ act as identity in $\cM_2$. On the other hand, by (\ref{ZA}), we have $1- A_i = (\hbar a_i) (1 - q\hbar^{-1} A_i ) Z_i$. By similar arguments, we obtain 2).
\end{proof}

The vertex functions can be uniquely characterized as solutions of the  $q$-difference equations (with respect to either K\'ahler or equivariant parameters), with prescribed asymptotes.

\begin{Lemma} \label{uniqueness}
Given $\bp \in X^\TT$, the function $\widetilde V^{(\bone)} (q, z) \big|_\bp$ is the unique solution of the $q$-difference system (\ref{q-diff-Z}), with asymptotic behavior
$$
\widetilde V^{(\bone)} (q, z) \big|_\bp \sim e^{\sum_{i=1}^n \frac{\ln z_{\epsilon, i} (\bp) \ln x_i |_\bp}{\ln q} }   \cdot \prod_{i\in \bp} \frac{\phi (q x_i |_\bp )}{\phi ( \hbar x_i |_\bp )} \cdot (1 + o(\zeta (\bp) ) )  ,
$$
as $\zeta \xrightarrow{\theta} 0$, where the limit is explained in Remark \ref{limit-point}.
\end{Lemma}

\begin{proof}

By the standard approach, the higher order $q$-difference system (\ref{q-diff-Z}) can be written as a first-order holonomic $q$-difference system, of rank $\rk K_\TT (X)$. The existence and uniqueness of the solution follows from discussions in \cite{Aom, FR, AOelliptic}.
\end{proof}

\subsection{Relations for PSZ and divisorial quantum $K$-theory}

The PSZ quantum $K$-theory ring we introduced in Section \ref{QK} can be explicitly determined by the $q$-difference equations \ref{q-diff-Z} with respect to K\"ahler parameters.

\begin{Theorem} \label{PSZ-relations}
We have the following presentations of ring structures (which are equivalent to each other):

\begin{enumerate}[1)]

\item The PSZ quantum $K$-theory ring of $X$ is generated by the quantum tautological line bundles $\widehat L_i (z)$, $1\leq i\leq n$, up to the relations
$$
\prod_{i\in S^+} ( 1 - \widehat L_i (z) ) * \prod_{i\in S^-} ( 1 - \hbar  \widehat L_i (z) )  -  z_\sharp^\beta \prod_{i\in S^+} ( 1 - \hbar  \widehat L_i (z) ) * \prod_{i\in S^-} ( 1 - \widehat L_i (z) ) , \qquad S = S^+ \sqcup S^- : \text{circuit},
$$
where $z_{\sharp, i}:= z_i (-\hbar^{-1/2})$, $\beta$ is the curve class corresponding to $S$, $z_\sharp^\beta := \prod_{i\in S^+} z_{\sharp, i} \prod_{i\in S^-} z_{\sharp, i}^{-1}$, and all products $\prod$ are quantum products $*$.

\item The divisorial quantum $K$-theory ring of $X$ is generated by the line bundles $L_i$, $1\leq i\leq n$, up to the relations
$$
\prod_{i\in S^+} ( 1 -  L_i )  \prod_{i\in S^-} ( 1 - \hbar L_i )  -  z_\sharp^\beta \prod_{i\in S^+} ( 1 - \hbar L_i )  \prod_{i\in S^-} ( 1 -  L_i ) , \qquad S = S^+ \sqcup S^- : \text{circuit},
$$
where $z_{\sharp, i}:= z_i (-\hbar^{-1/2})$, $\beta$ is the curve class corresponding to $S$, and $z_\sharp^\beta := \prod_{i\in S^+} z_{\sharp, i} \prod_{i\in S^-} z_{\sharp, i}^{-1}$.

\end{enumerate}

\end{Theorem}

\begin{proof}
Recall that in Lemma \ref{q-diff-operator}, the action of the $q$-difference operator $Z_i$ on the bare vertex function is
$$
Z_i \widetilde V^{(\bone)} (q,z) \big|_\bp =
\int_{q \gamma (\bp)} x_i \cdot e^{W(x)} = \widetilde V^{(x_i)} (q,z) \big|_\bp.
$$
In general, let $\tau (x_1, \cdots, x_n)$ be a Laurent polynomial in $x_1, \cdots, x_n$, with coefficients in $K_{\bT \times \CC_q^*} (\pt)$. We have
$$
\tau (Z_1, \cdots, Z_n) \widetilde V^{(\bone)} (q,z) \big|_\bp  = \widetilde V^{(\tau)} (q,z) \big|_\bp.
$$
Now suppose that for some $\tau$, and for any $\bp \in X^\bT$, we have $\tau (Z_1, \cdots, Z_n) \widetilde V^{(\bone)} (q,z) \big|_\bp = \widetilde V^{(\tau)} (q,z) \big|_\bp = 0$. It follows that $V^{(\tau)} (q,z) = 0$, and hence $\widehat V^{(\tau)}(q,z) = \Psi(q,z) \cdot V^{(\tau)}(q,z) = 0$ by Proposition \ref{capping-eqn}. Evaluating at $q=1$, we have
$$
\tau (\widehat L_1 (z), \cdots, \widehat L_n (z) ) = \widehat\tau (z) = 0,
$$
by Lemma \ref{quantum-prod}, where for  products in $\tau$ on the LHS we take the quantum product $*$. The theorem then follows from equation (\ref{q-diff-Z}).
\end{proof}

\begin{Remark}
The result can also be obtained following the approach in \cite{PSZ}. The quantum $K$-theory relations here can be interpreted as Bethe-ansatz equations.
\end{Remark}


\vspace{3ex}

\section{3d mirror symmetry for hypertorics}

\subsection{Abelian mirror construction}

To construct the dual of the hypertoric variety $X$, consider the dual of the sequence (\ref{knd-seq}):
\begin{equation} \label{knd-seq-dual}
\xymatrix{
	0 \ar[r] & \ZZ^d \ar[r]^{\iota'} & \ZZ^n \ar[r]^{\beta'} & \ZZ^k \ar[r] & 0,
}
\end{equation}
where $\iota' = \beta^\vee$, and $\beta' = \iota^\vee$.

With given stability parameter $\widetilde\theta$ and chamber parameter $\widetilde\sigma$ of $X$, we choose the stability and chamber parameter for the mirror as
\begin{equation} \label{mirror-theta-sigma}
\widetilde\theta' = - \widetilde\sigma, \qquad \widetilde\sigma' = - \widetilde\theta.
\end{equation}

We now view the dual sequence (\ref{knd-seq-dual}) as the defining sequence for a new hypertoric variety, denote by $X'$.  We define $X'$ as the \emph{3d mirror} of the hypertoric variety $X$. Denote by $\bK'$, $\bT'$, $\bA'$ the corresponding tori, and by $\cH'$ the hyperplane arrangment. Let $L'_i$ be the tautological line bundle defined by the $i$-th standard basis vector in $\ZZ^n$, and let $x'_i$ be its $K$-theory class. Let $a'_i$, $z'_i$, $s'_i$ be the equivariant parameters, K\"ahler parameters and Chern roots for $X'$ respectively.

Let $\bp \subset \{1, \cdots, n\}$ be the subset corresponding to a vertex in $\cH$. Take $\bp' := \cA_\bp =  \{ 1, \cdots, n \} \backslash \bp$.

\begin{Lemma}
$\bp'$ is a vertex in the dual hyperplane arrangement $\cH'$.
\end{Lemma}

\begin{proof}
Recall that the hyperplane $H_i$ has normal vector $a_i = \beta (e_i) \in \RR^d$. The fact that $\bp = \bigcap_{i\in \bp} H_i \neq \emptyset$ is equivalent to the linear independence of the vectors $\{ a_i \mid i\in \bp \}$. In particular, in the standard $\bp$-frame, this is equivalent to the fact that the matrix $\beta$ is of the form $(-C, I)$, where the identity submatrix $I$ is for the columns $\{i \in \bp\}$. Also, the matrix $\iota$ is of the form $\begin{pmatrix}
I \\
C
\end{pmatrix}$, where $I$ is for the rows $\{ j \not\in \bp \}$.

Now we look at the dual picture. The matrix $\beta' = \iota^\vee$ is of the form $(I, C^T)$, whose columns indexed by $\{j \in \bp' \} = \{j\not\in \bp \}$ are linearly independent. Therefore, $\bp'$ is a vertex in $\cH'$.
\end{proof}

As a result, we have the following natural bijection between the fixed point sets:
\begin{equation} \label{bj-fixed}
\textsf{bj}: X^\TT \xrightarrow{\sim} (X')^{\TT'}, \qquad \bp \mapsto \bp'.
\end{equation}

For a given $\bp$, if we choose the standard $\bp$-frame (\ref{V-frame-i}) (\ref{V-frame-b}), the dual variety $X'$ will also be in the standard $\bp'$-frame. More precisely, this means
$$
\iota'_{li} (\bp') = \delta_{li}, \qquad \iota'_{ji} (\bp') = - C_{ij} (\bp), \qquad l, i \not\in \bp', \ j \in \bp',
$$
$$
\beta'_{ji} (\bp') = C_{ij} (\bp), \qquad \beta'_{jm} (\bp') = \delta_{jm}, \qquad i \not\in \bp', \ j, m \in \bp'.
$$

Recall that the choice of $\widetilde\theta$ determines the splitting $\cA_\bp = \cA_\bp^+ \sqcup \cA_\bp^-$, and the choice of $\widetilde\sigma$ determines the splitting $\bp = \bp^+ \sqcup \bp^-$. The stability and chamber parameters on the mirror side, specified by (\ref{mirror-theta-sigma}), also determines the similar splittings of $\cA_{\bp'} = \bp$ and $\bp' = \cA_\bp$.

\begin{Lemma} \label{V-V'}
We have
$$
\cA_{\bp'}^{\pm} = \bp^\mp, \qquad (\bp')^\pm = \cA_\bp^\mp .
$$
\end{Lemma}

\begin{proof}
The vertex $\bp = (v_i)_{i\in \bp} \in \RR^d$ is the unique solution to the equations $\langle \bp, \beta (e_i) \rangle = - \langle \widetilde\theta, e_i \rangle$, $i\in \bp$. In particular, in the standard $\bp$-frames, the unique solution is $v_i = -\widetilde\theta_i$, $i\in \bp$.

Recall that by definition, $\cA_\bp^+ \subset \bp'$ consists of those $j\in \bp'$ for which $\bp \in H_j^+$. This means that $\langle \beta^* (\bp), e_j \rangle = \langle \bp, \beta (e_j) \rangle > - \widetilde\theta_j$. Therefore, in the standard $\bp$-frame, the inequalities become
$$
\sum_{i\in \bp} (-C_{ij}) v_i =  \sum_{i\in \bp} C_{ij} \widetilde\theta_i > - \widetilde\theta_j.
$$
Applying the mirror construction $\widetilde\sigma' = - \widetilde\theta$, we have
$$
\langle a'_j \prod_{i\in \bp} (a'_i)^{C_{ij}} , \widetilde\sigma' \rangle < 0,
$$
which is exactly the condition (\ref{split-V}) characterizing $(\bp')^-$. Hence we've proved that $\cA_\bp^+ = (\bp')^-$, and the others are similar.
\end{proof}

The restriction formula of tautological line bundles to a fixed point $\bp'$ (under the standard $\bp'$-frame) is
\be \label{restriction-V'}
x'_j |_{\bp'}  := \left.  L'_j \right|_{\bp'} &=& \left\{ \begin{aligned}
	& 1 , \qquad && j \in \cA_{\bp'}^+ \\
	& \hbar^{-1}, \qquad && j \in \cA_{\bp'}^- \\
	& a'_j \prod_{i\not\in \bp'} (a'_j)^{-C'_{ji}} \cdot  \hbar^{- \sum_{j \in \cA_{\bp'}^-} C'_{ji} } .  \qquad && j \in \bp',
\end{aligned} \right.    \nonumber  \\
&=& \left\{ \begin{aligned}
	& 1 , \qquad && j \in \bp^- \\
	& \hbar^{-1}, \qquad && j \in \bp^+ \\
	& a'_j \prod_{i\in \bp} (a'_j)^{C_{ij}} \cdot  \hbar^{ \sum_{i \in \bp^+} C_{ij} } .  \qquad && j \not\in \bp,
\end{aligned} \right.
\ee

\subsection{Duality of wall-and-chamber structures}

Recall that in Section \ref{section-eff}, the space $\RR^k$ of stability conditions of $X$ admits a wall-and-chamber structure. Each circuit $S$ defines a wall $P_S$, and the stability condition $\widetilde\theta$, or essentially its image $\theta = \iota^\vee \widetilde\theta$, specifies the K\"ahler cone $\fK$ in the complement of all walls. The boundary walls of $\fK$ form a basis of the effective cone $\Eff (X)$.

Similarly in Section \ref{section-root}, the space $\RR^d$ also admits a wall-and-chamber structure, where each wall $W_\alpha$ is indexed by a cocircuit, or equivalently a root $\alpha$. A generic choice of the chamber parameter $\widetilde\sigma$, or essentially its image $\sigma = \beta \widetilde\sigma$, singles out a chamber $\fC$, whose boundary walls give the positive simple roots.

Now let's consider the dual hypertoric variety $X'$, with the choice of $\widetilde\theta'$ and $\widetilde\sigma'$ specified in (\ref{mirror-theta-sigma}). We denote $\fK'$ and $\fC'$ the chambers determined by the choice. The following duality between the wall-and-chamber structures follows from the definition of circuits and cocircuits.

\begin{Proposition}
	
	The wall-and-chamber structures of $X$ and $X'$ are dual to each other. More precisely, this means that
	
	\begin{itemize}
		
		\item There is a bijection between the circuits of $X$ and the (negative) cocircuits of $X'$, and vice versa.
		
		\item $\fK = \fC'$, $\fC = \fK'$. In particular, indecomposable effective curves of $X$ can be identified with negative simple roots of $X'$, and vice versa.
		
	\end{itemize}
	
\end{Proposition}

\subsection{Duality interface and elliptic stable envelopes}

Recall the definition of equivariant and K\"ahler parameters in Section \ref{sec-para}. For a hypertoric variety $X$, equivariant parameters are defined as coordinates on the quotient torus $\bA = (\CC^*)^d$. On the other hand, by our mirror construction, $\bA^\vee = \bK'$ for $X'$. In other words, there is a canonical isomorphism between the equivariant parameters of $X$ and the K\"ahler parameters of $X'$. We also include the $\CC^*_\hbar$ factor and denote it by
\begin{equation} \label{id-para}
\kappa_{\Stab}: \  \bK^\vee \times \bA \times \CC^*_\hbar \xrightarrow{\sim} \bA' \times (\bK')^\vee \times \CC^*_\hbar, \qquad ( z_i , a_i, \hbar) \mapsto ( a'_i, z'_i,  \hbar^{-1}).
\end{equation}
We see that $\kappa_{\Stab}$ is actually induced by the identity map of $\bT^\vee \times \bT = \bT' \times (\bT')^\vee$, together with the natural interpretation of $\bK$, $\bA$, $\bK^\vee$ and $\bA^\vee$ as sub-tori or quotient tori of $\bT$ or $\bT^\vee$. In particular, in the standard $\bp$-frame, the isomorphism $\kappa$ can also be written explicitly as
$$
\zeta_j (\bp) \mapsto \alpha'_j (\bp'), \qquad \alpha_i (\bp) \mapsto \zeta'_i (\bp'), \qquad \hbar\mapsto \hbar^{-1},
$$
where $j\not\in \bp$, $i\in \bp$.

By the explicit formula of elliptic stable envelopes, we can write down the formula on the dual side, which by Lemma \ref{V-V'} is
\ben
\Stab'_{\sigma'} (\bp') &=& \prod_{j \in (\bp')^+} \theta ( \hbar x'_j ) \prod_{j \in (\bp')^-} \theta ( x'_j ) \prod_{i \in \cA_{\bp'}^+} \dfrac{\theta \Big( x'_i \zeta'_i (\bp') \hbar^{- \sum_{j\in (\bp')^+} C'_{ji} } \Big) }{\theta \Big( \zeta'_i (\bp') \hbar^{- \sum_{j\in (\bp')^+} C'_{ji} }  \Big) } \prod_{i \in \cA_{\bp'}^-} \dfrac{\theta \Big( x'_i \zeta'_i (\bp')  \hbar^{- \sum_{j\in (\bp')^+} C'_{ji} } \Big) }{\theta \Big( \hbar^{-1} \zeta'_i (\bp') \hbar^{- \sum_{j\in (\bp')^+} C'_{ji} } \Big) } \\
&=& \prod_{j \in \cA_\bp^-} \theta ( \hbar x'_j ) \prod_{j \in \cA_\bp^+} \theta ( x'_j ) \prod_{i \in \bp^-} \dfrac{\theta \Big( x'_i \zeta'_i (\bp') \hbar^{ \sum_{j\in \cA_\bp^-} C_{ij} } \Big) }{\theta \Big( \zeta'_i (\bp') \hbar^{ \sum_{j\in \cA_\bp^-} C_{ij} }  \Big) } \prod_{i \in \bp^+} \dfrac{\theta \Big( x'_i \zeta'_i (\bp') \hbar^{ \sum_{j\in \cA_\bp^-} C_{ij} } \Big) }{\theta \Big( \hbar^{-1} \zeta'_i(\bp') \hbar^{ \sum_{j\in \cA_\bp^-} C_{ij} } \Big) } .
\een
The diagonal elements are
$$
\left. \Stab'_{\sigma'} (\bp') \right|_{\bp'} = \prod_{j \in \cA_\bp^+} \theta ( x'_j |_{\bp'} ) \prod_{j \in \cA_\bp^-} \theta ( \hbar x'_j |_{\bp'} )
= (-1)^{|\cA_\bp^-|} \Theta (N_{\bp'}^{'-}) ,
$$
whose image under $\kappa^{-1}$ are exactly the denominators in $\Stab_\sigma (\bp)$. In other words, we can consider the following normalized version of elliptic stable envelopes
\ben
\bStab_\sigma (\bp) &:=& \Stab_\sigma (\bp) \cdot \left. \Stab'_{\sigma'} (\bp') \right|_{\bp'} \\
&=& \prod_{i \in \bp^+} \theta ( \hbar x_i ) \prod_{i \in \bp^-} \theta (x_i ) \prod_{j \not\in \bp} \theta \Big( x_j \zeta_j (\bp) \hbar^{- \sum_{i \in \bp^+} C_{ij}} \Big)  .
\een

Consider the product $X \times X'$, viewed as a $\bT \times \bT' \times \CC_\hbar^*$-variety, and the following equivariant embeddings:
$$
\xymatrix{
	X = X \times \{ \bp' \} \ar@{^{(}->}[r]^-{i_{\bp'}} & X \times X' & \{ \bp \} \times X' = X' \ar@{_{(}->}[l]_-{i_\bp} .
}
$$
We view $X\times \{\bp\}$ as a $\bT \times \bT' \times \CC_\hbar^*$-variety with trivial action on the second factor. Then
$$
\Ell_{\bT \times \bT' \times \CC_\hbar^*} (X \times \{ \bp' \} ) = \Ell_\TT (X) \times \cE_{\bT'} = \textsf{E}_\TT (X),
$$
where we apply the identification (\ref{id-para}) $\cE_{\bT'} \cong \cE_{\bT^\vee}$. Similarly, $\Ell_{\bT \times \bT' \times \CC_\hbar^*} (\{ \bp \} \times X' ) = \textsf{E}_{\TT'} (X')$.

Let $\mathfrak{X} := [\mu^{-1} (0) / \bK ]$ be the stacky quotient, and $\mathfrak{X}'$ similarly. The diagram above hence induces the following
$$
\xymatrix{
	\textsf{E}_\TT (X) \ar[r]^-{i_{\bp'}^*} & \Ell_{\bT \times \bT' \times \CC_\hbar^*} (X \times X') & \textsf{E}_{\TT'} (X') \ar[l]_-{i_{\bp}^*} \\
	\Ell_{\bT \times \bT' \times \CC_\hbar^* \times \bT^\vee \times (\bT')^\vee } (\pt) \ar[r] & \Ell_{\bT \times \bT' \times \CC_\hbar^*} (\mathfrak{X} \times \mathfrak{X}' ) \ar[u] & .
}
$$

We have the following main theorem.

\begin{Theorem} \label{Thm-Stab}
	
	Under the isomorphism of parameters
	$$
	\kappa_{\Stab}:  \ \bK^\vee \times \bA \times \CC^*_\hbar \xrightarrow{\sim} \bA' \times (\bK')^\vee \times \CC^*_\hbar, \qquad ( z_i , a_i, \hbar) \mapsto ( a'_i, z'_i,  \hbar^{-1}),
	$$
	we have:
	\begin{enumerate}[1)]
		
		\item There is a line bundle $\mathfrak{M}$ on $\Ell_{\bT \times \bT' \times \CC_\hbar^*} (X \times X')$ such that
		$$
		(i_{\bp'}^*)^* \mathfrak{M} = \mathfrak{M} (\bp) , \qquad (i_\bp^*)^* \mathfrak{M} = \mathfrak{M} (\bp').
		$$
		
		\item There is a section $\mathfrak{m}$ of $\mathfrak{M}$, called the ``duality interface", such that
		$$
		(i_{\bp'}^*)^* \mathfrak{m} = \bStab_{\sigma} (\bp) , \qquad (i_\bp^*)^* \mathfrak{m} = \bStab'_{\sigma'} (p_{\bp'}) .
		$$
		
		\item In the hypertoric case, the duality interface $\mathfrak{m}$ admits a simple explicit form:
		$$
		\mathfrak{m} = \prod_{i=1}^n \theta (x_i x'_i)  .
		$$
		In particular, it comes from a section of a universal line bundle on the prequotient $\Ell_{\bT \times \bT' \times \CC_\hbar^* \times \bT^\vee \times (\bT')^\vee} (\pt)$, and does not depend on the choices of $\theta$ or $\sigma$.
	\end{enumerate}
\end{Theorem}

\begin{proof}
It suffices to check 3). We compute $(i_{\bp'}^*)^* \mathfrak{m}$ by restricting $x'_j$ to the fixed point $\bp' \in X'$, followed by the change of variables $\kappa$. By (\ref{restriction-V'}) we have
\ben
(i_{\bp'}^*)^* \mathfrak{m} &=& \prod_{j\in \bp^+} \theta (\hbar^{-1} x_j) \prod_{j\in \bp^-} \theta (x_j) \prod_{j\not\in \bp} \theta \Big( a'_j \prod_{i \in \bp} (a'_i)^{C_{ij}} \cdot \hbar^{\sum_{i\in \bp^+} C_{ij} } \cdot  x_j  \Big) \\
&\xrightarrow{\kappa}& \prod_{j\in \bp^+} \theta (\hbar x_j) \prod_{j\in \bp^-} \theta ( x_j) \prod_{j\not\in \bp} \theta \Big( \zeta_j (\bp) \cdot \hbar^{ - \sum_{i\in \bp^+} C_{ij} } \cdot  x_j  \Big) \\
&=& \bStab_\sigma (\bp).
\een
The computation for $(i_\bp^*) \mathfrak{m}$ is similar.
\end{proof}

\begin{Corollary} \label{Cor-Stab}
We have the following symmetry between elliptic stable envelopes:
$$
\frac{\Stab_\sigma (\bp) |_\bq }{\Stab_\sigma (\bq) |_\bq } = \frac{\Stab'_{\sigma'} (\bq') |_{\bp'}}{\Stab'_{\sigma'} (\bp') |_{\bp'}},
$$
where $\bp, \bq \in X^\TT$, and $\bp', \bq' \in (X')^{\TT'}$ are fixed points corresponding to each other.
\end{Corollary}

\subsection{Opposite polarization and vertex function for mirror} \label{oppo}

Recall that to define the vertex function, we need to choose a global polarization; to add prefactors and form modified vertex functions, we need to choose a localized polarization. We choose them for the mirror $X'$ as follows.

\begin{enumerate}[(i)]
	
	\setlength{\parskip}{1ex}
	
	\item The global polarization $T_{X'}^{1/2}$ in the definition of vertex functions would be
	$$
	T_{X'}^{1/2} = \sum_{i=1}^n \hbar^{-1} L_i^{-1} - \hbar^{-1} \cO^{\oplus d} = \sum_{i=1}^n \hbar^{-1} x_i^{-1} - d \hbar^{-1} ,
	$$
	which is \emph{opposite}, compared to the choice $T_X^{1/2}$ in (\ref{Polar-X}). As a result, the shifted K\"ahler parameters $z'_\sharp$ are
	$$
	z'_{\sharp, i} = z'_i \cdot (-\hbar^{1/2}).
	$$
	
	\item The localized polarization $\Pol'_{\bp'}$ is chosen as
	$$
	\Pol'_{\bp'} (x') = \sum_{\cA_{\bp'}^+ \cup (\bp')^-} \hbar^{-1} (x'_i)^{-1} + \sum_{\cA_{\bp'}^- \cup (\bp')^+} x'_i - d \hbar^{-1}.
	$$
	The shifted K\"ahler parameters $z'_\epsilon$ are
	$$
	z'_{\epsilon, i} = z'_{\sharp, i} \cdot (q\hbar^{-1})^{- \epsilon' (i)},
	$$
	where $\epsilon' (i) = 0$ for $i\in \cA_{\bp'}^+ \cup (\bp')^-$ , and $\epsilon' (i) = -1$ for $i\in \cA_{\bp'}^- \cup (\bp')^+$.
	
\end{enumerate}

Under these choices, we restate the characterization for vertex functions of $X'$.

\begin{Lemma}[Theorem (\ref{q-diff-eqn}) and Lemma (\ref{uniqueness}) for $X'$] \label{uniqueness-X'}
Given $\bp' \in (X')^{\TT'}$, the modified bare vertex function $( \widetilde V')^{(\bone)} (q, z') |_{\bp'}$ is uniquely characterized by the following.

\begin{enumerate}[(i)]

\item It is annihilated by the following $q$-difference operators
\begin{equation} \label{q-diff-Z'}
\prod_{i\in (S')^+} ( 1 - Z'_i ) \prod_{i\in (S')^-} ( 1 - \hbar  Z'_i )  -  (q \hbar^{-1} z'_\sharp)^\beta \prod_{i\in (S')^+} ( 1 - \hbar  Z'_i ) \prod_{i\in (S')^-} ( 1 - Z'_i ) ,
\end{equation}
where $Z'_i$ is the operater $z'_i \mapsto q z'_i$, $S' = (S')^+ \sqcup (S')^-$ runs through all circuits of $X'$.

\item It admits the asymptotic behavior
$$
( \widetilde V')^{(\bone)} (q, z') \big|_{\bp'} \sim e^{\sum_{i=1}^n \frac{\ln z'_{\epsilon, i} (\bp') \ln x'_i |_{\bp'}}{\ln q} }   \cdot \prod_{i\in \cA_{\bp'}^+ \cup (\bp')^-} \frac{\phi (q\hbar^{-1} (x'_i)^{-1} |_{\bp'} )}{\phi ( (x'_i)^{-1} |_{\bp'} )} \prod_{i\in \cA_{\bp'}^- \cup \bp'} \frac{\phi (q x'_i |_{\bp'} )}{\phi ( \hbar x'_i |_{\bp'} )} \cdot (1 + o(\zeta' (\bp') ) )  ,
$$
as $\zeta' \xrightarrow{\theta'} 0$, where the limit is explained in Remark \ref{limit-point}.

\end{enumerate}

\end{Lemma}

\subsection{Mirror symmetry for $q$-difference equations and vertex functions}

In this subsection, we study the mirror symmetry statement for $q$-difference equations and vertex functions.

Consider the matrix $\fP$ whose entries are defined by \footnote{This matrix $\fP$ differs from the one in \cite{AOelliptic} by an exponential factor.}
$$
\fP_{\bq, \bp} := \frac{\Stab^\sharp_{\sigma} (\bq) |_\bp}{\Theta (T^{1/2}_X \big|_\bp )}  \cdot \frac{\Phi ((q-\hbar) T^{1/2}_X \big|_\bp ) }{\kappa_{\mathrm{vtx}} (\Phi ((q-\hbar) \Pol'_{\bq'} |_{\bq'} ) )}    \cdot (-\hbar^{1/2})^{|\bq^+|} \prod_{i\in \bq^+} x_i |_\bq ,
$$
for $\bq, \bp \in X^\bT$, where $\Stab^\sharp_{\sigma} (\bq)$ is defined as the stable envelope $\Stab_\sigma (\bq)$ with change of variables $\zeta_j (\bq) \mapsto \zeta_{\sharp, j} (\bq)^{-1}$, and $\kappa_{\mathrm{vtx}}$ is the identification of parameters in the following main theorem.

\begin{Theorem} \label{main-theorem}
Under the identification of parameters
$$
\kappa_{\mathrm{vtx}}:  \qquad \bK^\vee \times \bA \times \CC^*_\hbar \xrightarrow{\sim} \bA' \times (\bK')^\vee \times \CC^*_\hbar, \qquad ( z_{\sharp, i} , a_i, \hbar) \mapsto ( (a'_i)^{-1}, z'_{\sharp, i},  q \hbar^{-1}),
$$	
the product
$$
V'(q, z', a') = \fP \cdot V (q, z, a)  \quad \in \quad  K_\TT (X^\TT)
$$
forms a global class in $K_{\TT'} (X')$, and coincides with the vertex function $V'(q, z', a')$ of the 3d-mirror $X'$, with the opposite polarization $T^{1/2}_{X'}$.
\end{Theorem}

The following two lemmas provide the key incredients in the proof of this main theorem.

\begin{Lemma} \label{Criterion-1}
Given $\bq' \in (X')^{\TT'}$, under the change of variables $\kappa_{\mathrm{vtx}}$, the functions
$$
\kappa_{\mathrm{vtx}} \Big( e^{\sum_{i=1}^n \frac{\ln z'_{\epsilon, i} (\bq')  \ln x'_i |_{\bq'}}{\ln q}} \Phi ((q-\hbar) \Pol'_{\bq'} |_{\bq'} ) \Big) \cdot ( \fP \cdot V(q,z,a) ) \big|_\bq
$$
are annihilated by the $q$-difference operators (\ref{q-diff-Z'}) for the mirror $X'$.
\end{Lemma}

\begin{proof}
Explicit computation shows that the $q$-shifts of
$$
\dfrac{\Stab^\sharp_{\sigma} (\bq) }{\Theta ( \Pol_\bp )} = \prod_{i \in \bq^+} \frac{\theta ( \hbar x_i )}{ \theta ( x_i )}  \prod_{j \in \cA_\bq^+} \dfrac{\theta \Big( x_j \zeta_{\sharp, j} (\bq)^{-1} \hbar^{-\sum_{i \in \bq^+} C_{ij}} \Big) }{ \theta (  x_j ) \theta \Big(  \zeta_{\sharp, j} (\bq)^{-1} \hbar^{- \sum_{i \in \bq^+} C_{ij}}  \Big) } \prod_{j \in \cA_\bq^-} \dfrac{\theta \Big(  x_j \zeta_{\sharp, j} (\bq)^{-1} \hbar^{- \sum_{i \in \bq^+} C_{ij}} \Big) }{ \theta (  x_j ) \theta \Big( \hbar^{-1} \zeta_{\sharp, j} (\bq)^{-1} \hbar^{ -\sum_{i \in \bq^+} C_{ij} }  \Big) } \prod_{i\in \cA_\bp^+} \frac{\theta (x_i) }{ \theta (\hbar^{-1} x_i^{-1} ) } ,
$$
with respect to variables $a_i$, $z_i$ and $s_i$ are the same as
$$
\exp \Big( \sum_{i=1}^n \frac{\ln z_{\sharp, i} \ln x_i}{\ln q} - \sum_{i=1}^n \frac{\ln z_{\sharp, i} \ln x_i |_\bq }{\ln q} - \sum_{i\in \bq^+} \frac{\ln x_i |_\bq \ln \hbar}{\ln q} + \sum_{i\in \cA_\bp^+} \frac{\ln x_i \ln \hbar}{\ln q} \Big).
$$
Therefore $\fP \cdot V(q,z, a)$ satisfies the same $q$-difference equations (with respect to $z_i$'s and $a_i$'s) as
$$
\sum_\bp  e^{ \sum_{i=1}^n \frac{\ln z_{\epsilon, i} (\bp) \ln x_i |_\bp }{\ln q} - \sum_{i=1}^n \frac{\ln z_{\epsilon, i} (\bq) \ln x_i |_\bq }{\ln q} - \sum_{i\in \bq^+} \frac{\ln x_i |_\bq \ln \hbar}{\ln q} } \cdot \prod_{i\in \bq^+} x_i |_\bq \cdot \prod_{i\in \bp} \frac{\phi ( q x_i |_\bp )}{\phi ( \hbar x_i |_\bp  )} \cdot V (q,z,a) \big|_\bp
$$
where we have used the observation that $\sum_{i=1}^n \ln z_{\epsilon, i} (\bp) \ln x_i |_\bp = \sum_{i=1}^n \ln z_{\sharp, i} \ln x_i |_\bp$. Hence, by Lemma \ref{id-prime}, we have
$$
\kappa_{\mathrm{vtx}} \Big( e^{\sum_{i=1}^n \frac{\ln z'_{\epsilon, i} (\bq')  \ln x'_i |_{\bq'}}{\ln q}} \Phi ((q-\hbar) \Pol'_{\bq'} |_{\bq'} ) \Big) \cdot ( \fP \cdot V(q,z,a) ) \big|_\bq
$$
satisfies the same $q$-difference equations as
\ben
&& e^{\sum_{i=1}^n \frac{\ln z'_{\epsilon, i} (\bq') \ln x'_i |_{\bq'} }{\ln q} - \sum_{i=1}^n \frac{\ln z_{\epsilon, i} (\bq) \ln x_i |_\bq }{\ln q} - \sum_{i\in \bq^+} \frac{\ln x_i |_\bq  \ln \hbar}{\ln q} } \cdot \prod_{i\in \bq^+} x_i |_\bq \cdot \widetilde V (q, z, a) \\
&=& e^{- \sum_{i=1}^n \frac{\ln z_{\sharp, i} \ln a_i}{\ln q} - \sum_{i\in \bq^+, j\in \cA_\bq^-} C_{ij} \frac{\ln (q\hbar^{-1} ) \ln \hbar }{\ln q} } \cdot \widetilde V (q, z, a).
\een
By Theorem \ref{q-diff-eqn}, it satisfies the equation (\ref{q-diff-A}), which under the change of variables $\kappa_{\mathrm{vtx}}$, is the same as the equation (\ref{q-diff-Z'}).
\end{proof}

\begin{Lemma} \label{Criterion-2}
Under the change of variables $\kappa_{\mathrm{vtx}}$, the functions $( \fP \cdot V(q,z,a) ) \big|_\bq$ admit the following asymptotic behavior
$$
( \fP \cdot V(q,z,a) ) \big|_\bq \sim (1 + o(\alpha(\bq)))
$$
as $\alpha (\bp) \xrightarrow{\sigma} 0$ (see Appendix \ref{limit-point-a}).
\end{Lemma}

\begin{proof}
We first consider the asymptotic behavior of  the function along a generic \emph{$q$-geometric progression} of the form
$$
\alpha_i (\bq) = w_i q^{\pm N}, \qquad i\in \bq^\pm,  \qquad N\to \infty .
$$
Consider $\kappa_{\mathrm{vtx}} \Big( \Phi ((q-\hbar) \Pol'_{\bq'} |_{\bq'} ) \Big) \cdot (\fP \cdot V(q,z,a) ) \big|_\bq$, which is
$$
\sum_{\bp \in X^\TT} \Cont_\bp := (-\hbar^{1/2})^{|\bq^+|}  \prod_{i\in \bq^+} x_i |_\bq \cdot \sum_{\bp \in X^\TT}   \frac{ \Stab^\sharp_{\sigma} (\bq) |_\bp }{\Theta (T_X^{1/2} |_\bp ) } \cdot \Phi ((q-\hbar) T^{1/2}_X |_\bp )  \cdot V(q,z,a) \big|_\bp  .
$$
We claim that along any generic $q$-geometric progression of the above form, the summation over $\bp$ is \emph{dominated by the diagonal term}. Let's analyze the contributions from each fixed point $\bp$ in more details.

\textbf{Case 1}: $\bp \neq \bq$.

Let $\bp \in X^\TT$ be such that $\bq \neq \bp$. It's clear that the matrix $\fP_{\bq, \bp}$ vanishes unless $\bq\in \overline{\Attr_\sigma (\bp)}$.  The contribution from $\bp$, viewed as a function in terms of $a_i$'s and $z_i$'s, is a product of factors of the following three types (using the identity $\theta (x) = x^{1/2} \phi (qx) \phi (x^{-1})$):
\begin{enumerate}[(i)]
	
\setlength{\parskip}{1ex}
	
\item $x_i^{\pm 1} |_\bp$ or $x_i^{\pm 1} |_\bq$, $1\leq i\leq n$;

\item $\dfrac{\phi (u x_i^{\pm 1} |_\bp) }{\phi (v x_i^{\pm 1} |_\bp)}$, $1 \leq i\leq n$, where $u$ and $v$ are some monomials of $q$ and $\hbar$; 

\item $\dfrac{\phi (u \zeta_{\sharp, j} (\bq)^{\mp 1} x_j^{\pm 1} |_\bp) }{ \phi (x_j^{\pm 1} |_\bp) \phi (v \zeta_{\sharp, j} (\bq)^{\mp 1} )}$, $j\not\in \bq$, where $u$ and $v$ are some monomials $q$ and $\hbar$, and $x_j |_\bp \neq 1$;

\item $V(q, z, a) |_\bp$.

\end{enumerate}

We see that along a generic $q$-geometric progression $\{\alpha_i (\bq) = w_i q^{\pm N} \}$, as $N\to \infty$, factors of type (i) and (ii), either have finite limits, or have asymptotes of the form $\text{const}\cdot (u/v)^N$. Factor (iv) admits a finite limit by Proposition \ref{limit-V}. However, factors of type (iii) have the asymptotic form
$( u \zeta_{\sharp, j} (\bq)^{\mp 1} )^N$. Therefore, the contribution from $\bp$ is asymptotically of the form
$$
\Cont_\bp \sim  f(q, \hbar)^N \cdot g(q, \hbar, \zeta_\sharp) \cdot \Big( \prod_{j\in \cA_\bq \cap \bp^+ } \zeta_{\sharp, j} (\bp) \prod_{j \in \cA_\bq \cap \bp^-} \zeta^{-1}_{\sharp, j} (\bq) \Big)^N
$$
where $f$ is a monomial in $q$ and $\hbar$, only depending on $q$ and $\hbar$, and independent of the initial point $w_i$, and $g$ is a function depending only on $q$, $\hbar$ and $\zeta_\sharp$. Now as $\bp\neq \bq$, we always have $\cA_\bq \cap \bp\neq \emptyset$, and one can uniformly choose $\zeta_{\sharp, j} (\bp)$ sufficiently small (or large, depending on $j\in \cA_\bq \cap \bp^\pm$), such that the asymptote above goes to $0$, as $N\to \infty$.

In other words, for $\bq \neq \bp$, locally in some open subset $U(\bq)$ of the domain of $\zeta_\sharp$, the limit of $\Cont_\bp (\fP \cdot V(q,z,a) ) \big|_\bq$ along any generic $q$-geometric progression is $0$. Note that as $\bq\in \overline{\Attr_\sigma (\bp)}$,  $\cA_\bq \cap \bp^\pm = \cA_\bq^\pm \cap \bp^\pm$ by Lemma \ref{Attr}, and one can always take the open subset $U(\bq)$ such that it does not depend on $\bp$.

\textbf{Case 2}: $\bp = \bq$.

We can compute the diagonal contribution explicitly. Recall that $\Stab^\sharp_\sigma (\bq) |_\bq = (-1)^{|\bq^+|} \Theta (N_\bq^-)$. Hence
\ben
\Cont_\bq &=& (-1)^{|\bq^+|} (-\hbar^{1/2})^{|\bq^+|}  \prod_{i\in \bq^+} x_i |_\bq \cdot \frac{ \Theta (N_\bq^-) }{\Theta (T_X^{1/2} |_\bq ) } \cdot \Phi ((q-\hbar) T^{1/2}_X |_\bq )  \cdot V(q,z,a) \big|_\bq \\
&=& (-\hbar^{1/2})^{|\bq^+|}  \prod_{i\in \bq^+} x_i  \prod_{i\in \bq^+} \frac{\theta (\hbar x_i |_\bq)}{\theta (x_i |_\bq)} \prod_{i\in \bq} \frac{\phi (q x_i |_\bq)}{\phi (\hbar x_i |_\bq)} \cdot V(q, z, a) |_\bq \\
&=& \prod_{i\in \bq^+} \frac{\phi (q \hbar^{-1} x_i^{-1} |_\bq) }{\phi (x_i^{-1} |_\bq)} \prod_{i\in \bq^-} \frac{\phi (q x_i |_\bq)}{\phi (\hbar x_i |_\bq)} \cdot V(q, z, a) |_\bq.
\een
As $\alpha(\bq) \xrightarrow{\sigma} 0$, by Proposition \ref{limit-V}, it has the limit $\lim_{\alpha(\bq) \xrightarrow{\sigma} 0} V(q,z,a) |_\bq = \kappa_{\mathrm{vtx}} \Big( \Phi ((q-\hbar) \Pol'_{\bq'} |_{\bq'} ) \Big)$.

Now let's go back to the lemma, and consider $\sum_{\bp\in X^\TT} \Cont_\bq$. We will apply the holomorphicity results from Aganagic--Okounkov \cite{AOelliptic}. By Theorem 5 of \cite{AOelliptic}, the function $( \fP \cdot V(q, z, a) ) |_\bq $ is holomorphic with respect to $a_i$'s in a \emph{punctured} neighborhood of the limit point $\alpha (\bp) \xrightarrow{\sigma} 0$. On the other hand, within the region $\zeta_\sharp \in U(\bq)$, comibing Case 1 and 2, we see that it has a finite limit $\kappa_{\mathrm{vtx}} \Big( \Phi ((q-\hbar) \Pol'_{\bq'} |_{\bq'} ) \Big)$ along any $q$-geometric progression. The lemma then follows from Riemann extension theorem and analytic continuation.
\end{proof}

\begin{proof}[Proof of Theorem \ref{main-theorem}]

To prove the theorem, it suffices to check that $e^{\sum_{i=1}^n \frac{\ln z'_{\epsilon, i} (\bp') \ln x'_i |_{\bp'}}{\ln q} } \cdot (\fP \cdot V(q,z,a) ) \big|_\bq$ satisfies the uniqueness criteria for $\widetilde V'$ in Lemma \ref{uniqueness-X'}. Now Criterion (i) and (ii) there are respectively checked in Lemma \ref{Criterion-1} and \ref{Criterion-2}. The theorem follows.
\end{proof}

\begin{Remark}
It is interesting to ask how the results can be generalized to the orbifold case, i.e. to hypertoric DM stacks, introduced in \cite{JT-1, JT-2}.
In that case, the unimodular assumption does not hold any more, which implies that entries of the matrix $\iota$ are no longer restricted to $\pm 1$ and $0$. 
There will be extra increase of orders for the $q$-difference equations and quantum $K$-theory relations.
We expect there are some new nontrivial phenomenon happening in the orbifold theory, which deserves future exploration. 
\end{Remark}


\section{Relationship to Givental's quantum $K$-theory}

\subsection{$K$-theoretic Gromov--Witten theory and $J$-function}

As in the usual Gromov--Witten theory, the $K$-theoretic analogue, introduced by Givental \cite{Giv-WDVV} and Lee \cite{Lee}, considers the moduli space of stable maps $\overline\cM_{0, N} (X, \beta)$, parameterizing genus-$0$ stable maps into $X$. The usual GW perfect obstruction theory defines a virtual structure sheaf $\cO_\vir$. Most of the properties of cohomological GW theory can be generalized to $K$-theory, although in an essentially nontrivial way.

In \cite{Giv1}, Givental introduced another variant of $K$-theoretic GW theory, called the \emph{permutation-equivariant} quantum $K$-theory. The idea is to consider the invariants not merely as numbers, but as $S_N$-modules, keeping track on the permutations of marked points. The permutation-equivariant theory turns out to behave much better than the ordinary version of $K$-theory.

Let $X$ be a quasiprojective variety. Let $\Lambda$ be a $\lambda$-algebra \footnote{The Adams operation of $\Lambda$ naturally includes the usual $\Psi^k$ operators on symmetric functions, and also on the K\"ahler parameters $\Psi^k (Q_i) = Q_i^k$.} over $\QQ$, which contains the ring of symmetric functions on a certain number of variables, and the K\"ahler parameters $Q_i$. Let $\bt(q)$ be a Laurent polynomial in $q$ with coefficients in $K(X) \otimes \Lambda$. The moduli space of stable maps $\overline\cM_{0, N} (X, \beta)$ admits an action of the group $S_N$, permuting the marked points. The virtual structure sheaf $\cO_\vir$ is equivariant under the $S_N$-action, and hence defines a $K$-theory class on the quotient.

The correlation functions are defined as
$$
\left\langle \bt (L), \cdots, \bt (L) \right\rangle_{0, N, \beta}^{S_N} := \chi \Big( \left[ \overline\cM_{0, N} (X, \beta) / S_N \right] , \cO_\vir \otimes \prod_{i=1}^N \bt (L_i) \Big),
$$
where $L_i$ is the tautological cotangent line bundle at the $i$-th marked point.

The genus-$0$ invariants are encoded in the \emph{big $J$-function}, defined as
$$
J (\bt(q), Q) := 1-q + \bt(q) + \sum_{i, N\geq 2, \beta} \phi_i \Big\langle \bt (L),  \cdots, \bt(L) , \frac{\phi^i}{1 - qL} \Big\rangle_{0, N+1}^{S_N}  Q^\beta,
$$
where $S_N$ only acts on the first $N$ points, $\{\phi_i\}$ is a basis of $K(X)$ and $\{\phi^i\}$ is the dual basis with respect to the Mukai pairing.

The language of the loop space is introduced to describe the range of the big $J$-function. Consider the space $\cK := K(X)\otimes \Lambda (q)$, consisting of rational functions in $q$ with coefficients in $K(X) \otimes \Lambda$. $\cK$ admits a symplectic form $\Omega (f, g) := \left( \Res_{q = 0} + \Res_{q = \infty} \right) (f (q), g (q^{-1}) ) \frac{dq}{q}$, and a decomposition into  Lagrangian subspaces $\cK = \cK_+ \oplus \cK_- = T^* \cK_+$. Here $\cK_+$ is the subspace $K(X)\otimes \Lambda [q, q^{-1}]$, and $\cK_-$ is the subspaces of \emph{reduced rational functions}, i.e., rational functions in $q$, regular as $q \to 0$, and tends to $0$ as $q\to \infty$. The big $J$-function can be viewed naturally as the graph of a function from $\cK_+$ to $\cK_-$, with $\bt(q) \in \cK_+$, and $1-q$ the \emph{dilaton shift} of the origin. When $X$ admits the action by a torus $\TT$, as before, everything can be made $\TT$-equivariantly.

The aim of this subsection, is to prove the vertex functions $V^{(1)} (q^{-1}, z)$, defined in previous sections for hypertoric varieties, multiplied by $(1-q)$, represents a value of the big $J$-function, up to a certain $q$-shift of K\"ahler parameters. To show that, we will apply a criteria by Givental which characterize the range of a big $J$-function.

Let $X$ be a GKM variety, with the torus action by $\TT$. We take the $\lambda$-algebra to be
$$
\Lambda := K_\TT (\pt) [[Q^{\Eff (X)}]],
$$
and let $\Lambda_+$ to be the ideal generated by $1 - a_i^{\pm 1}$, $1 - \hbar^{\pm 1}$ and $Q$. We assume that all coefficients $\bt_k$ of $\bt (q)$ are taken in $K_\TT (X) \otimes \Lambda_+$.

\begin{Remark}
When $\Lambda$ does not contain the ring of symmetric functions, following Example 4 of \cite{Giv1}, we will refer to the permutation-equivariant quantum $K$-theory as \emph{symmetrized} quantum $K$-theory. The full permutation-equivariant invariants, when $\Lambda$ includes the ring of symmetric functions, actually contain all information about the pushforward of $\cO_\vir$ along the projection $\left[ \overline{\cM}_{0,N}(X, \beta) / S_N \right] \to B S_N$.
\end{Remark}

Let $\{ f^{({\bf p})} \in  \Lambda (q) \mid {\bf p} \in X^\TT \}$ be a set of elements.
 For ${\bf p, q} \in X^\TT$ connected by a 1-dimensional $\TT$-invariant curve $C$, we denote by $\lambda_{\bf p,q}$ the $\TT$-character $T_{\bf p} C$. The proof of the following proposition is the same as in \cite{Giv2}.

\begin{Proposition} \label{Givental-cri}
Suppose that $\{ f^{({\bf p})} \in \Lambda (q) \mid {\bf p} \in X^\TT \}$ satiesfy the following two criteria.
\begin{enumerate}[(i)]

\item For each ${\bf p} \in X^\TT$, considered as a meromorphic function in $q$ with only poles at roots of unity, $f^{({\bf p})}$ represents a value of the big $J$-function $J_{\pt}$ of a point target space.

\item  Outside $q = 0, \infty$ and roots of unity, $f^{({\bf p})}$ may have poles only at $q = \lambda_{\bf p,q}^{1/m}$, $m = 1, 2, \cdots$, with residues
$$
\Res\limits_{ q = \lambda_{\bf p,q}^{ 1/m} } f^{({\bf p})} (q) \frac{dq}{q} = - \frac{ Q^{m [C]}}{m E_{\bf p,q} (m)} \cdot  f^{( {\bf q} )} ( \lambda_{\bf p, q}^{1/m} ),
$$
where $[C]$ is the curve class of the curve $C$, and
$$
E_{\bf p, q} (m) = \bigwedge^\bullet \left( T_\varphi \overline\cM_{0,2} (X, m[C])  - T_{\bf p} X \right)^\vee
$$
where $\varphi: \PP^1 \to C \subset X$ is the deg-$m$ covering map over $C$, ramified at $0$ and $\infty$.

\end{enumerate}
Then there exists $\bt (q) \in K_\TT (X) [q, q^{-1}]$, such that $f^{({\bf p})} = \left. J_X (\bt (q) ) \right|_{\bf p}$, for each ${\bf p} \in X^\TT$. In other words, $f^{({\bf p})}$'s represents a value of the permutation-equivariant big $J$-function $J_X$ of $X$.
\end{Proposition}

\begin{Theorem} \label{V}
Let $X$ be a hypertoric variety, and $z_\sharp$ be the $q$-shifted K\"ahler parameters defined as $z_{\sharp, i} := z_i \cdot (-\hbar^{-1/2})$. The vertex function
$$
(1-q) V^{(\bone)} (q^{-1} , z) \big|_{z_\sharp = Q}
$$
represents a value $J(t_0, Q)$ of the big $J$-function of $X$, for some $t_0 \in K_\TT (X) \otimes \Lambda$.
\end{Theorem}

\begin{proof}
Take $\bp \in X^\TT$. Explicitly, one can compute that
$$
V^{(\bone)}(q,z) \big|_\bp  = \sum_{\beta \in \Eff(X)} z_\sharp^\beta  \prod_{i=1}^n \frac{ ( \hbar x_i |_\bp )_{D_i} }{ ( q x_i |_\bp )_{D_i} } ,
$$
where $D_i = \beta \cdot L_i$. By similar arguments as in Corollary 1 of \cite{Giv4}, we see that Criterion (i) holds. Indeed, it follows from Theorem in \cite{Giv2}, which describes the big $J$-function of a point, and Lemma in \cite{Giv4}, which gives a set of difference operators that preserves the range of the big $J$-function of a point.

We now check that $V^{(\bone)} (q , z)$ satiesfies Criterion (ii), with $q$ replaced by $q^{-1}$.  Let  $\bq$ be another vertex in the hyperplane arrangement, such that $\bp = (\bq \backslash \{j\} ) \sqcup \{i\}$ and $\bq = (\bp \backslash \{i\} ) \sqcup \{j\}$, for some $1\leq i\neq  j\leq n$. Let $C$ be the $\TT$-invariant curve connecting $\bp$ and $\bq$. Then we have two cases as in Lemma \ref{bridge} (ii). For simplicity, we assume that $i\in \cA_\bq^+$; the other case $i\in \cA_\bq^-$ is similar. We have $\lambda:= \lambda_{\bp, \bq} = x_i |_\bp$.

Let $m\geq 0$ be an integer. Consider the coefficent of $Q^\beta$ in $V^{(\bone)} (q,z) |_\bp$. Its residue at $q = \lambda^{-1/m}$ vanishes unless $D_i \geq m$. In that case, for $\beta \in \Eff_\bp (X)$, we claim that the curve class $\beta - m[C] \in \Eff_\bq (X)$. In fact, for any $l\not\in \bp \cup \bq$, we know that in $\RR^d$ the vertices $\bp$ and $\bq$ lie on the same side of the hyperplane $H_l$, which implies that $\cA_\bp^\pm$ and $\cA_\bq^\pm$ only differ by the indices $i$ or $j$. On the other hand, the curve class $[C]$ is defined by the circuit $S_{\bp \bq} =  \bp \cup \bq$; in other words, it is of the form $[C] = e_i + \sum_{l\in \bq } \epsilon_l e_l$, where $\epsilon_l = \pm 1$, depending on $l \in S_{\bp \bq}^\pm$. Therefore, we have
$$
(\beta - m[C],  L_l) = \left\{ \begin{aligned}
& (\beta,  L_l ) , && \qquad l \in \cA_\bp^\pm \backslash \{j\} = \cA_\bq^\pm \backslash \{i\} \\
& D_i - m , && \qquad l = i .
\end{aligned} \right.
$$
We see that $\beta - m [ C] \in \Eff_\bq (X)$. The claim holds.

Direct computation shows that
$$
\Res_{q = \lambda^{-1/m}} \frac{ ( \hbar x_i |_\bp )_{D_i} }{ ( q x_i |_\bp )_{D_i} } \frac{dq}{q} = \frac{1}{m} \left. \frac{ ( \hbar x_i |_\bp )_{m} }{ ( q x_i |_\bp )_{m} } \right|_{q = \lambda^{-1/m}} \cdot \left. \frac{ ( \hbar x_i |_\bq )_{D_i - m} }{ ( q x_i |_\bq )_{D_i - m} } \right|_{q =  \lambda^{-1/m}} ,
$$
$$
\left. \frac{ ( \hbar x_l |_\bp )_{D_l} }{ ( q x_l |_\bp )_{D_l} } \right|_{q = \lambda^{-1/m}} = \left. \frac{ ( \hbar x_l |_\bq )_{D_l\mp m} }{ ( q x_l |_\bq )_{D_l \mp m} } \right|_{q = \lambda^{-1/m}} \cdot \left. \frac{ ( \hbar x_l |_\bp )_{\pm m} }{ ( q x_l |_\bp )_{\pm m} } \right|_{q = \lambda^{-1/m}}, \qquad l\in S_{\bp \bq}^\pm \backslash \{i\},
$$
and on the other hand,
$$
E_{\bp, \bq}^{-1} (m) =  \prod_{l\in S_{\bp \bq}^+} \left.  \frac{ ( \hbar x_l |_\bp )_{m} }{ ( q x_l |_\bp )_{m} } \right|_{q = \lambda^{-1/m}} \cdot \prod_{l\in S_{\bp \bq}^-} \left. \frac{ ( \hbar x_l |_\bp )_{-m} }{ ( q x_l |_\bp )_{-m} } \right|_{q = \lambda^{-1/m}}.
$$
We see that Criterion (ii) is satisfied. The theorem follows from Proposition \ref{Givental-cri}.
\end{proof}

\begin{Corollary} \label{V-tau}
Let $\tau$ be a Laurent polynomial in $q$ with coefficients in $K_\bT(X) \otimes \Lambda$. The descendent bare vertex function
$$
(1-q) V^{(\tau)} (q^{-1} , z) \big|_{z_\sharp = Q}
$$
lies in the range of big $J$-function of $X$.
\end{Corollary}

\begin{proof}
By the Theorem 2 of explicit reconstruction in \cite{Giv8}, given a point $\sum_\beta I_\beta Q^\beta$ in the rangle $\cL$ of big $J$-function, the point
$$
\sum_\beta I_\beta Q^\beta \tau (x_1 q^{D_1}, \cdots, x_n q^{D_n})
$$
also lies in $\cL$. The corollary then follows.
\end{proof}

\subsection{Quantum $K$-theory (in the sense of Givental)}

In Section \ref{QK}, we defined a quantum $K$-theory ring in terms of virtual counting of parameterized quasimaps from $\PP^1$ to $X$. On the other hand, it is standard to define quantum cohomology or $K$-theory ring \cite{Lee}, in terms of genus-zero stable maps into $X$. It is natural to ask whether these two versions of quantum $K$-theories coincide. The question is somehow complicated due to the lack of divisor axiom in $K$-theory, as opposed to the cohomological theory. We will see later that the PSZ quantum $K$-theory is essentially the same as the algebra generated by the $A_i$ operators, introduced in \cite{IMT}, but in general different from Givental's quantum $K$-theory.

Let's review the definition of Givental's quantum $K$-theory ring, which we denote by $\bullet$.  Let $X$, $\Lambda$ as in the previous subsection. In \cite{Giv7}, Givental introduced a genus-zero $K$-theoretic GW potential with \emph{mixed inputs}:
$$
\cF (\bx, \bt) = \sum_{\beta \in \Eff(X)} \sum_{M, N=0}^\infty \left\langle \bx (L) , \cdots, \bx (L); \bt (L) , \cdots, \bt (L) \right\rangle_{0, M+N, \beta}^{S_N}  \frac{Q^\beta}{M!},
$$
where $\bx (q)$, $\bt (q)$ are Laurent polynomials with coefficients in $K(X) \otimes \Lambda$; only $\bt$'s are considered as \emph{permutation-equivariant} inputs, and $\bx$'s are considered as \emph{ordinary} inputs.

The graph of the potential $\cF (\bx, \bt)$ (up to dilaton shift) therefore defines a \emph{mixed} $J$-function $J(\bx (q), \bt(q), Q)$. In particular, the permutation-equivariant $J$-function can be recovered by setting $\bx = 0$.

Let $\{\phi_\alpha\}$ be a basis of $K(X)$. We take constant Laurent polynomial as inputs: $\bx (q) = x = \sum_\alpha x^\alpha \phi_\alpha$, $\bt (q) = t \in K(X) \otimes \Lambda$. Given basis elements $\phi_\alpha, \phi_\beta, \phi_\gamma \in K (X)$, the \emph{quantum pairing} is defined as
$$
G (\phi_\alpha, \phi_\beta) := \frac{\partial^2}{\partial x^\alpha \partial x^\beta} \cF (x, t).
$$
The \emph{quantum product} is determined by the 3-point function of three basis elements $\phi_\alpha, \phi_\beta, \phi_\gamma$
$$
G (\phi_\alpha \bullet \phi_\beta, \phi_\gamma ) :=  \frac{\partial^3}{\partial x^\alpha \partial x^\beta \partial x^\gamma} \cF (x, t) .
$$
By the WDVV equation in \cite{Giv7}, the ring is equipped with a structure of Frobenius algebra, with pairing $G$, product $\bullet$ (both dependent on $x$ and $t$), and identity $\bone \in K_\bT (X)$. Similar as in the comological theory, one can define a quantum connection using the quantum product $\bullet$ and then the quantum $K$-ring structure can be packaged in the language of $\cD$-modules.

For given constant Laurent polynomials $x$ and $t$ as above, there is an operator $S (x, t, q)^{-1}: K(X) \otimes \Lambda \to \cK_-$, whose inverse is defined as
$$
S (x, t, q)^{-1} \phi := \phi + \sum_{i,  M, N, \beta} \phi_i \Big\langle \phi, x, \cdots, x, t, \cdots, t, \frac{\phi^i}{1 - qL} \Big\rangle_{0, M+N+2}^{S_N}  \frac{Q^\beta}{M!}.
$$
$S$ is a symplectomorphism, i.e., satisfying $S(q) = S^* (q^{-1})$. In particular, the $J$-function $J(x, t, Q) = (1-q) S (x, t, q)^{-1} \bone$. The image $S(x, t, q)^{-1} \cK_+$ is called the \emph{ruling space}, which satisfies the following properties \cite{Giv7, Giv8}:

\begin{enumerate}[(i)]
	
\setlength{\parskip}{1ex}
	
\item The range $\cL_{\mathrm{perm}}$ of permutation-equivariant big $J$-functions $\bt(q) \mapsto J(0, \bt (q), Q)$ of $X$ is swept by the images of $S(t,q)^{-1}$:
$$
\cL_{\mathrm{perm}} = \bigcup_{t \in K(X) \otimes \Lambda_+} (1-q) S(0, t, q)^{-1} \cK_+.
$$

\item For each fixed $t$, the tangent space of the Lagrangian cone $\cL_t$ for the ordinary big $J$-function $\bx (q) \mapsto J (\bx (q), t, Q)$ is $T_t := S(\bx (q), t, Q)^{-1} \cK_+$, and tangent to $\cL_t$ exactly along the subspace $(1-q) T_t$. In particular, $\cL_t$ is also swept by the union of those ruling spaces.

\item Let $Q_i$ be the K\"ahler parameter with respect to $\widetilde L_i$. Each ruling space admits a $\cD_q$-module structure under the $q$-difference operators $q^{Q_i \frac{\partial}{\partial Q_i}}$.

\end{enumerate}

The key observation of \cite{IMT} is that the operator $S(x, t, q)^{-1}$ serves simultaneously as the fundamental solution to a $q$-difference system with respect to $q^{Q_i \partial_{Q_i}}$, and the fundamental solution to a $q$-differential system with respect to the variables $\frac{\partial}{\partial x^\alpha}$. They take the following forms:
\begin{equation} \label{2-systems}
(1-q) \frac{\partial}{\partial x^\alpha} S(x, t, q)^{-1} = S(x, t, q)^{-1} \circ (\phi_\alpha \bullet - ), \qquad  L_i^{-1} q^{Q_i \frac{\partial}{\partial Q_i}} \circ S(x, t, q)^{-1} = S(x, t, q)^{-1} \circ B_i q^{Q_i \frac{\partial}{\partial Q_i}} ,
\end{equation}
where the first equation is by the definition of the quantum product $\bullet$, and the second is obtained from the $q$-difference module structure. Here $B_i \in \End K(X) \otimes \Lambda [q, q^{-1}]$ are uniquely characterized by the above equation. The two systems above are compatible to each other, in the sense that the quantum connection and $q$-difference operators commute. In particular, one can define the operators
$$
B_{i, \com} := B_i \big|_{q = 1} \in \End K(X) \otimes \Lambda.
$$
Now let $X$ be a hypertoric variety. Using the result on the vertex function and $J$-functions in the previous section, together with the  $q$-difference equations (\ref{q-diff-Z}), we obtain the following result.

\begin{Theorem} \label{GivQK-relation}
Let $X$ be a hypertoric variety, and $t_0 \in K_\TT (X) \otimes \Lambda$ be as in Theorem \ref{V}. We fix the insertions $x = 0$ and $t = t_0$.

\begin{enumerate}[1)]
	
		\setlength{\parskip}{1ex}
	
\item For any circuit $S = S^+\sqcup S^-$, and the corresponding curve class $\beta$, the identity class $\bone \in K_\bT (X)$ is annihilated by the following operator
$$
\prod_{i\in S^+} ( 1 - B_{i, \com}) \prod_{i\in S^-} (\hbar - B_{i, \com} ) - Q^\beta \prod_{i\in S^+} (\hbar - B_{i, \com} ) \prod_{i\in S^-} (1 - B_{i, \com}),
$$
where $Q^\beta := \prod_{i\in S^+} Q_i \prod_{i\in S^-} Q_i^{-1}$.

\item The Givental quantum $K$-theory ring of $X$ is generated by the classes $B_{i, \com} \bone$, $1\leq i\leq n$, up to the following relations: for any circuit $S = S^+\sqcup S^-$, and the corresponding curve class $\beta$
$$
\prod_{i\in S^+} ( 1 - B_{i, \com} \bone) \bullet \prod_{i\in S^-} (\hbar - B_{i, \com} \bone ) = Q^\beta \prod_{i\in S^+} (\hbar - B_{i, \com} \bone ) \bullet \prod_{i\in S^-} (1 - B_{i, \com} \bone ),
$$
where all the products are the quantum product $\bullet$.

\end{enumerate}

\end{Theorem}

\begin{proof}
By Theorem \ref{V}, the $J$-function $J(t_0, Q) = S(0, t_0, q)^{-1} \bone$ satisfies the $q$-difference equations (\ref{q-diff-Z}), with $Z_i$ replaced by $\widetilde L_i q^{-Q_i \partial_{Q_i}}$. By the 2nd equation in (\ref{2-systems}), for any Laurent polynomial $f(X_1, \cdots, X_n)$, we have
$$
S \circ f ( L_1^{-1} q^{Q_1 \partial_{Q_1}} , \cdots,  L_n^{-1} q^{Q_n \partial_{Q_n}} ) \circ S^{-1} = f(B_1 q^{Q_1 \partial_{Q_1}}, \cdots, B_n q^{Q_n \partial_{Q_n}} ).
$$
Apply operators on both sides to the identity $\bone$, and take $q\to 1$. We obtain 1).

For 2), it suffices to notice that by compatibility, for any $\alpha$, $i$, one has $B_{i, \com} \phi_\alpha = B_{i, \com} (\phi_\alpha \bullet \bone) = \phi_\alpha \bullet (B_{i, \com} \bone)$. Hence the operator $B_{i, \com}$ is the same as the quantum multiplication $(B_{i, \com} \bone) \bullet (-)$. 2) then follows from 1).
\end{proof}

\appendix

\section{Some identities of K\"ahler and equivariant parameters}

We list here some computations for K\"ahler and equivariant parameters. Let $\bp \in X^\TT$ be a fixed point, and $\iota = \begin{pmatrix}
I \\
C
\end{pmatrix}$ be the matrix in the standard $\bp$-frame.

\begin{Lemma}
$$
\sum_{i=1}^n \ln z_i \ln x_i |_\bp - \sum_{i=1}^n \ln z_i \ln a_i  = -\sum_{j\in \cA_\bp^+ } \ln \zeta_j (\bp) \ln a_j - \sum_{j\in \cA_\bp^- } \ln \zeta_j (\bp) \ln (\hbar a_j)
$$
\end{Lemma}

\begin{proof}
Recall the variable $\zeta_j (\bp)$, $j\not\in \bp$ is defined as $\zeta_j (\bp) := z_j \prod_{i\in \bp} z_i^{C_{ij}}$. Then by the restriction formula (\ref{restriction-V}),
\ben
LHS &=& \sum_{j\in \cA_\bp^-} \ln z_i \ln \hbar^{-1} + \sum_{i\in \bp} \ln z_i \Big( \ln a_i - \sum_{j\in \cA_\bp^+} C_{ij} \ln a_j - \sum_{j\in \cA_\bp^-} C_{ij} \ln (\hbar a_j) \Big) - \sum_{i=1}^n \ln z_i \ln a_i \\
&=& -\sum_{j\in \cA_\bp^+} \ln a_j \Big( \ln z_j + \sum_{i\in \bp} C_{ij} \ln z_i \Big) - \sum_{j\in \cA_\bp^-} \ln (\hbar a_j) \Big( \ln z_j + \sum_{i\in \bp} C_{ij} \ln z_i \Big) \\
&=& RHS.
\een
\end{proof}

Let $\bp' \in (X')^\TT$ be the corresponding fixed point in the mirror, and $z_\epsilon$, $z'_\epsilon$ be defined as in Definition \ref{loc-pol} 2) and Section \ref{oppo}. More precisely,
$$
z_{\epsilon, i} = \left\{ \begin{aligned}
& z_{\sharp, i} \cdot (q\hbar^{-1} )^{-1} , && \   i \in \cA_\bp^+  \\
& z_{\sharp, i} , && \   i \in \cA_\bp^- \cup \bp
\end{aligned}\right. , \qquad z'_{\epsilon, i} = \left\{ \begin{aligned}
& z'_{\sharp, i} , && \  i \in \cA_{\bp'}^+ \cup (\bp')^- \\
& z'_{\sharp, i} \cdot (q\hbar^{-1} ) , && \ i \in  \cA_{\bp'}^- \cup (\bp')^+  .
\end{aligned}\right.
$$

\begin{Lemma} \label{id-prime}
Under the change of variables $\kappa_{\mathrm{vtx}}$,
$$
\sum_{i=1}^n \ln z'_{\epsilon, i} (\bp') \ln x'_i |_{\bp'}  - \sum_{i=1}^n \ln z_{\epsilon, i} (\bp) \ln x_i |_\bp = - \sum_{i=1}^n \ln z_{\sharp, i} \ln a_i - \sum_{i\in \bp^+} \ln (q\hbar^{-1}) \ln x_i |_\bp + \sum_{j \in \cA_\bp^-, i\in \bp^+} C_{ij} \ln (q\hbar^{-1} ) \ln \hbar .
$$
\end{Lemma}

\begin{proof}
Compute
\ben
\sum_{i=1}^n \ln z_{\epsilon, j} (\bp) \ln x_i |_\bp &=& - \sum_{j\in \cA_\bp^-} \ln z_{\sharp, i} \ln \hbar + \sum_{i\in \bp} \ln z_{\sharp, i} \ln \Big( a_i \prod_{j\not\in \bp} a_j^{-C_{ij}} \cdot \hbar^{- \sum_{j\in \cA_\bp^-} C_{ij}} \Big) \\
&=& - \sum_{j\in \cA_\bp^-} \ln z_{\sharp, j} \ln \hbar + \sum_{i\in \bp} \ln z_{\sharp, i} \ln a_i - \sum_{i\in \bp, j\not\in \bp} C_{ij} \ln z_{\sharp, i} \ln a_j - \sum_{i\in \bp, j\in \cA_\bp^-} C_{ij} \ln z_{\sharp, i} \ln \hbar  . 
\een
On the other hand, since $\cA_{\bp'}^\pm = \bp^\mp$, $(\bp')^\pm = \cA_\bp^\mp$, and $C'_{ji} = - C_{ij}$,
\ben
\sum_{i=1}^n \ln z'_{\epsilon, i} (\bp) \ln x'_i |_{\bp'} &=& - \sum_{i \in \cA_{\bp'}^-} \ln z'_{\sharp, i} \ln \hbar + \sum_{j \in \bp'} \ln z'_{\sharp, j} \ln \Big ( a'_j \prod_{i \not\in \bp' } (a'_i)^{-C'_{ji}} \cdot \hbar^{- \sum_{i\in \cA_{\bp'}^-} C'_{ji}} \Big) \\
&& + \sum_{j\in (\bp')^+} \ln (q\hbar^{-1}) \ln \Big( a'_j \prod_{i \not\in \bp' } (a'_i)^{-C'_{ji}} \cdot \hbar^{- \sum_{i\in \cA_{\bp'}^-} C'_{ji}} \Big) \\
&=& - \sum_{i\in \bp^+} \ln z'_{\sharp, i} \ln \hbar + \sum_{j\not\in \bp} \ln z'_{\sharp, j} \ln a'_j + \sum_{j\not\in \bp, i\in \bp} C_{ij} \ln z'_{\sharp, j} \ln a'_i  + \sum_{j\not\in \bp, i\in (\bp')^+} C_{ij} \ln z'_{\sharp, j} \ln \hbar \\
&& + \sum_{j\in \cA_\bp^-} \ln (q\hbar^{-1}) \ln a'_j + \sum_{j\in \cA_\bp^-, i\in \bp} C_{ij} \ln (q\hbar^{-1}) \ln a'_i  + \sum_{j \in \cA_\bp^-, i\in (\bp')^+} C_{ij} \ln (q\hbar^{-1} ) \ln \hbar ,
\een
which under the change of variables $\kappa_{\mathrm{vtx}}$ is
\ben
&& - \sum_{i \in \bp^+} \ln a_i \ln (q\hbar^{-1} ) - \sum_{j\not\in \bp} \ln a_j \ln z_{\sharp, j} - \sum_{j\not\in \bp, i\in \bp} C_{ij} \ln a_j \ln z_{\sharp, i} + \sum_{j \not\in \bp, i\in \bp^+} C_{ij} \ln a_j \ln (q\hbar^{-1} ) \\
&& - \sum_{j \in \cA_\bp^-} \ln \hbar \ln z_{\sharp, j} - \sum_{j \in \cA_\bp^-, i \in \bp} C_{ij} \ln \hbar \ln z_{\sharp, i} + \sum_{j \in \cA_\bp^-, i\in \bp^+} C_{ij} \ln (q\hbar^{-1} ) \ln \hbar .
\een
The lemma follows by direct comparison.
\end{proof}

\section{Limit of bare vertex functions} \label{limit-point-a}

Let $\bp \in X^\TT$ be a fixed point, and $\iota = \begin{pmatrix}
I \\
C
\end{pmatrix}$ be the matrix in the standard $\bp$-frame. By the explicit formula (\ref{vertex}), we have the bare vertex function
$$
V^{(\bone)} (q, z, a) |_\bp = \sum_{\substack{d_j \geq 0, j\in \cA_\bp^+ \\ d_j \leq 0, j \in \cA_\bp^-} } \zeta (\bp)^d q^{ - \frac{1}{2} \sum_{j\not\in \bp} d_j - \frac{1}{2} \sum_{i\in \bp, j\not\in \bp} C_{ij} d_j } \prod_{j\in \cA_\bp^+} \{1\}_{d_j} \prod_{j\in \cA_\bp^-} \{\hbar^{-1} \}_{d_j} \prod_{i\in \bp} \{ x_i |_\bp \}_{\sum_{j\not\in \bp} C_{ij} d_j}
$$
where $\zeta (\bp)^d := \prod_{j\not\in \bp} \zeta_j (\bp)^{d_j}$.

Now we consider $V^{(\bone)} (q, z, a) |_\bp$ as a meromorphic function in terms of equivariant parameters $a_i$'s, in the region specified by the following condition:
$$
| a^\alpha | \ll 1,
$$
for any \emph{positive} root $\alpha$. In particular, we have $|\alpha_i (\bp)| \ll 1$ for $i\in \bp^+$, and $|\alpha_i (\bp)| \gg 1$ for $i\in \bp^-$.

We denote by the notation $\alpha (\bp) \xrightarrow{\sigma} 0$ the following process of taking limit:
$$
\alpha_i (\bp) \to \infty, \qquad i \in \bp^+; \qquad \alpha_i (\bp) \to 0 \qquad i \in \bp^-.
$$
Note that under the change of variables $\kappa_{\mathrm{vtx}}$, this limit is the as the one as described in Remark \ref{limit-point}, applied to $X'$ and $\theta'$.

\begin{Proposition} \label{limit-V}
In the region of $a_i$'s described above, we have
$$
\lim_{\alpha (\bp) \xrightarrow{\sigma} 0} V^{(1)} (q, z, a) = \prod_{j\in \cA_\bp^+} \frac{\phi \left( \hbar \zeta_{\sharp, j} (\bp) (q\hbar^{-1} )^{-\sum_{i\in \bp^+} C_{ij} }  \right) }{\phi \left( \zeta_{\sharp, j} (\bp) (q\hbar^{-1} )^{-\sum_{i\in \bp^+} C_{ij} }  \right)} \prod_{j\in \cA_\bp^-} \frac{\phi \left( q \zeta_{\sharp, j} (\bp)^{-1} (q\hbar^{-1} )^{\sum_{i\in \bp^+} C_{ij} }  \right) }{\phi \left( q\hbar^{-1}  \zeta_{\sharp, j} (\bp)^{-1} (q\hbar^{-1} )^{\sum_{i\in \bp^+} C_{ij} }  \right)} .
$$
In particular, under the change of variables $\kappa_{\mathrm{vtx}}$, it is
$$
\prod_{j\in (\bp')^-} \frac{\phi \left( q\hbar^{-1} (x'_j)^{-1} |_{\bp'} \right) }{\phi \left( (x'_j)^{-1} |_{\bp'}  \right)} \prod_{j\in (\bp')^+} \frac{\phi \left( q x'_j |_{\bp'} \right) }{\phi \left( \hbar x'_j |_{\bp'} \right)} = \Phi ((q - \hbar) \Pol'_{\bp'} (x') |_{\bp'} ).
$$
\end{Proposition}

\begin{proof}
It is easy to see that, for any $a$ and $d\in \ZZ$,
$$
\{a \}_d \to \left\{ \begin{aligned}
& (- q^{1/2} \hbar^{-1/2} )^d , && \qquad a \to 0 ; \\
& (- q^{1/2} \hbar^{-1/2} )^{- d} , && \qquad a \to \infty .
\end{aligned} \right.
$$
Therefore,
\ben
\lim_{\alpha (\bp) \xrightarrow{\sigma} 0} V^{(1)} (q, z, a)  &=& \sum_{\substack{d_j \geq 0, j\in \cA_\bp^+ \\ d_j \leq 0, j \in \cA_\bp^-} } \zeta (\bp)^d q^{ - \frac{1}{2} \sum_{j\not\in \bp} d_j - \frac{1}{2} \sum_{i\in \bp, j\not\in \bp} C_{ij} d_j } \prod_{j\in \cA_\bp^+} (- q^{1/2} \hbar^{-1/2} )^{d_j} \frac{(\hbar)_{d_j}}{(q)_{d_j} } \\
&& \prod_{j\in \cA_\bp^-} ( - q^{1/2} \hbar^{-1/2} )^{-d_j} \frac{(\hbar)_{-d_j}}{(q)_{-d_j}} \prod_{i\in \bp^+} (- q^{1/2} \hbar^{-1/2})^{ -  \sum_{j\not\in \bp} C_{ij} d_j } \prod_{i\in \bp^-} (- q^{1/2} \hbar^{-1/2})^{  \sum_{j\not\in \bp} C_{ij} d_j } \\
&=&  \sum_{\substack{d_j \geq 0, j\in \cA_\bp^+ \\ d_j \leq 0, j \in \cA_\bp^-} } \prod_{j\in \cA_\bp^+}  \frac{(\hbar)_{d_j}}{(q)_{d_j} } \zeta_{\sharp, j}^{d_j} \prod_{j\in \cA_\bp^-}  \frac{(\hbar)_{-d_j}}{(q)_{-d_j}} (q \hbar^{-1} \zeta_{\sharp, j}^{-1} )^{- d_j} \prod_{i\in \bp^+} (q \hbar^{-1} )^{ - \sum_{j\not\in \bp} C_{ij} d_j } \\
&=&  \sum_{\substack{d_j \geq 0, j\in \cA_\bp^+ \\ d_j \leq 0, j \in \cA_\bp^-} } \prod_{j\in \cA_\bp^+}  \frac{(\hbar)_{d_j}}{(q)_{d_j} } \left( \zeta_{\sharp, j} (q\hbar^{-1} )^{-\sum_{i\in \bp^+} C_{ij} }  \right)^{d_j} \prod_{j\in \cA_\bp^-}  \frac{(\hbar)_{-d_j}}{(q)_{-d_j}} \left( q \hbar^{-1}  \zeta_{\sharp, j}^{-1} (q\hbar^{-1} )^{\sum_{i\in \bp^+} C_{ij} }   \right)^{- d_j}
\een
The lemma then follows from the $q$-binomial formula $\dfrac{\phi (xz)}{\phi (z)} = \sum_{d\geq 0} \dfrac{(x)_d}{(q)_d} z^d$.
\end{proof}


\bibliographystyle{abbrv}
\bibliography{reference}

\vspace{12 mm}

\noindent
Andrey Smirnov\\
Department of Mathematics,\\
University of North Carolina at Chapel Hill,\\
Chapel Hill, NC 27599-3250, USA;\\
Steklov Mathematical Institute, \\
of Russian Academy of Sciences, \\
Gubkina str. 8, Moscow, 119991, Russia; \\
asmirnov@email.unc.edu

\vspace{3 mm}

\noindent
Zijun Zhou\\
Department of Mathematics,\\
Stanford University,\\
450 Serra Mall, Stanford, CA 94305, USA\\
zz2224@stanford.edu

\end{document}